\documentclass[11pt,twoside, leqno]{article}

\usepackage{amssymb}
\usepackage{amsmath}
\usepackage{mathrsfs}
\usepackage{amsthm}
\usepackage{txfonts}
\usepackage{color}

\allowdisplaybreaks

\def\ls{\lesssim}

\def\fz{\infty}

\def\red{\color{red}}

\renewcommand{\r}{\right}
\newcommand{\lf}{\left}

\def\ls{\lesssim}

\def\paz{{\partial}}

\def\supp{{\mathop\mathrm{\,supp\,}}}

\def\rr{{\mathbb R}}

\def\rn{{{\rr}^n}}

\newcommand{\wz}{\widetilde}

\newcommand{\cm}{{\mathcal M}}

\def\az{\alpha}
\def\lz{\lambda}

\def\dz{\delta}

\def\bz{\beta}

\def\fai{\varphi}
\def\gz{{\gamma}}

\def\tz{\theta}
\def\sz{\sigma}

\def\wz{\widetilde}

\def\ls{\lesssim}

\def\ol{\overline}

\def\boz{\Omega}

\def\uc{{\varepsilon}}

\def\esup{\mathop\mathrm{\,ess\,sup\,}}

\def\dfrac{\displaystyle\frac}

\newtheorem{theorem}{Theorem}[section]
\newtheorem{lemma}[theorem]{Lemma}
\newtheorem{corollary}[theorem]{Corollary}

\theoremstyle{definition}
\newtheorem{remark}[theorem]{Remark}
\newtheorem{definition}[theorem]{Definition}
\def\supp{{\mathop\mathrm{\,supp\,}}}
\def\diam{{\mathop\mathrm{diam}}}
\def\dist{{\mathop\mathrm{\,dist\,}}}
\def\loc{{\mathop\mathrm{loc}}}

\numberwithin{equation}{section}

\begin{document}

\title{\Large\bf Weighted Global Regularity Estimates for Elliptic Problems with
Robin Boundary Conditions in Lipschitz Domains
\footnotetext{\hspace{-0.35cm} 2010 {\it Mathematics Subject
Classification}. {Primary 35J25; Secondary 35J15, 42B35, 42B37.}
\endgraf{\it Key words and phrases}. elliptic equation, Robin boundary problem,
Lipschitz domain, (semi-)convex domain,  weak reverse H\"older inequality,
gradient estimate, Muckenhoupt weight, Morrey space, variable Lebesgue space.
\endgraf This work is supported by the National Natural Science Foundation
of China (Grant Nos. 11871254, 11971058, 11761131002, 11671185 and 11871100).}}
\author{Sibei Yang, Dachun Yang\,\footnote{Corresponding
author, E-mail: \texttt{dcyang@bnu.edu.cn}/{\red March 17, 2020}/Final version.}
\ and Wen Yuan}
\date{ }
\maketitle

\vspace{-0.8cm}

\begin{center}
\begin{minipage}{13.5cm}\small
{{\bf Abstract.}
Let $n\ge2$ and $\Omega$ be a bounded Lipschitz domain in $\mathbb{R}^n$.
In this article, the authors investigate global (weighted) estimates
for the gradient of solutions to Robin boundary value problems of second order
elliptic equations of divergence form with real-valued, bounded, measurable coefficients
in $\Omega$. More precisely, let $p\in(n/(n-1),\infty)$.
Using a real-variable argument, the authors obtain
two necessary and sufficient conditions for $W^{1,p}$ estimates of solutions
to Robin boundary value problems, respectively, in terms of a weak reverse H\"older
inequality with exponent $p$ or weighted $W^{1,q}$ estimates of solutions
with $q\in(n/(n-1),p]$ and some Muckenhoupt weights. As applications, the authors
establish some global regularity estimates for solutions to Robin boundary value problems of second order elliptic equations of divergence form with small
$\mathrm{BMO}$ coefficients, respectively, on bounded Lipschitz domains, $C^1$
domains or (semi-)convex domains, in the scale of
weighted Lebesgue spaces, via some quite subtle approach which is different from the
existing ones and, even when $n=3$ in case of bounded $C^1$ domains,
also gives an alternative correct proof of some know result. By this and some
technique from harmonic analysis, the authors further
obtain the global regularity estimates, respectively, in Morrey spaces, (Musielak--)Orlicz
spaces and variable Lebesgue spaces.}
\end{minipage}
\end{center}

\vspace{0.1cm}

\section{Introduction\label{s1}}

\hskip\parindent The study of regularity estimates in various function spaces
for linear or non-linear elliptic equations (or systems) in non-smooth domains
is one of the most interesting and important topics in partial differential equations
(see, for instance, \cite{amp18,aq02,b05,bw04,d96,dk10,fmm98,g12,ho19,jk95}
for the linear case and \cite{ap16,ap15,b18,byz08,cp98,cm16,mp12,mp11} for the non-linear case).
Furthermore, it is well known that the global regularity estimates for solutions to
elliptic boundary problems depend not only on the structure of equations
and the properties of the right-hand side datum and the coefficients appearing
in equations, but also on the smooth property or
the geometric property of the boundary of domains
(see, for instance, \cite{amp18,aq02,bw05,cm16,dk10,g12,jk95,k07,mp12}).

Motivated by \cite{d00,g12,ycyy20}, in this article, our aim is to study the
weighted global regularity estimates for Robin boundary value problems of
second-order elliptic equations of divergence form with real-valued, bounded, measurable
coefficients in bounded Lipschitz domains and their applications.
More precisely, let $n\ge2$, $\Omega$ be a bounded Lipschitz
domain in $\mathbb{R}^n$ and $p\in(n/(n-1),\fz)$.
Using a real-variable argument, we obtain
two necessary and sufficient conditions for $W^{1,p}$ estimates of
solutions to Robin boundary value problems, respectively,
in terms of a weak reverse H\"older
inequality with exponent $p$ or weighted $W^{1,q}$ estimates of solutions
with $q\in(n/(n-1),p]$ and some Muckenhoupt weights. As applications,
we establish some global regularity estimates for solutions to Robin boundary value problems
of second-order elliptic equations of divergence form with small
$\mathrm{BMO}$ coefficients on bounded Lipschitz domains, $C^1$
domains or (semi-)convex domains, in the scale of
weighted Lebesgue spaces. Applying those weighted global estimates and some
technique from harmonic analysis, such as properties of Muckenhoupt weights and
the extrapolation theorem, we further
obtain the global regularity estimates, respectively,
in Morrey spaces, (Musielak--)Orlicz spaces and variable Lebesgue spaces.
We point out that, in cases of bounded $C^1$ domains or (semi-)convex domains,
the approach used in this article to obtain the global weighted regularity
estimates is different from that used in \cite{amp18,acgg19,g18,ycyy20},
which is quite subtle and, even when $n=3$ in case of bounded $C^1$ domains,
also gives an alternative correct proof of \cite[Theorem 1.1]{acgg19}
(see Remark \ref{r1.6} below for the details).

To describe the main results of this article,
we first recall the notions of the Muckenhoupt weight class
and the reverse H\"older class (see, for instance, \cite{am07,cf13,g14,St93}).

\begin{definition}\label{d1.1}
Let $q\in[1,\fz)$. A non-negative and locally integrable function $\omega$ on $\rn$
is said to belong to the \emph{Muckenhoupt weight class} $A_q(\rn)$, denoted by $\omega\in A_q(\rn)$,
if, when $q\in(1,\fz)$,
\begin{equation*}
[\omega]_{A_q(\rn)}:=\sup_{B\subset\rn}\lf[\frac{1}{|B|}\int_B
\omega(x)\,dx\r]\lf\{\frac{1}{|B|}\int_B
[\omega(x)]^{-\frac{1}{q-1}}\,dx\r\}^{q-1}<\fz
\end{equation*}
or
\begin{equation*}
[\omega]_{A_1(\rn)}:=\sup_{B\subset\rn}\lf[\frac{1}{|B|}\int_B \omega(x)\,dx\r]
\lf\{\esup_{y\in B}[\omega(y)]^{-1}\r\}<\fz,
\end{equation*}
where the suprema are taken over all balls $B\subset\rn$.

Let $r\in(1,\fz]$. A non-negative and locally integrable function $\omega$ on $\rn$
is said to belong to the \emph{reverse H\"older class} $RH_r(\rn)$,
denoted by $\omega\in RH_r(\rn)$, if, when $r\in(1,\fz)$,
\begin{align*}
[\omega]_{RH_r(\rn)}:=\sup_{B\subset\rn}\lf\{\frac{1}
{|B|}\int_B [\omega(x)]^r\,dx\r\}^{\frac1r}\lf[\frac{1}{|B|}\int_B
\omega(x)\,dx\r]^{-1}<\fz
\end{align*}
or
\begin{equation*}
[\omega]_{RH_\fz(\rn)}:=\sup_{B\subset\rn}\lf[\esup_{y\in
B}\omega(y)\r]\lf[\frac{1}{|B|}\int_B\omega(x)\,dx\r]^{-1} <\fz,
\end{equation*}
where the suprema are taken over all balls $B\subset\rn$.
\end{definition}

Let $n\ge2$ and $\boz$ be a bounded Lipschitz domain in $\rn$.
Assume that $p\in[1,\fz)$ and $\omega\in A_q(\rn)$ with some $q\in[1,\fz)$.
Recall that the \emph{weighted Lebesgue space} $L^p_\omega(\Omega)$ is defined by setting
\begin{align}\label{1.1}
L^p_\omega(\Omega):=\lf\{f\ \text{is measurable on}\ \Omega: \
\|f\|_{L^p_\omega(\Omega)}:=\lf[\int_{\boz}
|f(x)|^p\omega(x)\,dx\r]^{\frac1p}<\fz\r\}.
\end{align}
Moreover, let
\begin{equation}\label{1.2}
L^p_\omega(\boz;\rn):=\lf\{\mathbf{f}:=(f_1,\,\ldots,\,f_n):\ \text{for any}
\ i\in\{1,\,\ldots,\,n\},\ f_i\in L^p_\omega(\boz)\r\}
\end{equation}
and
$$\|\mathbf{f}\|_{L^p_\omega(\boz;\rn)}:
=\sum_{i=1}^n\|f_i\|_{L^p_\omega(\boz)}.
$$
Denote by $W^{1,p}_\omega(\boz)$ the \emph{weighted Sobolev space on $\boz$} equipped
with the \emph{norm}
$$\|f\|_{W^{1,p}_\omega(\boz)}:=\|f\|_{L^p_\omega(\boz)}+\|\nabla f\|_{L^p_\omega
(\boz;\rn)},$$
where $\nabla f$ denotes the \emph{distributional
gradient} of $f$. Furthermore, $W^{1,p}_{0,\,\omega}(\boz)$ stands for
the \emph{closure} of $C^{\fz}_{\mathrm{c}} (\boz)$ in $W^{1,p}_\omega(\boz)$, where
$C^{\fz}_{\mathrm{c}}(\boz)$ denotes the set of all \emph{infinitely differentiable functions on
$\boz$ with compact supports contained in $\boz$}.
In particular, when $\omega\equiv1$, the weighted spaces $L^p_\omega(\boz)$
and $W^{1,p}_\omega(\boz)$ are denoted simply, respectively,
by $L^p(\boz)$ and $W^{1,p}(\boz)$, which are
just, respectively, the classical Lebesgue space
and the classical Sobolev space.

For any given $x\in\rn$, let $A(x):=\{a_{ij}(x)\}_{i,j=1}^n$ denote
an $n\times n$ matrix with real-valued, bounded and measurable entries.
Then $A$ is said to satisfy the \emph{uniform ellipticity condition}
if there exists a positive constant $\mu_0\in(0,1]$ such that,
for any $x:=(x_1,\,\ldots,\,x_n),\ \xi:=(\xi_1,\,\ldots,\,\xi_n)\in\rn$,
\begin{equation}\label{1.3}
\mu_0|\xi|^2\le\sum_{i,j=1}^na_{ij}(x)\xi_i\xi_j\le \mu_0^{-1}|\xi|^2.
\end{equation}
Throughout this article, we \emph{always assume} that the matrix
$A$ is real-valued, bounded and measurable, and satisfies
the uniform ellipticity condition \eqref{1.3}.

Denote by $\boldsymbol{\nu}:=(\nu_1,\,\ldots,\,\nu_n)$ the \emph{outward unit normal}
to $\partial\boz$, the \emph{boundary} of $\boz$.
Throughout this article, we always assume that
\begin{equation}\label{1.4}
0\le\az\in L^\fz(\partial\boz)\quad\text{and}\quad \az\not\equiv0\quad\text{on}\ \partial\boz.
\end{equation}
For any given $p\in(1,\fz)$, let
\begin{equation}\label{1.5}
p_\ast:=
\begin{cases}
\dfrac{np}{n+p} \ &\text{when}\ p\in(n/(n-1),\fz),\\
1+\epsilon\ & \text{when}\ p\in(1,n/(n-1)],
\end{cases}
\end{equation}
where $\epsilon\in(0,\fz)$ is an arbitrary given constant.
Assume that $p\in(1,\fz)$, $\mathbf{f}\in L^p(\boz;\rn)$, $F\in L^{p_\ast}(\boz)$
and $g\in W^{-1/p,p}(\partial\boz)$, where $W^{-1/p,p}(\partial\boz)$ denotes
the \emph{dual space} of the Sobolev space $W^{1/p,p'}(\partial\boz)$ on $\partial\boz$.
Here and thereafter, for any given $p\in[1,\fz]$, $p'\in[1,\fz]$ denotes the \emph{conjugate
number} of $p$, namely,  $1/p+1/p'=1$.
Then a function $u$ is called a \emph{weak solution} of the following
\emph{Robin boundary value problem}
\begin{equation}\label{1.6}
\begin{cases}
-\mathrm{div}(A\nabla u)=\mathrm{div}(\mathbf{f})+F\ \ & \text{in}\ \ \boz,\\
\dfrac{\partial u}{\partial\boldsymbol{\nu}}+\az u=\mathbf{f}\cdot\boldsymbol{\nu}+g\ \ & \text{on}\ \ \partial\boz,
\end{cases}
\hspace{4cm}\hfill (R)_{p}
\end{equation}
where $\frac{\partial u}{\partial\boldsymbol{\nu}}:=(A\nabla u)\cdot\boldsymbol{\nu}$
denotes the \emph{conormal derivative} of $u$ on $\partial\boz$,
if $u\in W^{1,p}(\boz)$ and, for any $\varphi\in C^\fz(\rn)$
(the set of all infinitely differentiable functions on $\rn$),
\begin{align}\label{1.7}
&\int_{\boz}A(x)\nabla u(x)\cdot\nabla\varphi(x)\,dx+\int_{\partial\boz}\az(x) u(x)
\varphi(x)\,d\sz(x)\\
&\quad=-\int_\boz\mathbf{f}(x)\cdot\nabla\varphi(x)\,dx+\int_{\boz}F(x)\varphi(x)\,dx
+\langle g,\varphi\rangle_{\partial\boz}.\nonumber
\end{align}
Here and thereafter, $\langle\cdot,\cdot\rangle_{\partial\boz}$ denotes the duality between
$W^{-1/p,p}(\partial\boz)$ and $W^{1/p,p'}(\partial\boz)$.
Moreover, the Robin problem $(R)_{p}$
is said to be \emph{uniquely solvable} if, for any $\mathbf{f}\in L^p(\boz;\rn)$,
$F\in L^{p_\ast}(\boz)$ and $g\in W^{-1/p,p}(\partial\boz)$,
there exists a \emph{unique} $u\in W^{1,p}(\boz)$ such that \eqref{1.7} holds true.
It is worth pointing out that the Robin boundary condition
naturally arises in the heat conduction problem
as well as in physical geodesy (see, for instance, \cite{k67,o98}).

Furthermore, let $p\in(1,\fz)$, $\omega\in A_q(\rn)$ with some $q\in[1,p]$,
$\mathbf{f}\in L^p_\omega(\boz;\rn)$ and $F\in L^{p_\ast}_{\omega^{p_\ast/p}}(\boz)$.
A function $u$ is called a \emph{weak solution} of the following
\emph{weighted Robin boundary value problem}
\begin{equation}\label{1.8}
\begin{cases}
-\mathrm{div}(A\nabla u)=\mathrm{div}(\mathbf{f})+F\ \ & \text{in}\ \ \boz,\\
\dfrac{\partial u}{\partial\boldsymbol{\nu}}+\az u=\mathbf{f}\cdot\boldsymbol{\nu}\ \ & \text{on}\ \ \partial\boz,
\end{cases}
\hspace{4cm}(R)_{p,\,\omega}
\end{equation}
if $u\in W^{1,p}_\omega(\boz)$ and, for any $\varphi\in C^\fz(\rn)$, \eqref{1.7} holds true with $g\equiv0$.
The weighted Robin boundary value problem $(R)_{p,\,\omega}$
is said to be \emph{uniquely solvable} if, for any $\mathbf{f}\in L^p_\omega(\boz;\rn)$
and $F\in L^{p_\ast}_{\omega^{p_\ast/p}}(\boz)$,
there exists a \emph{unique} $u\in W^{1,p}_\omega(\boz)$ such that \eqref{1.7}
holds true with $g\equiv0$. In particular, if $\az\equiv0$
in \eqref{1.6} and \eqref{1.8}, then the Robin problem $(R)_p$
and the weighted Robin problem $(R)_{p,\,\omega}$
are, respectively, the Neumann problem $(N)_p$
and the weighted Neumann problem $(N)_{p,\,\omega}$
(see, for instance, \cite{g12,g18,ycyy20}).
The Neumann problem $(N)_p$ is said to be \emph{uniquely solvable} if,
for any $\mathbf{f}\in L^p(\boz,\rn)$, $F\in L^{p_\ast}(\boz)$ and
$g\in W^{-1/p,p}(\partial\boz)$ satisfying the compatibility condition
$$
\int_\boz F(x)\,dx=\langle g,1\rangle_{\partial\boz},
$$
there exists a $u\in W^{1,p}(\boz)$, unique up to constants, such that \eqref{1.7}
holds true. Furthermore, the weighted Neumann boundary value problem $(N)_{p,\,\omega}$
is said to be \emph{uniquely solvable} if, for any $F\in L^{p_\ast}_{\omega^{p^\ast/p}}(\boz)$
satisfying $\int_\boz F(x)\,dx=0$ and $\mathbf{f}\in L^p_\omega(\boz;\rn)$,
there exists a $u\in W^{1,p}_\omega(\boz)$,
unique up to constants, such that \eqref{1.7} holds true.

Moreover, a function $u$ is called a \emph{weak solution} of the
following \emph{weighted Dirichlet boundary value problem}
\begin{equation}\label{1.9}
\begin{cases}
-\mathrm{div}(A\nabla u)=\mathrm{div}(\mathbf{f})+F\ \ & \text{in}\ \ \boz,\\
u=0 \ \ & \text{on}\ \ \partial\boz,
\end{cases}
\hspace{4cm}(D)_{p,\,\omega}
\end{equation}
if $u\in W^{1,p}_{0,\,\omega}(\boz)$ and \eqref{1.7} with $g\equiv0$ holds true for
any $\varphi\in C^\fz_{\mathrm{c}}(\boz)$. In particular, when $\omega\equiv1$,
the weighted Dirichlet problem $(D)_{p,\,\omega}$ is just the Dirichlet problem $(D)_{p}$.
The weighted Dirichlet problem $(D)_{p,\,\omega}$ is said to
be \emph{uniquely solvable} if, for any $F\in L^{p_\ast}_{\omega^{p_\ast/p}}(\boz)$ and
$\mathbf{f}\in L^p_\omega(\boz;\rn)$, there exists a \emph{unique} $u\in W^{1,p}_{0,\,\omega}(\boz)$
such that \eqref{1.7} with $g\equiv0$ holds true for
any $\varphi\in C^\fz_{\mathrm{c}}(\boz)$.

\begin{remark}\label{r1.1x}
Let $n\ge2$ and $\boz\subset\rn$ be a bounded Lipschitz domain.
By the Lax--Milgram theorem (see, for instance, \cite[Theorem 5.8]{gt01})
and the Sobolev embedding theorem (see, for instance, \cite[Thoerem 7.26]{gt01} and
\cite[Section 2.3.4, Theorem 3.4]{n12}), we know that the Dirichlet problem $(D)_2$
and the Neumann problem $(N)_2$ are uniquely solvable. Furthermore,
by the Lax--Milgram theorem, the Friedrichs inequality
(see, for instance, \cite[Section 1.1.8, Theorem 1.9]{n12}
and \cite[Theorem 6.1]{ls04}) and the Sobolev embedding theorem, we find that
the Robin problem $(R)_2$ is uniquely solvable. We point out that
the conditions that $\az\ge0$ and $\az\not\equiv0$ on $\partial\boz$
is necessary for the unique solvability of the Robin problem $(R)_2$
(see, for instance, \cite[p.\,97]{ls04} for a counterexample).

Moreover, via an example given by Meyers \cite[Section 5]{m63} (see also \cite[p.\,1285]{bw04}),
we conclude that, for a general $p\in(1,\fz)$ with $p\neq2$,
the Dirichlet problem $(D)_p$, the Neumann problem $(N)_p$
and the Robin problem $(R)_p$ may not be uniquely solvable,
even when the domain $\boz$ and the coefficient matrix $A$ are smooth.
\end{remark}

Let $n\ge2$ and $\boz\subset\rn$ be a bounded Lipschitz domain.
Assume that the matrix $A$ satisfies the $(\dz,R)$-BMO condition
(see Definition \ref{d1.2} below) or $A$ belongs to the space $\mathrm{VMO}(\rn)$
(see, for instance, \cite{s75}). The main aim of this article is to obtain
the following Calder\'on--Zygmund type estimates
\begin{equation}\label{1.10}
\|u\|_{W^{1,p}(\boz)}\le C\lf[\|\mathbf{f}\|_{L^p(\boz;\rn)}
+\|F\|_{L^{p_\ast}(\boz)}+\|g\|_{W^{-1/p,p}(\partial\boz)}\r]
\end{equation}
for the Robin problem \eqref{1.6},
the following weighted Calder\'on--Zygmund type estimates
\begin{equation}\label{1.11}
\|u\|_{W^{1,p}_\omega(\boz)}\le C\lf[\|\mathbf{f}\|_{L^p_\omega(\boz;\rn)}
+\|F\|_{L^{p_\ast}_{\omega^{p_\ast/p}}(\boz)}\r]
\end{equation}
for the weighted Robin problem \eqref{1.8}, and then give some applications,
where $C$ is a positive constant independent of $u$, $\mathbf{f}$, $F$ and $g$.

Now we recall some known results for (weighted) Calder\'on--Zygmund type
estimates, respectively, to the (weighted) Dirichlet problem, the (weighted) Neumann problem
and the (weighted) Robin problem.

For the Dirichlet problem $(D)_p$, the estimate \eqref{1.10}
with $F\equiv0,\,g\equiv0$ and $p\in(1,\fz)$ was obtained in \cite{d96},
under the assumptions that $A\in\mathrm{VMO}(\rn)$
and $\partial\boz\in C^{1,1}$, which was then weakened to $\partial\boz\in C^{1}$
in \cite{aq02}. Furthermore, for the Dirichlet problem $(D)_p$,
the estimate \eqref{1.10} with $F\equiv0,\,g\equiv0$ and $p\in(1,\fz)$
was established in \cite{b05,bw04}, under the assumptions that $A$ satisfies the $(\dz,R)$-BMO
condition for sufficiently small $\dz\in(0,\fz)$, and $\boz$ is a bounded Lipschitz domain
with small Lipschitz constant or a bounded Reifenberg flat domain (see, for instance, \cite{r60,t97}).
Moreover, for the Dirichlet problem $(D)_p$ with partial small $\mathrm{BMO}$ coefficients,
the estimate \eqref{1.10} with $g\equiv0$ and $p\in(1,\fz)$
was systematically studied in \cite{dk10,k07}, under the assumption that
$\boz$ is a bounded Lipschitz domain with small Lipschitz constant.
For the Dirichlet problem $(D)_p$ in a general Lipschitz domain $\boz$,
it was proved in \cite{sh05a} that, if $A$ is symmetric and $A\in\mathrm{VMO}(\rn)$, then
\eqref{1.10} with $F\equiv0$ and $g\equiv0$ holds true for any $p\in(\frac32-\uc,3+\uc)$
when $n\ge3$, or $p\in(\frac43-\uc,4+\uc)$ when $n=2$, where $\uc\in(0,\fz)$ is a positive constant
depending only on the Lipschitz constant of $\boz$ and $n$. It is worth pointing out that, when $A:=I$ (the identity matrix)
in \eqref{1.9}, the range of $p$ obtained in \cite{sh05a} is even sharp for
general Lipschitz domains (see, for instance, \cite{jk95}).
Moreover, for the weighted Dirichlet problem $(D)_{p,\,\omega}$ with partial small
$\mathrm{BMO}$ coefficients, \eqref{1.7} with $F\equiv0$, $p\in(2,\fz)$
and $\omega\in A_{p/2}(\rn)$ was obtained in \cite{bp14} under the assumption that
$\boz$ is a bounded Reifenberg flat domain. For the problem $(D)_{p,\,\omega}$ with symmetric and
small $\mathrm{BMO}$ coefficients, \eqref{1.11} with $F\equiv0$, $p\in(1,\fz)$ and
$\omega\in A_{p}(\rn)$ was established in \cite{amp18} under the assumption that
$\boz$ is a bounded Lipschitz domain with small Lipschitz constant.

For the Neumann problem $(N)_p$, the estimate \eqref{1.10}
with $F\equiv0,\,g\equiv0$ and $p\in(1,\fz)$ was obtained in \cite{aq02},
under the assumptions that $A\in\mathrm{VMO}(\rn)$ and $\partial\boz\in C^1$.
Moreover, when $A$ has the small $\mathrm{BMO}$
coefficients and $\boz$ is a bounded Reifenberg flat domain, or $A$ has
partial small $\mathrm{BMO}$ coefficients and $\boz$ is a bounded Lipschitz domain
with small Lipschitz constant, the estimate \eqref{1.10} with $g\equiv0$
and $p\in(1,\fz)$ was established, respectively, in \cite{bw05} and \cite{dk10}
for the Neumann problem $(N)_p$.
Furthermore, for the Neumann problem $(N)_p$ on a general Lipschitz domain,
it was proved in \cite{g12,g18} that, if $A$ is symmetric and $A\in\mathrm{VMO}(\rn)$, then
\eqref{1.10} holds true for any $p\in(\frac32-\uc,3+\uc)$
when $n\ge3$, or $p\in(\frac43-\uc,4+\uc)$ when $n=2$,
where $\uc\in(0,\fz)$ is a positive constant depending only
on the Lipschitz constant of $\boz$ and $n$.
We point out that, when $A:=I$ in the Neumann problem $(N)_p$,
the range of $p$ obtained in \cite{g12,g18} is even sharp for
general Lipschitz domains (see, for instance, \cite{fmm98}).
In particular, if $A$ is symmetric, $A\in\mathrm{VMO}(\rn)$ and $\boz$ is convex,
it was proved in \cite{g18} that \eqref{1.10} with $F\equiv0$
holds true for the Neumann problem $(N)_p$ with any given $p\in(1,\fz)$.
Moreover, for any given $p\in(1,\fz)$ and $\omega\in A_p(\rn)$,
the weighted estimate \eqref{1.11}, with $\|F\|_{L^{p_\ast}_{\omega^{p_\ast/p}}(\boz)}$
replaced by $\|F\|_{L^{p}_{\omega}(\boz)}$, was established in \cite{ycyy20}
for the weighted Neumann problem $(N)_{p,\,\omega}$ when $\boz$ is a bounded (semi-)convex domain.
Furthermore, for the weighted Neumann problem $(N)_{p,\,\omega}$ on a
general Lipschitz domain, if $p\in(\frac32-\uc,3+\uc)$
when $n\ge3$, or $p\in(\frac43-\uc,4+\uc)$ when $n=2$, and $\omega\in
A_{p/(\frac32-\uc)}(\rn)\cap RH_{(\frac{3+\uc}{p})'}(\rn)$ when $n\ge3$,
or $\omega\in A_{p/(\frac43-\uc)}(\rn)\cap RH_{(\frac{4+\uc}{p})'}(\rn)$ when $n=2$,
the estimate \eqref{1.11} with $\|F\|_{L^{p_\ast}_{\omega^{p_\ast/p}}(\boz)}$
replaced by $\|F\|_{L^{p}_{\omega}(\boz)}$ is a simple corollary of
\cite[Theorem 1.2]{g12} and \cite[Theorem 1.2]{ycyy20},
where $\uc\in(0,\fz)$ is a positive constant depending
only on the Lipschitz constant of $\boz$ and $n$.

For $C^1$ domains, Lipschitz domains, Reifenberg flat domains and (semi-)convex domains mentioned as above,
we have the following relations. It is known that
$C^1$ domains are Lipschitz domains with small Lipschitz
constants, Lipschitz domains with small Lipschitz constants
are Reifenberg flat domains and hence $C^1$ domains are Reifenberg flat domains,
but general Lipschitz domains may not be Reifenberg flat domains.
Moreover, (semi-)convex domains are Lipschitz domains, but may not be $C^1$ domains,
Lipschitz domains with small Lipschitz constants or Reifenberg flat domains.
Furthermore, convex domains are semi-convex domains (see \cite[Remark 1.10]{ycyy20}
or Remarks \ref{r1.4}(iv) below).

For the Robin problem \eqref{1.6}, the general well-posedness of solutions
was studied in \cite{d00}. Moreover, when $\boz\subset\rn$ ($n\ge3$) is a bounded
Lipschitz domain, the regularity theory for the Robin problem
\eqref{1.6} with $\mathbf{f}\equiv0$ and $A:=I$
was investigated in \cite{lm06} via the layer potential method.
Recently, for the Robin problem $(R)_p$ on a general Lipschitz domain $\boz\subset\rn$
($n\ge3$), when $A$ is symmetric and $A\in\mathrm{VMO}(\rn)$, the global regularity estimate
\eqref{1.10} was established in \cite{gz20} for any given
$p\in(\frac32-\uc,3+\uc)$, where $\uc\in(0,\fz)$ is a positive constant
depending only on the Lipschitz constant of $\boz$ and $n$. Moreover, for the Robin problem $(R)_p$ on a bounded $C^1$
domain $\boz\subset\rr^3$, when $A$ is symmetric and $A\in\mathrm{VMO}(\rr^3)$,
the estimate \eqref{1.10} was obtained in \cite{acgg19}
for any given $p\in(1,\fz)$. By the way, we point out that H\"older
regularity estimates for the Robin problem \eqref{1.6} in bounded Lipschitz domains
were studied in \cite{n11}.

Now we give the main results of this article as follows. In what follows,
for any $x\in\rn$ and $r\in(0,\fz)$, we always let
$B(x,r):=\{y\in\rn:\ |y-x|<r\}$.

\begin{theorem}\label{t1.1}
Let $n\ge2$, $\boz\subset\rn$ be a bounded Lipschitz domain, $p_0\in(n/(n-1),\fz)$, $p\in(p_0,\fz)$
and $\az$ be as in \eqref{1.4}.
Assume that the matrix $A$ is real-valued, symmetric, bounded and measurable,
and satisfies \eqref{1.3}, and the weak solution $u$ of the Robin problem $(R)_{p_0}$
with $\mathbf{f}\in L^{p_0}(\boz;\rn)$, $F\in L^{(p_0)_\ast}(\boz)$
and $g\equiv0$ exists and satisfies
\begin{equation}\label{1.12}
\|u\|_{W^{1,p_0}(\boz)}\le C\lf[\|\mathbf{f}\|_{L^{p_0}(\boz;\rn)}
+\|F\|_{L^{(p_0)_\ast}(\boz)}\r],
\end{equation}
where $(p_0)_\ast$ is as in \eqref{1.5} with $p$ replaced by $p_0$
and $C$ is a positive constant depending only on $n$, $p_0$, $\mu_0$,
$\|\az\|_{L^\fz(\partial\boz)}$ and the Lipschitz constant of $\boz$.
Then the following three statements are mutually equivalent.
\begin{itemize}
\item[\rm(i)] A weak solution $u\in W^{1,p_0}(\boz)$ of the Robin problem
$(R)_p$ with $\mathbf{f}\in L^p(\boz;\rn)$, $F\in L^{p_\ast}(\boz)$ and $g\equiv0$
exists, where $p_\ast$ is as in \eqref{1.5}, and, moreover, $u\in W^{1,p}(\boz)$ and
there exists a positive constant $C$, depending only on $n$, $p$, $\mu_0$,
$\|\az\|_{L^\fz(\partial\boz)}$ and the Lipschitz constant of $\boz$, such that
\begin{equation}\label{1.13}
\|u\|_{W^{1,p}(\boz)}\le C\lf[\|\mathbf{f}\|_{L^p(\boz;\rn)}
+\|F\|_{L^{p_\ast}(\boz)}\r].
\end{equation}

\item[\rm(ii)] There exist positive constants $C_0\in(0,\fz)$ and $r_0\in(0,\diam(\boz))$
such that, for any ball $B(x_0,r)$ having the property that $r\in(0,r_0/4)$
and either $x_0\in\partial\boz$ or $B(x_0,2r)\subset\boz$,
the weak reverse H\"older inequality
\begin{align}\label{1.14}
&\lf\{\frac{1}{r^n}\int_{B(x_0,r)\cap\boz}[|v(x)|+|\nabla v(x)|]^p\,dx\r\}^{\frac1p}\\
&\quad\le C_0\lf\{\frac{1}{r^n}\int_{B(x_0,2r)\cap\boz}
[|v(x)|+|\nabla v(x)|]^{p_0}\,dx\r\}^{\frac1{p_0}}\nonumber
\end{align}
holds true for any function $v\in W^{1,p_0}(B(x_0,2r)\cap\boz)$ satisfying
$\mathrm{div}(A\nabla v)=0$ in $B(x_0,2r)\cap\boz$ and
$\frac{\partial v}{\partial\boldsymbol{\nu}}+\az v=0$
on $B(x_0,2r)\cap\partial\boz$ when $x_0\in\partial\boz$,
where $\diam(\boz):=\sup\{|y-z|:\ y,\,z\in\boz\}$.

\item[\rm(iii)] Let $q\in[p_0,p]$, $q_0\in[1,\frac{q}{p_0}]$, $r_0\in[(\frac{p}{q})',\fz]$
and $\omega\in A_{q_0}(\rn)\cap RH_{r_0}(\rn)$.
A weak solution $u$ of the weighted Robin problem
$(R)_{q,\,\omega}$ with $\mathbf{f}\in L^q_\omega(\boz;\rn)$ and
$F\in L^{q_\ast}_{\omega^{q_\ast/q}}(\boz)$,
where $q_\ast$ is as in \eqref{1.5} with $p$ replaced by $q$,
exists and, moreover, $u\in W^{1,q}_\omega(\boz)$ and there exists a positive constant $C$,
depending only on $n$, $\mu_0$, $p$, $q$, $[\omega]_{A_{q_0}(\rn)}$,
$[\omega]_{RH_{r_0}(\rn)}$, $\|\az\|_{L^\fz(\partial\boz)}$ and the Lipschitz constant of $\boz$, such that
\begin{equation}\label{1.15}
\|u\|_{W^{1,q}_\omega(\boz)}\le C\lf[\|\mathbf{f}\|_{L^q_\omega(\boz;\rn)}+
\|F\|_{L^{q_\ast}_{\omega^{q_\ast/q}}(\boz)}\r].
\end{equation}
\end{itemize}
\end{theorem}

\begin{remark}\label{r1.1}
\begin{itemize}
\item[{\rm(a)}] The assumption $p_0>n/(n-1)$ in Theorem \ref{t1.1} is to guarantee
$1/(p_0)_\ast=1/p_0+1/n$, which is necessary in the proof of Theorem \ref{t1.1}.
Moreover, by the proof of Theorem \ref{t1.1}, we know that, if $F\equiv0$
in (i) and (iii) of Theorem \ref{t1.1}, the condition $p_0\in(n/(n-1),\fz)$ can be weaken
to $p_0\in(1,\fz)$ for that Theorem \ref{t1.1}(ii) implies (i) and (iii)
of Theorem \ref{t1.1}.
\item[{\rm(b)}] If replacing the assumption $p_0\in(n/(n-1),\fz)$ by $p_0\in(1,\fz)$
in Theorem \ref{t1.1}, then Theorem \ref{t1.1}(ii) implies the following conclusion.
\begin{itemize}
\item[(iv)] Let $q\in[p_0,p]$, $q_0\in[1,\frac{q}{p_0}]$, $r_0\in[(\frac{p}{q})',\fz]$
and $\omega\in A_{q_0}(\rn)\cap RH_{r_0}(\rn)$.
A weak solution $u$ of the weighted Robin problem
$(R)_{q,\,\omega}$ with $\mathbf{f}\in L^q_\omega(\boz;\rn)$ and
$F\in L^{q_0}_{\omega^{q_0/q}}(\boz)$, where $q_0\in(1,\fz)$ is given by
$1/q_0=1/q+1/(p_0)_\ast-1/p_0$, exists and, moreover,
$u\in W^{1,q}_\omega(\boz)$ and there exists a positive constant $C$,
depending only on $n$, $\mu_0$, $p$, $q$, $[\omega]_{A_{q_0}(\rn)}$,
$[\omega]_{RH_{r_0}(\rn)}$, $\|\az\|_{L^\fz(\partial\boz)}$ and the Lipschitz constant of $\boz$, such that
\begin{equation*}
\|u\|_{W^{1,q}_\omega(\boz)}\le C\lf[\|\mathbf{f}\|_{L^q_\omega(\boz;\rn)}+
\|F\|_{L^{q_0}_{\omega^{q_0/q}}(\boz)}\r].
\end{equation*}
\end{itemize}
\end{itemize}
\end{remark}

Theorem \ref{t1.1} is proved by using a weighted real-variable argument obtained
in Theorem \ref{t3.1} below, the linear structure of the Robin
boundary value problem, the Sobolev embedding theorem and properties of
the Muckenhoupt weight class and the reverse H\"older class.
The weighted real-variable argument, established
in Theorem \ref{t3.1} below, is inspired by \cite{cp98,sh07,w03}
and indeed a natural generalization of the off-diagonal case in \cite[Theorem 3.4]{sh07} and
\cite[Theorems 2.1 and 2.2]{g12} (see also \cite{g18,sh18,sh05a}).
We also point out that a similar (weighted) real-variable
argument with the different motivation was used in \cite{a07,am07}.

Denote by $L^1_{\loc}(\rn)$ the \emph{set of all locally integrable functions on $\rn$}.
To establish the (weighted) global regularity estimates for solutions
to Robin boundary problems on bounded Lipschitz domains via using Theorem
\ref{t1.1}, we need to recall notions of the $(\dz,R)$-$\mathrm{BMO}$ condition and
the space $\mathrm{VMO}(\rn)$ as follows (see, for instance, \cite{bw05,bw04,s75}).

\begin{definition}\label{d1.2}
Let $R,\,\dz\in(0,\fz)$.
\begin{itemize}
\item[{\rm(i)}] A function $f\in L^1_{\loc}(\rn)$ is said to satisfy
the \emph{$(\dz,R)$-$\mathrm{BMO}$ condition} if
\begin{equation}\label{1.16}
\|f\|_{\ast,\,R}:=\sup_{r\in(0,R)}\sup_{x\in\rn}\frac{1}{|B(x,r)|}
\int_{B(x,r)}|f(y)-f_{B(x,r)}|\,dy\le\dz,
\end{equation}
where the suprema are taken, respectively, over all $r\in(0,R)$ and $x\in\rn$, and $$f_{B(x,r)}:=\frac{1}{|B(x,r)|}\int_{B(x,r)}f(y)\,dy.$$
Furthermore, $f$ is said to belong to the space $\mathrm{VMO}(\rn)$ if $f$ satisfies the
$(\dz,R)$-$\mathrm{BMO}$ condition for some $\dz,\,R\in(0,\fz)$ and
$$\lim_{r\to0^{+}}\sup_{x\in\rn}\frac{1}{|B(x,r)|}\int_{B(x,r)}|f(y)-f_{B(x,r)}|\,dy=0,
$$
where $r\to0^{+}$ means $r\in(0,\fz)$ and $r\to0$.

\item[{\rm(ii)}] A matrix $A:=\{a_{ij}\}_{i,j=1}^n$ is said to
satisfy the \emph{$(\dz,R)$-$\mathrm{BMO}$ condition} [resp., $A\in\mathrm{VMO}(\rn)$]
if, for any $i,\,j\in\{1,\,\ldots,\,n\}$, $a_{ij}$
satisfies the $(\dz,R)$-$\mathrm{BMO}$ condition [resp., $a_{ij}\in\mathrm{VMO}(\rn)$].
\end{itemize}
\end{definition}

\begin{remark}\label{r1.2}
A function $f\in L^1_{\loc}(\rn)$ is said to belong to
the \emph{space $\mathrm{BMO}(\rn)$} (the \emph{space of bounded mean oscillation}),
denoted by $f\in\mathrm{BMO}(\rn)$,
if $\|f\|_{\mathrm{BMO}(\rn)}:=\|f\|_{\ast,\,\fz}<\fz$,
where $\|f\|_{\ast,\,\fz}$ is as in \eqref{1.16}.
Similarly, a matrix $A:=\{a_{ij}\}_{i,j=1}^n$ is said to belong to the \emph{space}
$\mathrm{BMO}(\rn)$, denoted by $A\in\mathrm{BMO}(\rn)$, if,
for any $i,\,j\in\{1,\,\ldots,\,n\}$, $a_{ij}\in\mathrm{BMO}(\rn)$.

Let $\dz\in(0,\fz)$. If $f\in\mathrm{BMO}(\rn)$ and $\|f\|_{\mathrm{BMO}(\rn)}\le\dz$,
then $f$ satisfies the $(\dz,R)$-$\mathrm{BMO}$ condition for any $R\in(0,\fz)$.
Moreover, if $f\in\mathrm{VMO}(\rn)$, then $f$ satisfies the $(\gz,R)$-$\mathrm{BMO}$
condition for any $\gz\in(0,\fz)$ and some $R\in(0,\fz)$.
\end{remark}

\begin{theorem}\label{t1.2}
Let $n\ge3$, $\boz\subset\rn$ be a bounded Lipschitz domain and $\az\in L^\fz(\partial\boz)$
as in \eqref{1.4}.
Assume that the matrix $A$ is real-valued,
symmetric, bounded and measurable, and satisfies \eqref{1.3}.
\begin{itemize}
\item[\rm(i)] Then there exist positive constants $\uc_0,\,\dz_0\in(0,\fz)$,
depending only on $n$ and the Lipschitz constant of $\boz$, such that, for any given $p\in((3+\uc_0)', 3+\uc_0)$,
if $A$ satisfies the $(\dz,R)$-$\mathrm{BMO}$
condition for some $\dz\in(0,\dz_0)$ and $R\in(0,\fz)$ or $A\in\mathrm{VMO}(\rn)$, then
the Robin problem $(R)_{p}$ with $\mathbf{f}\in L^p(\boz;\rn)$,
$F\in L^{p_\ast}(\boz)$ and $g\in W^{-1/p,p}(\partial\boz)$ is uniquely solvable and
there exists a positive constant $C$, depending only on $n$, $p$,
$\|\az\|_{L^\fz(\partial\boz)}$ and the Lipschitz constant of $\boz$, such that,
for any weak solution $u$, $u\in W^{1,p}(\boz)$ and
\begin{equation}\label{1.17}
\|u\|_{W^{1,p}(\boz)}\le C\lf[\|\mathbf{f}\|_{L^p(\boz;\rn)}
+\|F\|_{L^{p_\ast}(\boz)}+\|g\|_{W^{-1/p,p}(\partial\boz)}\r],
\end{equation}
where $p_\ast$ is as in \eqref{1.5}.

\item[\rm(ii)] Let $\uc_0$ be as in (i) and $\wz{p}_0:=\max\{\frac n{n-1},(3+\uc_0)'\}$.
For any given $p\in(\wz{p}_0,3+\uc_0)$ and any
$\omega\in A_{\frac{p}{\wz{p}_0}}(\rn)\cap RH_{(\frac{3+\uc_0}{p})'}(\rn)$,
there exists a positive constant $\dz_0\in(0,\fz)$,
depending only on $n$, $p$, the Lipschitz constant of $\boz$, $[\omega]_{A_{\frac{p}{\wz{p}_0}}(\rn)}$
and $[\omega]_{RH_{(\frac{3+\uc_0}{p})'}(\rn)}$,  such that, if
$A$ satisfies the $(\dz,R)$-$\mathrm{BMO}$ condition for some $\dz\in(0,\dz_0)$
and $R\in(0,\fz)$ or $A\in\mathrm{VMO}(\rn)$,
then the weighted Robin problem $(R)_{p,\,\omega}$ with
$\mathbf{f}\in L^p_\omega(\boz;\rn)$ and $F\in L^{p_\ast}_{\omega^{p_\ast/p}}(\boz)$
is uniquely solvable and there exists a positive constant
$C$, depending only on $n$, $p$, $\|\az\|_{L^\fz(\partial\boz)}$,
$[\omega]_{A_{\frac{p}{\wz{p}_0}}(\rn)}$, $[\omega]_{RH_{(\frac{3+\uc_0}{p})'}(\rn)}$
and the Lipschitz constant of $\boz$, such that,
for any weak solution $u$, $u\in W^{1,p}_\omega(\boz)$ and
\begin{equation}\label{1.18}
\|u\|_{W^{1,p}_\omega(\boz)}\le C\lf[\|\mathbf{f}\|_{L^p_\omega(\boz;\rn)}
+\|F\|_{L^{p_\ast}_{\omega^{p_\ast/p}}(\boz)}\r].
\end{equation}
\end{itemize}
\end{theorem}

We show Theorem \ref{t1.2} via using Theorem \ref{t1.1}
and a weak reverse H\"older inequality
for the local Robin problem, which is as in Theorem \ref{t1.1}(i)
and was obtained in \cite{gz20}.

\begin{remark}\label{r1.3}
Let $\boz$, $\az$ and $A$ be as in Theorem \ref{t1.2}. If $F\equiv0$ in
Theorem \ref{t1.2}(ii), via using Theorem \ref{t1.1}, Remark \ref{r1.1}(a)
and Theorem \ref{t1.2}(i), we find that \eqref{1.18} holds true without
the restriction $p>n/(n-1)$.
More precisely, assume that $p\in((3+\uc_0)',3+\uc_0)$ and $\omega\in A_{\frac{p}{(3+\uc_0)'}}(\rn)
\cap RH_{(\frac{3+\uc_0}{p})'}(\rn)$, where $\uc_0$ is as in Theorem \ref{t1.2}(i).
Then there exists a positive constant $\dz_0\in(0,\fz)$,
depending only on $n$, $p$, the Lipschitz constant of $\boz$, $[\omega]_{A_{\frac{p}{(3+\uc_0)'}}(\rn)}$
and $[\omega]_{RH_{(\frac{3+\uc_0}{p})'}(\rn)}$,  such that, if
$A$ satisfies the $(\dz,R)$-$\mathrm{BMO}$ condition for some $\dz\in(0,\dz_0)$
and $R\in(0,\fz)$ or $A\in\mathrm{VMO}(\rn)$,
then the weighted Robin problem $(R)_{p,\,\omega}$ with
$\mathbf{f}\in L^p_\omega(\boz;\rn)$ and $F\equiv0$
is uniquely solvable and there exists a positive constant
$C$, depending only on $n$, $p$, $\|\az\|_{L^\fz(\partial\boz)}$, $[\omega]_{A_{\frac{p}{(3+\uc_0)'}}(\rn)}$,
$[\omega]_{RH_{(\frac{3+\uc_0}{p})'}(\rn)}$ and the Lipschitz constant of $\boz$, such that,
for any weak solution $u$, $u\in W^{1,p}_\omega(\boz)$ and
\begin{equation}\label{1.19}
\|u\|_{W^{1,p}_\omega(\boz)}\le C\|\mathbf{f}\|_{L^p_\omega(\boz;\rn)}.
\end{equation}
\end{remark}

To describe the (weighted) global regularity estimates for the
Robin problems on (semi-)convex domains,
we recall the notion of semi-convex domains as follows.

\begin{definition}\label{d1.4}
\begin{enumerate}
\item[(i)] Let $O$ be an open set in $\rn$. The collection of
\emph{semi-convex functions} on $O$ is defined to be the set of all continuous functions
$u:\ O\rightarrow\rr$ having the property that there exists a positive
constant $C$ such that, for any $x,\,h\in\rn$ with the ball
$B(x,|h|)\subset O$,
$$2u(x)-u(x+h)-u(x-h)\le C|h|^2.
$$
The minimal positive constant $C$ as above is referred as the
\emph{semi-convexity constant} of $u$.

\item[(ii)] A non-empty, proper open subset $\boz$ of $\rn$ is said to be
\emph{semi-convex} provided that there exist constants $c,\,d\in(0,\fz)$ such that,
for every $x_0\in\paz\boz$, there exist an
$(n-1)$-dimensional affine variety $H\subset\rn$ passing through
$x_0$, a choice $N$ of the unit normal to $H$, and an open set
$$\mathcal{C}:=\{\wz{x}+tN:\
\wz{x}\in H,\ |\wz{x}-x_0|<c,\ |t|<d\}$$
(which is called a \emph{coordinate cylinder} near $x_0$ with axis along $N$) satisfying,
for some semi-convex function $\fai:\ H\rightarrow\rr$,
$$\mathcal{C}\cap\boz=\mathcal{C}\cap\{\wz{x}+tN:\ \wz{x}\in H,\
t>\fai(\wz{x})\},$$
$$\mathcal{C}\cap\paz\boz=\mathcal{C}\cap\{\wz{x}+tN:\ \wz{x}\in H,\
t=\fai(\wz{x})\},$$
$$\mathcal{C}\cap\ol{\boz}
^{\complement}=\mathcal{C}\cap\{\wz{x}+tN:\ \wz{x}\in H,\
t<\fai(\wz{x})\},$$
$$\fai(x_0)=0\ \ \text{and}\ \ |\fai(\wz{x})|<d/2\ \ \text{if}\ \
|\wz{x}-x_0|\le c,$$
where $\ol{\boz}$ and $\ol{\boz}^{\complement}$ respectively denote the closure of $\boz$ in $\rn$
and the complementary set of $\overline{\boz}$ in $\rn$.
\end{enumerate}
\end{definition}

\begin{remark}\label{r1.4}
\begin{enumerate}
\item[(i)] A set $E\subset\rn$ is said to satisfy an \emph{exterior ball condition} at $x\in\partial E$
if there exist $v\in S^{n-1}$ and $r\in(0,\fz)$ such that
\begin{equation}\label{1.20}
B(x+rv,r)\subset(\rn\setminus E),
\end{equation}
where $S^{n-1}$ denotes the  \emph{unit sphere} of $\rn$.
For such an $x\in \partial E$, let
$$r(x):=\sup\lf\{r\in(0,\fz):\ \eqref{1.20} \ \text{holds true for some}\ v\in S^{n-1}\r\}.
$$
A set $E$ is said to satisfy a \emph{uniform exterior ball condition} (for short,
UEBC) with radius $r\in(0,\fz]$ if
\begin{equation}\label{1.21}
\inf_{x\in\partial E}r(x)\ge r,
\end{equation}
and the value $r$ in \eqref{1.21} is referred to the \emph{UEBC constant}.
A set $E$ is said to satisfy a UEBC if there exists $r\in(0,\fz]$ such that $E$ satisfies
the uniform exterior ball condition with radius $r$. Moreover, the largest positive constant $r$
as above is called the \emph{uniform ball constant} of $E$.

It is well known that, for any open set $\boz\subset\rn$ with compact boundary,
$\boz$ is a \emph{Lipschitz domain satisfying a UEBC} if and only if $\boz$
is a \emph{semi-convex domain} in $\rn$ (see, for instance, \cite[Theorem 2.5]{mmmy10} or
\cite[Theorem 3.9]{mmy10}).
\vspace{-0.5em}
\item[(ii)] It is worth pointing out that, if $\boz\subset\rn$ is convex, then $\boz$ satisfies
a UEBC with the uniform ball constant $\fz$ (see, for instance, \cite{e58}).
Thus, convex domains in $\rn$ are Lipschitz domains satisfying a UEBC and hence
convex domains in $\rn$ are semi-convex domains
(see, for instance, \cite{mmmy10,mmy10,y16,yyy18}).

\item[\rm(iii)] Let $n\ge2$, $\dz\in(0,1)$ and $R_0\in(0,\fz)$.
A domain $\boz\subset\rn$ is called a $(\dz,R_0)$-\emph{Reifenberg
flat domain} if, for any $x_0\in\partial\boz$ and $r\in(0,R_0]$, there exists
a system of coordinates, $\{y_1,\,\ldots,\,y_n\}$, which may depend on $x_0$ and $r$,
such that, in this coordinate system, $x_0=\mathbf{0}_n$ and
\begin{equation}\label{1.22}
B(\mathbf{0}_n,r)\cap\{y\in\rn:\ y_n>\dz r\}\subset B(\mathbf{0}_n,r)\cap\boz\subset
B(\mathbf{0}_n,r)\cap\{y\in\rn:\ y_n>-\dz r\},
\end{equation}
where $\mathbf{0}_{n}$ denotes the \emph{origin} of $\rr^{n}$.
The Reifenberg flat domain was introduced by Reifenberg \cite{r60},
which naturally appears in the theory of minimal surfaces and free boundary problems.
In recent years, boundary value problems of elliptic or parabolic equations
on Reifenberg flat domains have been widely concerned and studied
(see, for instance, \cite{b20,b18,bd17a,bd17,bow14,bw05,bw04,byz08,mp12,mp11}).
\vspace{-0.5em}
\item[\rm(iv)] On Reifenberg flat domains, Lipschitz domains, and (semi-)convex domains,
we have the following relations. Lipschitz domains with small Lipschitz constants
are Reifenberg flat domains, but generally Lipschitz domains may
not be Reifenberg flat domains (see, for instance, \cite{t97}).
(Semi-)convex domains are Lipschitz domains, but may not be Lipschitz domains with
small Lipschitz constants or Reifenberg flat domains.
For instance, let $\boz:=\{(x_1,x_2)\in\rr^2:\
1>x_2>|x_1|\}$. It is easy to see that $\boz\subset\rr^2$ is
a bounded convex domain. However, at the points on $\partial\boz$ near
the vertexes of $\boz$, \eqref{1.22} does not hold true. Thus,
$\boz$ is not a Reifenberg flat domain
and hence not a Lipschitz domain with small Lipschitz constant.
\end{enumerate}
\end{remark}

Using Theorem \ref{t1.1} and a weak reverse H\"older inequality
for the non-tangential maximal function of the gradient of the solution of the Robin problem
of the local Laplace equation established in \cite[Theorems 1.2 and 3.2]{yyy18},
we obtain the following (weighted) global $W^{1,\,p}$ estimates for the Robin
problems \eqref{1.6} and \eqref{1.8} on bounded $C^1$ or (semi-)convex domains.

\begin{theorem}\label{t1.3}
Let $n\ge3$, $\boz\subset\rn$ be a bounded $C^1$ or (semi-)convex domain,
$p\in(1,\fz)$, $\omega\in A_p(\rn)$, and $\az\in L^\fz(\partial\boz)$
satisfy $\az\ge \az_0$, where $\az_0\in(0,\fz)$ is a constant.
Assume that the matrix $A$ is real-valued,
symmetric, bounded and measurable, and satisfies \eqref{1.3}.
\begin{itemize}
\item[\rm(i)] Then there exists a positive constant $\dz_0\in(0,\fz)$,
depending only on $n$, $p$ and $\boz$, such that, if $A$ satisfies the $(\dz,R)$-$\mathrm{BMO}$
condition for some $\dz\in(0,\dz_0)$ and $R\in(0,\fz)$ or $A\in\mathrm{VMO}(\rn)$, then
the Robin problem $(R)_{p}$ with $\mathbf{f}\in L^p(\boz;\rn)$,
$F\in L^{p_\ast}(\boz)$ and $g\in W^{-1/p,p}(\partial\boz)$ is uniquely solvable and
there exists a positive constant $C$, depending only on $n$, $p$ and $\boz$, such that,
for any weak solution $u$ of the Robin problem $(R)_{p}$, $u\in W^{1,p}(\boz)$ and
\begin{equation}\label{1.23}
\|u\|_{W^{1,p}(\boz)}\le C\lf[\|\mathbf{f}\|_{L^p(\boz;\rn)}
+\|F\|_{L^{p_\ast}(\boz)}+\|g\|_{W^{-1/p,p}(\partial\boz)}\r],
\end{equation}
where $p_\ast$ is as in \eqref{1.5}.

\item[\rm(ii)]
Let
\begin{equation*}
q:=\frac{np}{n+p-1}
\quad\text{and}\quad a:=\frac{n-1}{n+p-1}.
\end{equation*}
Then there exists a positive constant $\dz_0\in(0,\fz)$,
depending only on $n$, $p$, $\boz$ and $[\omega]_{A_p(\rn)}$,  such that, if
$A$ satisfies the $(\dz,R)$-$\mathrm{BMO}$ condition for some $\dz\in(0,\dz_0)$
and $R\in(0,\fz)$ or $A\in\mathrm{VMO}(\rn)$, then the weighted Robin problem $(R)_{p,\,\omega}$ with
$\mathbf{f}\in L^p_\omega(\boz;\rn)$ and $F\in L^{q}_{\omega^a}(\boz)$ is uniquely solvable and there exists a positive constant
$C$, depending only on $n$, $p$, $[\omega]_{A_p(\rn)}$ and $\boz$, such that,
for any weak solution $u$, $u\in W^{1,p}_\omega(\boz)$ and
\begin{equation}\label{1.24}
\|u\|_{W^{1,p}_\omega(\boz)}\le C\lf[\|\mathbf{f}\|_{L^p_\omega(\boz;\rn)}
+\|F\|_{L^{q}_{\omega^{a}}(\boz)}\r].
\end{equation}

\item[\rm(iii)]
Assume further that $p\in(n/(n-1),\fz)$ and $\omega\in A_1(\rn)$.
Then there exists a positive constant $\dz_0\in(0,\fz)$,
depending only on $n$, $p$, $\boz$ and $[\omega]_{A_1(\rn)}$,  such that, if
$A$ satisfies the $(\dz,R)$-$\mathrm{BMO}$ condition for some $\dz\in(0,\dz_0)$
and $R\in(0,\fz)$ or $A\in\mathrm{VMO}(\rn)$, then the weighted Robin problem $(R)_{p,\,\omega}$ with
$\mathbf{f}\in L^p_\omega(\boz;\rn)$ and $F\in L^{p_\ast}_{\omega^{p_\ast/p}}(\boz)$ is uniquely solvable
and there exists a positive constant $C$, depending only on $n$, $p$, $[\omega]_{A_1(\rn)}$ and $\boz$, such that,
for any weak solution $u$, $u\in W^{1,p}_\omega(\boz)$ and
\begin{equation}\label{1.25}
\|u\|_{W^{1,p}_\omega(\boz)}\le C\lf[\|\mathbf{f}\|_{L^p_\omega(\boz;\rn)}
+\|F\|_{L^{p_\ast}_{\omega^{p_\ast/p}}(\boz)}\r].
\end{equation}
\end{itemize}
\end{theorem}

To prove Theorem \ref{t1.3},
using Theorem \ref{t1.1}, we need to prove that
the weak reverse H\"older inequality \eqref{1.14} is valid for any
$p\in(n/(n-1),\fz)$ when $\boz$ is a bounded $C^1$ or (semi-)convex domain.
To this end, we apply the real-variable argument obtained in Theorem \ref{t3.1} below,
a comparison principle (see Lemma \ref{l5.4} below) inspired by \cite{cp98}, and a weak reverse H\"older inequality
for the non-tangential maximal function of the gradient of the solution of the Robin problem
of the local Laplace equation obtained in \cite[Theorems 1.2 and 3.2]{yyy18}.
Moreover, the weighted Sobolev inequality (see, for instance, \cite[Theorem 3.1]{fks82}),
the Friedrichs inequality (see, for instance, \cite[Section 1.1.8, Theorem 1.9]{n12}
and \cite[Theorem 6.1]{ls04}) and the duality argument
are subtly used in the proof of Theorem \ref{t1.3}.

We point out that the approach used in this article to obtain the weak reverse H\"older
inequality \eqref{1.14} is deferent from that used in \cite{acgg19,g18,ycyy20}.
More precisely, via applying geometric properties of (semi-)convex domains,
a Bernstein type identity on the boundary of the domain
(see \cite[(3.1.1.2)]{g85} and \cite[Theorem 3.2]{mmmy10}) and some ideas from \cite{gs10},
it was proved in \cite[Lemma 1.6]{g18} and \cite[Theorem 2.4]{y16} that,
for the local Neumann problem similar to that in Theorem \ref{t1.1}(ii), \eqref{1.14}
holds true for any given $p\in(2,\fz)$ if the coefficient matrix $A:=I$.
Combining this conclusion, a perturbation argument
and the $(\dz,R)$-BMO condition for $A$, the weak reverse H\"older
inequality \eqref{1.14} was obtained in \cite{g18,ycyy20} for any given
$p\in(2,\fz)$ in the case of bounded (semi-)convex domains for the Neumann problem.
However, the Bernstein type identity, used in \cite{g18,ycyy20}, is
only compatible with the Dirichlet or the Neumann boundary condition,
but not valid for the Robin boundary condition.
To overcome this difficulty, in this article,
we resort to a weak reverse H\"older inequality
for the non-tangential maximal function of the gradient of the solution
of the local Laplace equation with the Robin boundary condition
obtained in \cite[Theorems 1.2 and 3.2]{yyy18} to prove the inequality \eqref{1.14}
for any $p\in(n/(n-1),\fz)$ when $\boz$ is a bounded $C^1$ or (semi-)convex domain.
It is worth pointing out that the approach used in this article is also valid
for the Neumann problem studied in \cite{g18,ycyy20}.

\begin{remark}\label{r1.6}
In the special case of Theorem \ref{t1.3} that
$\boz\subset\rr^3$ is a bounded $C^1$ domain and
$A\in\mathrm{VMO}(\rr^3)$, the global estimate \eqref{1.23}
was established in \cite[Theorem 1.1]{acgg19} via a different way, which is similar to
that used in \cite{g12,g18}.
However, two inequalities used in \cite[p.\,8, line 10 and p.\,11, line 12]{acgg19},
which are essential for the proof of \cite[Theorem 1.1]{acgg19},
are not correct. Thus, even in the special case that
$\boz\subset\rr^3$ is a bounded $C^1$ domain and $A\in\mathrm{VMO}(\rr^3)$, the proof of
Theorem \ref{t1.3}(i) provides an alternative correct proof of \cite[Theorem 1.1]{acgg19}.
\end{remark}

\begin{remark}\label{r1.5}
Let $\boz$, $\az$ and $A$ be as in Theorem \ref{t1.3}.
\begin{itemize}
\item[(i)] By Remark \ref{r1.1}(b), Theorem \ref{t1.3}(i)
and Lemma \ref{l5.5} below, we have the following conclusion.

Let $p\in(1,\fz)$ and $\omega\in A_{p}(\rn)$.
Then there exists a positive constant $\dz_0\in(0,\fz)$,
depending only on $n$, $p$, $\boz$ and $[\omega]_{A_p(\rn)}$, such that, if
$A$ satisfies the $(\dz,R)$-$\mathrm{BMO}$ condition for some $\dz\in(0,\dz_0)$
and $R\in(0,\fz)$ or $A\in\mathrm{VMO}(\rn)$, then the weighted Robin problem $(R)_{p,\,\omega}$ with
$\mathbf{f}\in L^p_\omega(\boz;\rn)$ and $F\in L^{p}_{\omega}(\boz)$ is uniquely solvable
and there exists a positive constant $C$, depending only on $n$, $p$,
$[\omega]_{A_p(\rn)}$ and $\boz$, such that, for any weak solution $u$,
$u\in W^{1,p}_\omega(\boz)$ and
\begin{equation}\label{1.26}
\|u\|_{W^{1,p}_\omega(\boz)}\le C\lf[\|\mathbf{f}\|_{L^p_\omega(\boz;\rn)}
+\|F\|_{L^{p}_{\omega}(\boz)}\r].
\end{equation}

We point out that for the weighted Neumann problem $(N)_{p,\,\omega}$,
the estimate \eqref{1.26} was obtained in \cite[Theorem 1.6 and Corollary 1.7]{ycyy20}.
\item[(ii)] We now clarify the relations among three weighted global estimates,
respectively, obtained in \eqref{1.24}, \eqref{1.25} and \eqref{1.26} as follows.

For \eqref{1.24} and \eqref{1.26}, the conditions for $p$ and $\omega$
are the same but the estimate in \eqref{1.24} is better than that in \eqref{1.26} because
$q<p$ and $\|F\|_{L^{q}_{\omega^{a}}(\boz)}\le C\|F\|_{L^{p}_{\omega}(\boz)}$,
where $C$ is a positive constant independent of $F$.
For \eqref{1.24} and \eqref{1.25}, the conditions for $p$ and $\omega$ in \eqref{1.24}
are weaker that those in \eqref{1.25}; however,  when $p\in(n/(n-1),\fz)$ and
$\omega\in A_1(\rn)$, the estimate in \eqref{1.25} is better than that in \eqref{1.24} because
$p_\ast=\frac{np}{n+p}<\frac{np}{n+p-1}=q$ and $\|F\|_{L^{p_\ast}_{\omega^{p_\ast/p}}(\boz)}
\le C\|F\|_{L^{q}_{\omega^a}(\boz)}$,
where $C$ is a positive constant independent of $F$.
Moreover, for \eqref{1.25} and \eqref{1.26}, the conditions for $p$ and $\omega$ in \eqref{1.26}
are weaker that those in \eqref{1.25}; however,  when $p\in(n/(n-1),\fz)$ and
$\omega\in A_1(\rn)$, the estimate in \eqref{1.25} is better than that in \eqref{1.26} because
$p_\ast<p$ and $\|F\|_{L^{p_\ast}_{\omega^{p_\ast/p}}(\boz)}
\le C\|F\|_{L^{p}_{\omega}(\boz)}$,
where $C$ is a positive constant independent of $F$.

Furthermore, we point out that we establish \eqref{1.25} and \eqref{1.26} by using
Theorem \ref{t1.1}, Remark \ref{r1.1} and Theorem \ref{t1.3}(i).
However, we obtain \eqref{1.24} via using Theorems \ref{t1.1} and \ref{t1.3}(i),
and the weighted Sobolev inequality (see, for instance, \cite[Theorem 1.5]{fks82}
or Lemma \ref{l5.9} below).
\end{itemize}
\end{remark}

By applying the weighted norm inequality obtained in Theorems \ref{t1.2}
and \ref{t1.3}, and some tools from harmonic analysis, such as the properties
of Muckenhoupt weights and the Rubio de Francia extrapolation theorem established in \cite{ch18,cw17},
we further obtain the global regularity estimates for the Robin problem \eqref{1.6},
respectively, in Morrey spaces, (Musielak--)Orlicz spaces (also called generalized Orlicz
spaces) and variable Lebesgue spaces, which have independent interests
and are presented in Section \ref{s2} below.
We point out that the approach used in this
article to establish the globally gradient estimates in both Orlicz spaces
and variable Lebesgue spaces is quite different from that used in \cite{bow14,byz08}.
In \cite{bow14,byz08}, the globally gradient estimates in variable Lebesgue spaces or in Orlicz spaces
were established via the so-called ``maximum function free technique".
However, in this article, we obtain the globally gradient estimates in both Orlicz spaces and
variable Lebesgue spaces by simply using weighted norm inequalities in
Theorems \ref{t1.2} and \ref{t1.3}, and the extrapolation theorem.
It is worth pointing out that the extrapolation theorem used in this article
is also valid for the boundary value problem studied in
\cite{bow14,byz08} and independent of the boundary value condition
and the considered equation.

This article is organized as follows. In Section \ref{s2},
several applications of the global weighted estimates in
Theorems \ref{t1.2} and \ref{t1.3} are given, but their proofs are given in Section \ref{s6}.
In Section \ref{s3}, we establish a weighted real-variable argument,
which is a key tool for the proof of Theorem \ref{t1.2} and
has the independent interest. In Section \ref{s4}, we give the proofs of Theorems \ref{t1.1}
and \ref{t1.2} via using the weighted real-variable argument obtained in Section \ref{s3}
and, in Section \ref{s5}, we prove Theorem \ref{t1.3} by using Theorem \ref{t1.1} and
the weighted Sobolev inequality.

Finally, we make some conventions on notation.
Throughout the whole article, we always denote by $C$ a \emph{positive constant} which is
independent of main parameters, but it may vary from line to
line. We also use $C_{(\gz,\,\bz,\,\ldots)}$ or $c_{(\gz,\,\bz,\,\ldots)}$
to denote a  \emph{positive constant} depending on the indicated parameters $\gz,$ $\bz$,
$\ldots$. The \emph{symbol} $f\ls g$ means that $f\le Cg$. If $f\ls
g$ and $g\ls f$, then we write $f\sim g$. We also use the following
convention: If $f\le Cg$ and $g=h$ or $g\le h$, we then write $f\ls g\sim h$
or $f\ls g\ls h$, \emph{rather than} $f\ls g=h$
or $f\ls g\le h$. For any given normed spaces $\mathcal A$ and $\mathcal
B$ with the corresponding norms $\|\cdot\|_{\mathcal A}$ and
$\|\cdot\|_{\mathcal B}$, the \emph{symbol} ${\mathcal
A}\subset{\mathcal B}$ means that, for any $f\in \mathcal A$, then
$f\in\mathcal B$ and $\|f\|_{\mathcal B}\ls \|f\|_{\mathcal A}$.
For each ball $B:=B(x_B,r_B)$ in $\rn$, with some $x_B\in\rn$,
$r_B\in (0,\fz)$, and $\az\in(0,\fz)$, let $\az B:=B(x_B,\az r_B)$;
furthermore, denote the set $B(x,r)\cap\boz$ by $B_\boz(x,r)$ and the set $(\az B)\cap\boz$ by $\az B_\boz$.
For any subset $E$ of $\rn$, we denote the \emph{set} $\rn\setminus E$ by $E^\complement$
and its \emph{characteristic function} by $\mathbf{1}_{E}$ .
For any $\omega\in A_p(\rn)$ with $p\in[1,\fz)$ and any measurable set $E\subset\rn$,
let $\omega(E):=\int_E\omega(x)\,dx$. For any given $q\in[1,\fz]$, we denote by $q'$
its \emph{conjugate exponent}, namely, $1/q + 1/q'= 1$.
Finally, for any measurable set $E\subset\rn$, $\omega\in A_q(\rn)$ with some $q\in[1,\fz)$
and $f\in L^1(E)$, we denote the integral $\int_E|f(x)|\omega(x)\,dx$
simply by $\int_E|f|\omega\,dx$ and, when $|E|<\fz$, we use the notation
$$\fint_E fdx:=\frac{1}{|E|}\int_Ef(x)dx.$$

\section{Several Applications of Theorems \ref{t1.2} and \ref{t1.3}\label{s2}}

\hskip\parindent In this section, we give several applications of the weighted global
estimates obtained in Theorems \ref{t1.2} and \ref{t1.3}, whose proofs are given in
Section \ref{s6}. More precisely, using Theorems \ref{t1.2}(ii) and \ref{t1.3}(iii),
we obtain the global regularity estimates, respectively, in Morrey spaces,
(Musielak--)Orlicz spaces (also called generalized Orlicz spaces) and variable Lebesgue spaces.
We first recall the definition of the Morrey space
$\cm^{\tz}_p(\boz)$ on the domain $\boz$ as follows.

\begin{definition}\label{d2.1}
Let $n\ge2$, $\boz\subset\rn$ be a bounded Lipschitz domain, $p\in(1,\fz)$
and $\tz\in[0,n]$. The \emph{Morrey space} $\cm^{\tz}_p(\boz)$ is defined by setting
$$\cm^{\tz}_p(\boz):=\lf\{f\ \text{is measurable on}\ \boz:\
\|f\|_{\cm^{\tz}_p(\boz)}<\fz\r\},
$$
where
$$\|f\|_{\cm^{\tz}_p(\boz)}:=\sup_{\rho\in(0,\diam(\boz)]}\sup_{x\in\boz}
\lf\{\rho^{\frac{\tz-n}{p}}\|f\|_{L^{p}(B(x,\rho)\cap\boz)}\r\}.
$$
Moreover, the \emph{space} $\cm^{\tz}_p(\boz;\rn)$ is defined via replacing
$L^p_\omega(\boz)$ [see \eqref{1.1}] by the above $\cm^{\tz}_p(\boz)$
in the definition of $L^p_\omega(\boz;\rn)$ [see \eqref{1.2}].
\end{definition}

Via applying the weighted global regularity estimates obtained in Theorem \ref{t1.2}(ii)
and \eqref{1.19} and the relation between weighted Lebesgue spaces and Morrey spaces,
we obtain the following global regularity estimates in Morrey spaces
for the Robin problem \eqref{1.6} in bounded Lipschitz domains.

\begin{theorem}\label{t2.1}
Let $n\ge3$, $A$, $\az$ and $\boz$ be as in Theorem \ref{t1.2},
$p\in((3+\uc_0)',3+\uc_0)$ and $\tz\in(pn/(3+\uc_0),n]$,
where $\uc_0\in(0,\fz)$ is as in Theorem \ref{t1.2}(i).
\begin{itemize}
\item[{\rm(i)}] There exists a positive constant $\dz_0\in(0,\fz)$,
depending only on $n$, $p$, $\tz$ and the Lipschitz constant of $\boz$,  such that, if
$A$ satisfies the $(\dz,R)$-$\mathrm{BMO}$
condition for some $\dz\in(0,\dz_0)$ and $R\in(0,\fz)$ or $A\in\mathrm{VMO}(\rn)$,
then, for any weak solution $u\in W^{1,2}(\boz)$ of the Robin problem $(R)_2$
with $\mathbf{f}\in \cm^{\tz}_p(\boz;\rn)$, $F\equiv0$ and $g\equiv0$,
$|u|+|\nabla u|\in \cm^{\tz}_p(\boz)$ and
\begin{equation}\label{2.1}
\|u\|_{\cm^{\tz}_p(\boz)}+\|\nabla u\|_{\cm^{\tz}_p(\boz;\rn)}\le
C\|\mathbf{f}\|_{\cm^{\tz}_p(\boz;\rn)},
\end{equation}
where $C$ is a positive constant depending only on $n$, $p$, $\tz$,
$\diam(\boz)$ and the Lipschitz constant of $\boz$.
\item[{\rm(ii)}] Assume further that $p\in(\max\{n/(n-1),(3+\uc_0)'\},3+\uc_0)$.
Then there exists a positive constant $\dz_0\in(0,\fz)$,
depending only on $n$, $p$, $\tz$ and the Lipschitz constant of $\boz$,  such that, if
$A$ satisfies the $(\dz,R)$-$\mathrm{BMO}$
condition for some $\dz\in(0,\dz_0)$ and $R\in(0,\fz)$ or $A\in\mathrm{VMO}(\rn)$,
then, for any weak solution $u\in W^{1,2}(\boz)$ of the Robin problem $(R)_2$
with $\mathbf{f}\in \cm^{\tz}_p(\boz;\rn)$, $F\in \cm^{\wz{\tz}}_{p_\ast}(\boz)$
and $g\equiv0$,
$|u|+|\nabla u|\in \cm^{\tz}_p(\boz)$ and
\begin{equation}\label{2.2}
\|u\|_{\cm^{\tz}_p(\boz)}+\|\nabla u\|_{\cm^{\tz}_p(\boz;\rn)}\le
C\lf[\|\mathbf{f}\|_{\cm^{\tz}_p(\boz;\rn)}+\|F\|_{\cm^{\wz{\tz}}_{p_\ast}(\boz)}\r],
\end{equation}
where $p_\ast$ is as in \eqref{1.5}, $\wz{\tz}$ is given by $\wz{\tz}:=p_\ast(1+\frac{\tz}{p})$
and $C$ is a positive constant depending only on $n$, $p$, $\tz$,
$\diam(\boz)$ and the Lipschitz constant of $\boz$.
\end{itemize}
\end{theorem}

Similarly to Theorem \ref{t2.1},
we have the following global regularity estimates in Morrey spaces
for the Robin problem \eqref{1.6} in bounded $C^1$ or (semi-)convex domains.

\begin{theorem}\label{t2.2}
Let $n\ge3$, $A$, $\az$ and $\boz$ be as in Theorem \ref{t1.3},
$p\in(1,\fz)$ and $\tz\in(0,n]$.
\begin{itemize}
\item[{\rm(i)}] Assume that $q:=\frac{np}{n+p-1}$ and $\wz{\tz}:=n-\frac{(n-1)(n-\tz)}{n+p-1}$.
Then there exists a positive constant $\dz_0\in(0,\fz)$,
depending only on $n$, $p$, $\tz$ and $\boz$, such that, if
$A$ satisfies the $(\dz,R)$-$\mathrm{BMO}$
condition for some $\dz\in(0,\dz_0)$ and $R\in(0,\fz)$ or $A\in\mathrm{VMO}(\rn)$,
then, for any weak solution $u\in W^{1,2}(\boz)$ of the Robin problem $(R)_2$
with $\mathbf{f}\in \cm^{\tz}_p(\boz;\rn)$, $F\in\cm^{\wz{\tz}}_q(\boz)$
and $g\equiv0$,
$|u|+|\nabla u|\in \cm^{\tz}_p(\boz)$ and
\begin{equation}\label{2.3}
\|u\|_{\cm^{\tz}_p(\boz)}+\|\nabla u\|_{\cm^{\tz}_p(\boz;\rn)}\le
C\lf[\|\mathbf{f}\|_{\cm^{\tz}_p(\boz;\rn)}+\|F\|_{\cm^{\wz{\tz}}_{q}(\boz)}\r],
\end{equation}
where $C$ is a positive constant depending only on $n$, $p$, $\tz$ and $\boz$.
\item[{\rm(ii)}] Assume further that $p\in(n/(n-1),\fz)$.
Then there exists a positive constant $\dz_0\in(0,\fz)$,
depending only on $n$, $p$, $\tz$ and $\boz$,  such that, if
$A$ satisfies the $(\dz,R)$-$\mathrm{BMO}$
condition for some $\dz\in(0,\dz_0)$ and $R\in(0,\fz)$ or $A\in\mathrm{VMO}(\rn)$,
then, for any weak solution $u\in W^{1,2}(\boz)$ of the Robin problem $(R)_2$
with $\mathbf{f}\in \cm^{\tz}_p(\boz;\rn)$, $F\in \cm^{\wz{\tz}}_{p_\ast}(\boz)$
and $g\equiv0$,
$|u|+|\nabla u|\in \cm^{\tz}_p(\boz)$ and
\begin{equation}\label{2.4}
\|u\|_{\cm^{\tz}_p(\boz)}+\|\nabla u\|_{\cm^{\tz}_p(\boz;\rn)}\le
C\lf[\|\mathbf{f}\|_{\cm^{\tz}_p(\boz;\rn)}+\|F\|_{\cm^{\wz{\tz}}_{p_\ast}(\boz)}\r],
\end{equation}
where $p_\ast$ is as in \eqref{1.5}, $\wz{\tz}$ is given by $\wz{\tz}:=p_\ast(1+\frac{\tz}{p})$
and $C$ is a positive constant depending only on $n$, $p$, $\tz$ and
$\boz$.
\end{itemize}
\end{theorem}

Recall that, for any $\az\in(0,1]$, the \emph{H\"older space} $C^{0,\az}(\overline{\boz})$ on $\boz$
is defined by setting
$$C^{0,\az}(\overline{\boz}):=\lf\{g\ \text{is continuous on}\ \boz:\
[g]_{C^{0,\az}(\overline{\boz})}:=\sup_{x,\,y\in\boz,\,x\neq y}\frac{|g(x)-g(y)|}{|x-y|^\az}<\fz\r\}.
$$

Then, by Theorems \ref{t2.1} and \ref{t2.2} and the Sobolev--Morrey embedding theorem (see, for instance,
\cite[Theorem 7.19]{gt01}), we obtain the following conclusion.

\begin{corollary}\label{c2.1}
\begin{itemize}
\item[{\rm(i)}] Let $n\ge3$, $A$, $\az$ and $\boz$ be as in Theorem \ref{t1.2}.
Assume that $3+\uc_0>n$, $p\in(n/(n-1),3+\uc_0)$
and $\tz\in(pn/(3+\uc_0),n]\cap(0,p)$,
where $\uc_0\in(0,\fz)$ is as in Theorem \ref{t1.2}(i).
Then there exists a positive constant $\dz_0\in(0,\fz)$,
depending only on $n$, $p$, $\tz$ and the Lipschitz constant of $\boz$, such that, if
$A$ satisfies the $(\dz,R)$-$\mathrm{BMO}$ condition for some $\dz\in(0,\dz_0)$
and $R\in(0,\fz)$ or $A\in\mathrm{VMO}(\rn)$,
then, for any weak solution $u\in W^{1,2}(\boz)$ of the Robin problem $(R)_2$
with $\mathbf{f}\in \cm^{\tz}_p(\boz;\rn)$, $F\in \cm^{\wz{\tz}}_{p_\ast}(\boz)$
and $g\equiv0$, $u\in C^{0,1-\frac{\tz}{p}}(\overline{\boz})$ and
\begin{equation*}
[u]_{C^{0,1-\frac{\tz}{p}}(\overline{\boz})}\le C\lf[\|\mathbf{f}\|_{\cm^{\tz}_p(\boz;\rn)}
+\|F\|_{\cm^{\wz{\tz}}_{p_\ast}(\boz)}\r],
\end{equation*}
where $p_\ast$ is as in \eqref{1.5}, $\wz{\tz}$ is
given by $\wz{\tz}:=p_\ast(1+\frac{\tz}{p})$ and $C$ is a positive constant depending only on $n$, $p$, $\tz$,
$\diam(\boz)$ and the Lipschitz constant of $\boz$.
\item[{\rm(ii)}] Let $n\ge3$, $A$, $\az$ and $\boz$ be as in Theorem \ref{t1.3},
$\tz\in(0,n]$ and $p\in(\max\{n/(n-1),\tz\},\fz)$.
Then there exists a positive constant $\dz_0\in(0,\fz)$,
depending only on $n$, $p$, $\tz$ and $\boz$, such that, if
$A$ satisfies the $(\dz,R)$-$\mathrm{BMO}$ condition for some $\dz\in(0,\dz_0)$
and $R\in(0,\fz)$ or $A\in\mathrm{VMO}(\rn)$,
then, for any weak solution $u\in W^{1,2}(\boz)$ of the Robin problem $(R)_2$
with $\mathbf{f}\in \cm^{\tz}_p(\boz;\rn)$, $F\in \cm^{\wz{\tz}}_{p_\ast}(\boz)$
and $g\equiv0$,
$u\in C^{0,1-\frac{\tz}{p}}(\overline{\boz})$ and
\begin{equation*}
[u]_{C^{0,1-\frac{\tz}{p}}(\overline{\boz})}\le C\lf[\|\mathbf{f}\|_{\cm^{\tz}_p(\boz;\rn)}
+\|F\|_{\cm^{\wz{\tz}}_{p_\ast}(\boz)}\r],
\end{equation*}
where $p_\ast$ is as in \eqref{1.5}, $\wz{\tz}$ is
given by $\wz{\tz}:=p_\ast(1+\frac{\tz}{p})$ and $C$ is a positive constant
depending only on $n$, $p$, $\tz$ and $\boz$.
\end{itemize}
\end{corollary}

\begin{remark}\label{r2.1}
We point out that, for the Dirichlet problem \eqref{1.9} with $F\equiv0$,
the estimate \eqref{2.3} was established in
\cite[Theorem 2.3]{amp18} under the assumptions that $A$ is symmetric
and satisfies the $(\dz,R)$-BMO condition for some small $\dz\in(0,\fz)$ and some $R\in(0,\fz)$,
and that $\boz$ is a bounded Lipschitz domain with a small Lipschitz constant.
Moreover, estimates similar to \eqref{2.3} with $F\equiv0$ for the Dirichlet problem
of some nonlinear elliptic or parabolic equations on Reifenberg flat domains
were obtained in \cite{ap16,ap15,bd17a,bd17,mp12,mp11}. Furthermore,
for the Neumann problem $(N)_p$ in bounded (semi-)convex domains,
the estimate \eqref{2.3}, with $\|F\|_{\cm^{\wz{\tz}}_{q}(\boz)}$ replaced
by $\|F\|_{\cm^{\tz}_{p}(\boz)}$, was obtained in \cite[Corollary 2.5]{ycyy20}.
\end{remark}

In what follows, a function $f:\ [0,\fz)\to[0,\fz]$ is said to be \emph{almost
increasing} (resp., \emph{almost decreasing}) if there exists a positive constant $L\in[1,\fz)$
such that, for any $s,\,t\in[0,\fz)$ with $s\le t$, $f(s)\le Lf(t)$
[resp., $f(s)\ge Lf(t)$]. In particular, if $L:=1$, then
$f$ is said to be increasing (resp., decreasing). Now we recall the definitions of weak
$\Phi$-functions and Musielak--Orlicz spaces (also called generalized Orlicz spaces)
as follows (see, for instance, \cite{ch18,hh19,ho19,rr91,ylk17}).
Recall that the \emph{symbol $t\to0^+$} means $t\in(0,\fz)$ and $t\to0$.

\begin{definition}\label{d2.2}
Let $\fai:\,[0,\fz)\to[0,\fz]$ be an increasing function satisfying that
$$\fai(0)
=\lim_{t\to0^+}\fai(t)=0\quad \text{and}\quad \lim_{t\to\fz}\fai(t)=\fz.
$$
\begin{itemize}
\item[{\rm(i)}] Then $\fai$ is called a \emph{weak $\Phi$-function}, denoted by
$\fai\in\Phi_w$, if $t\to\frac{\fai(t)}{t}$ is almost increasing on $(0,\fz)$.
\item[{\rm(ii)}] The \emph{left-continuous generalized inverse} of $\fai$,
denoted by $\fai^{-1}$, is defined by setting, for any $s\in[0,\fz]$,
$$\fai^{-1}(s):=\inf\lf\{t\in[0,\fz):\ \fai(t)\ge s\r\}.$$
\item[{\rm(iii)}] The \emph{conjugate $\Phi$-function} of $\fai$, denoted by $\fai^\ast$,
is defined by setting, for any $t\in[0,\fz)$,
$$\fai^\ast(t):=\sup_{s\in[0,\fz)}\{st-\fai(s)\}.
$$
\item[{\rm(iv)}] Let $E\subset\rn$ be a measurable set. A function
$\fai:\ E\times[0,\fz)\to[0,\fz]$ is called a
\emph{Musielak--Orlicz function} (or a \emph{generalized $\Phi$-function}) on $E$
if it satisfies
\begin{enumerate}
\item[$\mathrm{(iv)_1}$] for any $t\in[0,\fz)$, $\fai(\cdot,t)$ is measurable;
\item[$\mathrm{(iv)_2}$] for almost every $x\in E$,
$\fai(x,\cdot)\in\Phi_w$.
\end{enumerate}
Then the set $\Phi_w(E)$ is defined to be the collection of all
Musielak--Orlicz functions on $E$.
\end{itemize}
\end{definition}

\begin{definition}\label{d2.3}
Let $E\subset\rn$ be a measurable set and $\fai\in\Phi_w(E)$. For any given
$f\in L^1_\loc(E)$, the \emph{Musielak--Orlicz modular} of $f$ is defined by setting
$$\rho_\fai(f):=\int_{E}\fai(x,|f(x)|)\,dx.
$$
Then the \emph{Musielak--Orlicz space} (also called \emph{generalized Orlicz space})
$L^\fai(E)$ is defined by setting
\begin{align*}
&L^\fai(E):=\{u\ \text{is measurable on}\ E:\ \\
&\quad\quad\quad\quad\quad\text{there exists a}\ \lz\in(0,\fz)\
\text{such that}\ \rho_\fai(\lz f)<\fz\}
\end{align*}
equipped with the \emph{Luxemburg} (also called the \emph{Luxembourg--Nakano}) \emph{norm}
\begin{equation*}
\|u\|_{L^\fai(E)}:=\inf\lf\{\lz\in(0,\fz):\
\rho_\fai\lf(\frac{u}{\lz}\r)\le1\r\}.
\end{equation*}
\end{definition}

To obtain the global regularity estimates for the Robin problem
in the scale of Musielak--Orlicz spaces, we need several additional assumptions
for the Musielak--Orlicz function $\fai$.
Let $E\subset\rn$ be a measurable set, $\fai\in\Phi_w(E)$ and $p\in(0,\fz)$.

\medskip
\noindent {\bf Assumption (A0).} There exist positive constants
$\beta\in(0,1)$ and $\gamma\in(0,\fz)$ such that,
for any $x\in E$, $\fai(x,\beta\gamma)\le1\le\fai(x,\gamma)$.

\medskip

\noindent {\bf Assumption (A1).} There exists a $\bz\in(0,1)$
such that, for any $x,\,y\in E$ with $|x-y|\le1$ and any $t\in[1,|x-y|^{-n}]$,
$\bz\fai^{-1}(x,t)\le\fai^{-1}(y,t)$.

\medskip

\noindent {\bf Assumption (A2).} There exist $\bz,\,\sz\in(0,\fz)$
and $h\in L^1(E)\cap L^\fz(E)$ such that, for any $t\in[0,\sz]$,
$$\fai(x,\bz t)\le\fai(y,t)+h(x)+h(y).
$$

\medskip

\noindent {\bf Assumption $\mathrm{(aInc)}_p$.} The function
$s\to \frac{\fai(x,s)}{s^p}$ is almost increasing uniformly in $x\in E$.

\medskip

\noindent {\bf Assumption $\mathrm{(aDec)}_p$.} The function
$s\to \frac{\fai(x,s)}{s^p}$ is almost decreasing uniformly in $x\in E$.

\medskip

By using the weighted global regularity estimates obtained in Theorem \ref{t1.2}
and Remark \ref{r1.3}, and the limited range extrapolation theorem
established in \cite[Theorem 4.18 and Corollary 4.21]{ch18} in the scale of
Musielak--Orlicz spaces, we obtain the following global regularity estimates in
Musielak--Orlicz spaces for the Robin problem $(R)_p$ in bounded Lipschitz domains.

\begin{theorem}\label{t2.3}
Let $n\ge3$, $A$, $\az$ and $\boz$ be as in Theorem \ref{t1.2}, and
$p_0,\,p_1\in((3+\uc_0)',3+\uc_0)$ with $p_0\le p_1$, where $\uc_0\in(0,\fz)$ is as in Theorem \ref{t1.2}(i).
Assume that $\fai\in\Phi_w(\boz)$ satisfies Assumptions $(A0)$--$(A2)$,
$\mathrm{(aInc)}_{p_0}$ and $\mathrm{(aDec)}_{p_1}$.
Then there exists a positive constant $\dz_0\in(0,\fz)$,
depending only on $n$, $\fai$ and the Lipschitz constant of $\boz$,  such that, if
$A$ satisfies the $(\dz,R)$-$\mathrm{BMO}$
condition for some $\dz\in(0,\dz_0)$ and $R\in(0,\fz)$ or $A\in\mathrm{VMO}(\rn)$,
then, for any weak solution $u\in W^{1,2}(\boz)$ of the Robin problem $(R)_2$
with $\mathbf{f}\in L^\fai(\boz;\rn)$, $F\equiv0$ and $g\equiv0$,
$|u|+|\nabla u|\in L^\fai(\boz)$ and
\begin{equation}\label{2.5}
\|u\|_{L^\fai(\boz)}+\|\nabla u\|_{L^\fai(\boz;\rn)}\le
C\|\mathbf{f}\|_{L^\fai(\boz;\rn)},
\end{equation}
where $C$ is a positive constant depending only on $n$, $\fai$,
$\diam(\boz)$ and the Lipschitz constant of $\boz$.
\end{theorem}

Furthermore, similarly to Theorem \ref{t2.3}, via using Theorem \ref{t1.3}(iii)
and the off-diagonal extrapolation theorem
established  in \cite[Theorem 4.5 and Corollary 4.8]{ch18} in
Musielak--Orlicz spaces, we obtain the following global regularity estimates in
Musielak--Orlicz spaces for the Robin problem $(R)_p$ in
bounded $C^1$ or (semi-)convex domains.

\begin{theorem}\label{t2.4}
Let $n\ge3$, $A$, $\az$ and $\boz$ be as in Theorem \ref{t1.3},
and $p_0,\,p_1\in(n/(n-1),\fz)$ with $p_1>p_0$.
Assume that $\fai\in\Phi_w(\boz)$ satisfies Assumptions $(A0)$--$(A2)$,
$\mathrm{(aInc)}_{p_0}$ and $\mathrm{(aDec)}_{p_1}$, and,
for any $(x,t)\in\boz\times[0,\fz)$,
$\psi^{-1}(x,t):=t^{1/n}\fai^{-1}(x,t)$.
Then there exists a positive constant $\dz_0\in(0,\fz)$,
depending only on $n$, $\fai$ and $\boz$,  such that, if
$A$ satisfies the $(\dz,R)$-$\mathrm{BMO}$
condition for some $\dz\in(0,\dz_0)$ and $R\in(0,\fz)$ or $A\in\mathrm{VMO}(\rn)$,
then, for any weak solution $u\in W^{1,2}(\boz)$ of the Robin problem $(R)_2$
with $\mathbf{f}\in L^\fai(\boz;\rn)$, $F\in L^\psi(\boz)$
and $g\equiv0$,
$|u|+|\nabla u|\in L^\fai(\boz)$ and
\begin{equation}\label{2.6}
\|u\|_{L^\fai(\boz)}+\|\nabla u\|_{L^\fai(\boz;\rn)}\le
C\lf[\|\mathbf{f}\|_{L^\fai(\boz;\rn)}+\|F\|_{L^\psi(\boz)}\r],
\end{equation}
where $C$ is a positive constant depending only on $n$, $\fai$ and
$\boz$.
\end{theorem}

To describe more corollaries of Theorems \ref{t2.3} and \ref{t2.4},
we recall some necessary notions for variable exponent functions $p(\cdot)$
as follows (see, for instance, \cite{cf13,dhhr11}).
Let $\mathcal{P}(\rn)$ be the set of all measurable functions
$p:\,\rn\to[1,\fz)$. For any $p\in\mathcal{P}(\rn)$, let
\begin{equation}\label{2.7}
p_+:=\mathop{\mathrm{ess\,sup}}\limits_{x\in\rn}p(x)\ \
\text{and} \ \ p_-:=\mathop{\mathrm{ess\,inf}}\limits_{x\in\rn}p(x).
\end{equation}
Recall that a function $p:\ \rn\to\rr$ is said to satisfy the \emph{local
log-H\"older continuity condition} if there exists a positive constant $C_\loc$ such that, for any
$x,\,y\in\rn$ with $|x-y|\le1/2$,
$$|p(x)-p(y)|\le\frac{C_\loc}{-\log(|x-y|)};
$$
a function $p:\ \rn\to\rr$ is said to satisfy the \emph{log-H\"older decay condition} (at infinity)
if there exist positive constants $C_\fz\in(0,\fz)$ and $p_\fz\in[1,\fz)$ such that, for any
$x\in\rn$,
$$|p(x)-p_\fz|\le\frac{C_\fz}{\log(e+|x|)}.
$$
If a function $p$ satisfies both the local log-H\"older continuity condition
and the log-H\"older decay condition,
then the function $p$ is said to satisfy the \emph{log-H\"older continuity condition}.

Then we have the following two corollaries of Theorems \ref{t2.3} and \ref{t2.4}.

\begin{corollary}\label{c2.2}
Assume that $p\in\mathcal{P}(\rn)$ satisfies the log-H\"older continuity condition
and $(3+\uc_0)'<p_-\le p_+<3+\uc_0$, where $\uc_0\in(0,\fz)$ is as in
Theorem \ref{t1.2}(i), and $p_-$ and $p_+$ are as in \eqref{2.7}.
Then the conclusion of Theorem \ref{t2.3} holds true if $\fai$
satisfies one of the following cases:
\begin{itemize}
\item[\rm(i)] for any $x\in\boz$ and $t\in[0,\fz)$,
$\fai(x,t):=\phi(t)$, where $\phi\in\Phi_w$ satisfies Assumptions
$\mathrm{(aInc)}_{p_0}$ and $\mathrm{(aDec)}_{p_1}$ with $(3+\uc_0)'<p_0\le p_1<3+\uc_0$.
\item[\rm(ii)] for any $x\in\boz$ and $t\in[0,\fz)$,
$\fai(x,t):=a(x)t^{p(x)}$, where $C^{-1}\le a\le C$ with $C$ being a positive constant.
\item[\rm(iii)] for any $x\in\boz$ and $t\in[0,\fz)$, $\fai(x,t):=t^{p(x)}\log(e+t)$.
\item[\rm(iv)] for any $x\in\boz$ and $t\in[0,\fz)$, $\fai(x,t):=t^{p}+a(x)t^q$,
where $(3+\uc_0)'<p<q<3+\uc_0$ and $0\le a\in L^\fz(\boz)\cap C^{0,\frac{n}{p}(q-p)}(\boz)$.
\item[\rm(v)] for any $x\in\boz$ and $t\in[0,\fz)$, $\fai(x,t):=t^{p}+a(x)t^{p}\log(e+t)$,
where $(3+\uc_0)'<p<3+\uc_0$ and $0\le a\in L^\fz(\boz)$ satisfies the local
log-H\"older continuity condition.
\end{itemize}
\end{corollary}

\begin{corollary}\label{c2.3}
Assume that $p\in\mathcal{P}(\rn)$ satisfies the log-H\"older continuity condition
and $n/(n-1)<p_-\le p_+<\fz$, where $p_-$ and $p_+$ are as in \eqref{2.7}.
Then the conclusion of Theorem \ref{t2.4} holds true if $\fai$
satisfies one of the following cases:
\begin{itemize}
\item[\rm(i)] for any $x\in\boz$ and $t\in[0,\fz)$,
$\fai(x,t):=\phi(t)$, where $\phi\in\Phi_w$ satisfies Assumptions
$\mathrm{(aInc)}_{p_0}$ and $\mathrm{(aDec)}_{p_1}$ with $n/(n-1)<p_0\le p_1<\fz$.
\item[\rm(ii)] for any $x\in\boz$ and $t\in[0,\fz)$,
$\fai(x,t):=a(x)t^{p(x)}$, where $C^{-1}\le a\le C$ with $C$ being a positive constant.
\item[\rm(iii)] for any $x\in\boz$ and $t\in[0,\fz)$, $\fai(x,t):=t^{p(x)}\log(e+t)$.
\item[\rm(iv)] for any $x\in\boz$ and $t\in[0,\fz)$, $\fai(x,t):=t^{p}+a(x)t^q$,
where $n/(n-1)<p<q<\fz$ and $0\le a\in L^\fz(\boz)\cap C^{0,\frac{n}{p}(q-p)}(\boz)$.
\item[\rm(v)] for any $x\in\boz$ and $t\in[0,\fz)$, $\fai(x,t):=t^{p}+a(x)t^{p}\log(e+t)$,
where $p\in(n/(n-1),\fz)$ and $0\le a\in L^\fz(\boz)$ satisfies the local
log-H\"older continuity condition.
\end{itemize}
\end{corollary}

By \cite{ch18,cw17,hh19,ho19}, we know that the Musielak--Orlicz functions appearing in
Corollaries \ref{c2.2} and \ref{c2.3} satisfy the assumptions on $\fai$ in Theorems
\ref{t2.3} and \ref{t2.4}. We omit the details here.

\begin{remark}\label{r2.2}
For the Dirichlet problem \eqref{1.9} with $F\equiv0$,
Corollary \ref{c2.3}(ii) with $a\equiv1$ was established in \cite[Theorem 2.5]{bow14}
(see also \cite{dr03}) under the assumptions that $A$ is symmetric and has partial
small BMO coefficients, and that $\boz$ is a bounded Reifenberg flat domain. Moreover, the variable exponent
type regularity estimate similar to Corollary \ref{c2.3}(ii)
for the Dirichlet problem of some $p$-Laplace type elliptic equations on Reifenberg
flat domains was also obtained in \cite[Theorem 1.4]{b18}.
Furthermore, for the Neumann problem $(N)_p$,
the conclusion of Corollary \ref{c2.3}(ii) with $a\equiv1$ was established in \cite[Theorem 2.13]{ycyy20}
under the assumptions that $A$ is symmetric and has
small BMO coefficients, and that $\boz$ is a bounded (semi-)convex domain.
\end{remark}

\section{A real-variable argument\label{s3}}

\hskip\parindent In this section, we give a real-variable argument for (weighted) $L^p(\boz)$ estimates,
which is inspired by the work of Caffarelli and Peral \cite{cp98} (see also \cite{w03}).
We point out that, if $\gz=0$ in Theorem \ref{t3.1} below,
the conclusion of Theorem \ref{t3.1} was essentially established in \cite[Theorem 3.4]{sh07} (see also
\cite[Theorems 2.1 and 2.2]{g12}, \cite[Theorem 2.1]{g18},
\cite[Theorem 4.2.6]{sh18} and \cite[Theorem 3.3]{sh05a}). When $\gz\neq0$,
Theorem \ref{t3.1} is a natural generalization of \cite[Theorem 3.4]{sh07} and
\cite[Theorems 2.1 and 2.2]{g12} in the off-diagonal case.
We also mention that a similar argument with a different motivation was also established in \cite{a07,am07}.

\begin{theorem}\label{t3.1}
Let $\gz\in[0,1)$, $n\ge2$, $\boz\subset\rn$ be a bounded Lipschitz domain,
$p_1,\,p_2,\,p_3\in[1,\fz)$
satisfy $p_3>\max\{p_1,\,p_2\}$, $G\in L^{p_1}(\boz)$ and $f\in L^q(\boz)$ with some
$q\in(\max\{p_1,\,p_2\},p_3)$.
Suppose that, for any ball $B:=B(x_B,r_B)\subset\rn$ having the property
that $|B|\le \bz_1|\boz|$ and either $2B\subset\boz$ or $x_B\in\partial\boz$,
there exist two measurable functions $G_B$ and $R_B$ on $2B$
such that $|G|\le|G_B|+|R_B|$ on $2B\cap\boz$,
\begin{align}\label{3.1}
\lf(\fint_{2B_\boz}|R_B|^{p_3}\,dx\r)^{\frac1{p_3}}
\le C_1\lf[\lf(\fint_{\bz_2 B_\boz}|G|^{p_1}\,dx\r)^{\frac1{p_1}}
+\sup_{B\subset \wz{B}}\lf(
|\wz{B}_\boz|^{\gz}\fint_{\wz{B}_\boz}|f|^{p_2}\,dx\r)^{\frac1{p_2}}\r]
\end{align}
and
\begin{align}\label{3.2}
\lf(\fint_{2B_\boz}|G_B|^{p_1}\,dx\r)^{\frac1{p_1}}
\le \uc\lf(\fint_{\bz_2 B_\boz}|G|^{p_1}\,dx\r)^{\frac1{p_1}}
+C_2\sup_{B\subset\wz{B}}\lf(
|\wz{B}_\boz|^{\gz}\fint_{\wz{B}_\boz}|f|^{p_2}\,dx\r)^{\frac1{p_2}},
\end{align}
where $C_1,\,C_2,\,\uc\in(0,\fz)$ and $0<\bz_1<1<\bz_2<\fz$ are positive constants independent
of $G,\,f,\,R_B,\,G_B$ and $B$, and the suprema are taken over all balls $\wz{B}\supset B$.
Assume further that $q$, $p_2$ and $\gz$ satisfy
$$q(1-\gz)>p_2.
$$
Then, for any $\omega\in A_{q(1-\gz)/p_2}(\rn)\cap RH_s(\rn)$
with $s\in((\frac{p_3}{q})',\fz]$, there exists a positive constant $\uc_0$,
depending only on $C_1,\,C_2,\,n,\,p_1,\,p_2,\,q,\,\bz_1,\,\bz_2$, $[\omega]_{A_{q(1-\gz)/p_2}(\rn)}$ and
$[\omega]_{RH_s(\rn)}$, such that, if $\uc\in[0,\uc_0)$, then
\begin{align}\label{3.3}
\lf[\frac{1}{\omega(\boz)}\int_{\boz}|G|^q\omega\,dx\r]^{\frac1q}
&\le C\lf\{\lf(\frac{1}{|\boz|}\int_{\boz}|G|^{p_1}\,dx\r)^{\frac1{p_1}}\r.\\
&\qquad+\lf.
\lf[\frac{1}{\omega(\boz)}\int_{\boz}|f|^{q_0}\omega^{\frac{q_0}{q}}
\,dx\r]^{\frac{1}{q_0}}[\omega(\boz)]^{\frac{\gz}{p_2}}\r\},\nonumber
\end{align}
where $q_0\in(1,\fz)$ is given by $1/q_0=1/q+\gz/p_2$
and $C$ is a positive constant depending only on $C_1$, $C_2$, $n$, $p_1,\,p_2$,
$q$, $\bz_1,\,\bz_2$, $[\omega]_{A_{q(1-\gz)/p_2}(\rn)}$ and $[\omega]_{RH_s(\rn)}$.
\end{theorem}

To prove Theorem \ref{t3.1}, we need to recall some necessary notions as follows.
Let $\gz\in[0,1)$. Then the \emph{fractional Hardy--Littlewood maximal operator}
$\cm_\gz$ on $\rn$ is defined by setting, for any $f\in L^1_{\loc}(\rn)$ and $x\in\rn$,
$$\cm_\gz(f)(x):=\sup_{B\ni x}\lf(|B|^{\gz}\fint_{B}|f|\,dy\r),$$
where the supremum is taken over all balls $B\subset\rn$ containing $x$.
We point out that, when $\gz=0$, the fractional Hardy--Littlewood maximal operator
$\cm_\gz$ is just the well-known \emph{Hardy--Littlewood maximal operator}; in this case,
we denote $\cm_\gz$ simply by $\cm$.
Moreover, let $B_0\subset\rn$ be a ball. For any $f\in L^1_{\loc}(\rn)$,
the \emph{localized fractional Hardy--Littlewood maximal function} $\cm_{\gz,B_0}(f)$
is defined by setting, for any $x\in B_0$,
$$\cm_{\gz,B_0}(f)(x):=\sup_{B\ni x}\lf(|B|^{\gz}\fint_{B}|f|\,dy\r),$$
where the supremum is taken over all balls $B\subset B_0$ containing $x$.
It is easy to see that, for any $f\in L^1_{\loc}(\rn)$ and $x\in B_0$,
\begin{equation}\label{3.4}
\cm_{\gz,B_0}(f)(x)\le\cm_{\gz}\lf(f\mathbf{1}_{B_0}\r)(x).
\end{equation}
Furthermore, when $\gz=0$, we denote $\cm_{\gz,B_0}$ simply by $\cm_{B_0}$.

To show Theorem \ref{t3.1}, we need the following properties of $A_p(\rn)$ weights,
which are well known (see, for instance, \cite[Chapter 7]{g14}).
\begin{lemma}\label{l3.1}
Let $q\in(1,\fz)$, $\omega\in A_q(\rn)$, and $\boz\subset\rn$ be a bounded domain.
\begin{enumerate}
\item[\rm(i)] There exists a $q_1\in(1,q)$, depending only on $n$,
$q$ and $[\omega]_{A_q(\rn)}$, such that $\omega\in A_{q_1}(\rn)$.
\item[\rm(ii)] There exists a $\gamma\in(0,\fz)$,
depending only on $n$, $q$, and $[\omega]_{A_q(\rn)}$, such that $\omega\in RH_{1+\gamma}(\rn)$.
\item[\rm(iii)] If $q'$ denotes the conjugate number of $q$, namely, $1/q+1/q'=1$,
then $\omega^{-q'/q}\in A_{q'}(\rn)$ and
$[\omega^{-q'/q}]_{A_{q'}(\rn)}=[\omega]^{q'/q}_{A_q(\rn)}$.
\item[{\rm(iv)}] Let $\delta\in(0,1)$ and $s:=\delta(q-1)+1$.
Then $\omega^\delta\in A_s(\rn)$
and $[\omega^\delta]_{A_s(\rn)}\le[\omega]_{A_q(\rn)}^\delta$.
\item[\rm(v)] There exists a positive constant $C$, depending only on $n$, $q$ and
$[\omega]_{A_q(\rn)}$, such that, for any ball $B\subset\rn$ and any measurable set
$E\subset B$,
$$\frac{\omega(B)}{\omega(E)}\le C\lf[\frac{|B|}{|E|}\r]^{q}.$$
\item[\rm(vi)] If $\omega\in RH_s(\rn)$ with $s\in(1,\fz]$, then there
exists a positive constant $\wz{C}$, depending only on $n$, $s$ and $[\omega]_{RH_s(\rn)}$,
such that, for any ball $B\subset\rn$ and any measurable set $E\subset B$,
$$\frac{\omega(E)}{\omega(B)}\le \wz{C}\lf[\frac{|E|}{|B|}\r]^{\frac{s-1}s}.$$
\item[\rm(vii)] Let $q_2:=q(1+\frac{1}{\gamma})$ with
$\gamma$ as in (ii) and let $q_1$ be as in (i).
Then $L^{q_2}(\boz)\subset L^q_\omega(\boz)\subset
L^{\frac q{q_1}}(\boz)$.
\end{enumerate}
\end{lemma}

Furthermore, we also need the following boundedness of the fractional
Hardy--Littlewood maximal operator $\cm_\gz$ on weighted (weak) Lebesgue spaces
(see, for instance, \cite[Theorems 2 and 3]{mw74}).

\begin{lemma}\label{l3.2}
Let $\gz\in[0,1)$ and $\cm_\gz$ be the fractional Hardy--Littlewood maximal
operator on $\rn$. Assume that $p\in[1,\gz^{-1})$, $q\in(1,\fz)$ is given by
$\frac{1}{q}=\frac{1}{p}-\gz$,
and $0\le\omega\in L^1_\loc(\rn)$ satisfies $\omega^q\in A_{q(1-\gz)}(\rn)$.
Then there exists a positive constant $C$, depending only on $n$, $p$
and $[\omega^q]_{A_{q(1-\gz)}(\rn)}$, such that,
for any $f\in L^p_{\omega^p}(\rn)$,
$$\sup_{\lz\in(0,\fz)}\lz\lf[\omega^q\lf(\lf\{x\in\rn:\
|\cm_\gz(f)(x)|>\lz\r\}\r)\r]^{\frac1q}
\le C\|f\|_{L^p_{\omega^p}(\rn)}.
$$
Here, for any measurable set $E\subset\rn$, $\omega^q(E):=\int_E[\omega(x)]^q\,dx$.
Moreover, if $p\in(1,\gz^{-1})$, then $\cm_\gz$ is bounded
from $L^p_{\omega^p}(\rn)$ to $L^q_{\omega^q}(\rn)$.
\end{lemma}

Now we prove Theorem \ref{t3.1} by using Lemmas \ref{l3.1} and \ref{l3.2},
and borrowing some ideas from the proofs of \cite[Theorem 3.4]{sh07}
and \cite[Theorem 4.2.3]{sh18}.

\begin{proof}[Proof of Theorem \ref{t3.1}]
Take a ball $B_0\subset\rn$ such that $\boz\subset B_0$ and
its radius $r_{B_0}=\diam(\boz)$. Moreover, let $\wz{G}$ and $\wz{f}$ be, respectively,
the \emph{zero extensions} of $G$ and $f$ to $\rn$. To prove \eqref{3.3},
it suffices to show that
\begin{align}\label{3.5}
&\lf[\frac{1}{\omega(B_0)}\int_{B_0}|\wz{G}|^q\omega\,dx\r]^{\frac1q}\\
&\quad\ls\lf[\frac{1}{|B_0|}\int_{B_0}|\wz{G}|^{p_1}\,dx\r]^{\frac1{p_1}}+
\lf[\frac{1}{\omega(B_0)}\int_{B_0}|\wz{f}|^{q_0}\omega^{\frac{q_0}{q}}
\,dx\r]^{\frac{1}{q_0}}[\omega(B_0)]^{\frac{\gz}{p_2}}.\nonumber
\end{align}

Take a cube $Q_0\subset\rn$ such that $2Q_0\subset2B_0$ and $|Q_0|\sim|B_0|$.
To prove \eqref{3.5}, we only need to show that
\begin{align}\label{3.6}
&\lf[\frac{1}{\omega(Q_0)}\int_{Q_0}|\wz{G}|^q\omega\,dx\r]^{\frac1q}\\
&\quad\ls\lf[\frac{1}{|2B_0|}\int_{2B_0}|\wz{G}|^{p_1}\,dx\r]^{\frac1{p_1}}+
\lf[\frac{1}{\omega(2B_0)}\int_{2B_0}|\wz{f}|^{q_0}\omega^{\frac{q_0}{q}}
\,dx\r]^{\frac{1}{q_0}}[\omega(2B_0)]^{\frac{\gz}{p_2}}.\nonumber
\end{align}
Indeed, if \eqref{3.6} holds true, then we can obtain \eqref{3.5} via
using \eqref{3.6} and covering $B_0$ with a finite number of non-overlapping cubes
$Q_0$ of the same size such that $2Q_0\subset2B_0$.

Now we show \eqref{3.6}. For any $\lz\in(0,\fz)$, let
$$E(\lz):=\lf\{x\in Q_0:\ \cm_{2B_0}\lf(|\wz{G}|^{p_1}\r)(x)>\lz\r\}.
$$
We first claim that there exists a positive constant $\uc_0$,
depending only on $p_1$, $p_2$, $q$, $\bz_1,\,\bz_2$, $C_1,\,C_2$,
$[\omega]_{A_{q(1-\gz)/p_2}(\rn)}$
and $[\omega]_{RH_s(\rn)}$, such that, if
$\uc\in[0,\uc_0)$, one can choose positive constants $\dz\in(0,1/3)$,
$\kappa\in(0,1)$ and $C_3\in(0,\fz)$, depending only on $p_1$, $p_2$, $q$,
$\bz_1,\,\bz_2$, $C_1,\,C_2$, $[\omega]_{A_{q(1-\gz)/p_2}(\rn)}$
and $[\omega]_{RH_s(\rn)}$, such that, for any $\lz\in(\lz_0,\fz)$,
\begin{equation}\label{3.7}
\omega(E(\bz\lz))\le\dz^{\frac{s-1}s}\omega(E(\lz))+\omega\lf(\lf\{x\in Q_0:\
\cm_{\gz,2B_0}\lf(|\wz{f}|^{p_2}\r)(x)>(\kappa\lz)^b\r\}\r),
\end{equation}
where
\begin{equation}\label{3.8}
\lz_0:=C_3\fint_{2B_0}|\wz{F}|^{p_1}\,dx,
\end{equation}
$\bz:=[2\dz^{(s-1)/s}]^{-p_1/q}$ and $b:=p_2/p_1$.

Now we assume that \eqref{3.7} holds true for a moment and prove
\eqref{3.6} via using \eqref{3.7}. Let $T\in(\lz_0,\fz)$ be any given constant.
Multiplying both sides of \eqref{3.7} by $\lz^{\frac{q}{p_1}-1}$ and then integrating in
$\lz$ over the interval $(\lz_0,T)$, we find that
\begin{align}\label{3.9}
\int_{\lz_0}^T\lz^{\frac{q}{p_1}-1}\omega(E(\bz\lz))\,d\lz&\le
\dz^{\frac{s-1}s}\int_{\lz_0}^T\lz^{\frac{q}{p_1}-1}\omega(E(\lz))\,d\lz\\
&\quad+\int_{\lz_0}^T\lz^{\frac{q}{p_1}-1}
\omega\lf(\lf\{x\in Q_0:\ \cm_{\gz,2B_0}\lf(|\wz{f}|^{p_2}\r)(x)>(\kappa\lz)^b\r\}\r)\,d\lz.\nonumber
\end{align}
Moreover, from \eqref{3.4}, Lemma \ref{l3.2},
$q_0=p_2q/(p_2+\gz q)$ and a change of variables,
it follows that
\begin{align*}
&\int_0^\fz\lz^{\frac{q}{p_1}-1}\omega\lf(\lf\{x\in Q_0:\
\cm_{\gz,2B_0}\lf(|\wz{f}|^{p_2}\r)(x)>(\kappa\lz)^b\r\}\r)\,d\lz\\
&\quad=\frac{p_1}{p_2}\kappa^{-\frac{q}{p_1}}\int_0^\fz
\lz^{\frac{q}{p_2}-1}
\omega\lf(\lf\{x\in Q_0:\ \cm_{\gz,2B_0}\lf(|\wz{f}|^{p_2}\r)(x)>\lz\r\}\r)\,d\lz\\
&\quad=\frac{p_1}{q}\kappa^{-\frac{q}{p_1}}\int_{Q_0}
\lf[\cm_{\gz,2B_0}\lf(|\wz{f}|^{p_2}\r)\r]^{\frac{q}{p_2}}\omega\,dx
\ls\frac{p_1}{q}\kappa^{-\frac{q}{p_1}}\lf[\int_{2B_0}
|\wz{f}|^{q_0}
\omega^{\frac{q_0}{q}}\,dx\r]^{\frac{q}{q_0}},
\end{align*}
which, combined with a change of variables and \eqref{3.9}, implies that
\begin{align}\label{3.10}
&\bz^{-\frac{q}{p_1}}\int_{\bz\lz_0}^{\bz T}\lz^{\frac{q}{p_1}-1}\omega(E(\lz))\,d\lz\\
&\quad\le
\dz^{\frac{s-1}s}\int_{\lz_0}^T\lz^{\frac{q}{p_1}-1}\omega(E(\lz))\,d\lz
+C_{(\kappa,\,[\omega]_{A_{q(1-\gz)/p_2}(\rn)})}
\lf[\int_{2B_0}|\wz{f}|^{q_0}
\omega^{\frac{q_0}{q}}\,dx\r]^{\frac{q}{q_0}},\nonumber
\end{align}
where $C_{(\kappa,\,[\omega]_{A_{q(1-\gz)/p_2}(\rn)})}$ is a positive
constant depending only on $\kappa$ and $[\omega]_{A_{q(1-\gz)/p_2}(\rn)}$. Let
$$\mathrm{I}:=C_{(\kappa,\,[\omega]_{A_{q(1-\gz)/p_2}(\rn)})}
\lf[\int_{2B_0}|\wz{f}|^{q_0}
\omega^{\frac{q_0}{q}}\,dx\r]^{\frac{q}{q_0}}.
$$
By \eqref{3.10} and the fact that $\bz\in(1,\fz)$,
we conclude that, for any $T\in(\lz_0,\fz)$,
\begin{align*}
&\bz^{-\frac{q}{p_1}}\int_{0}^{T}\lz^{\frac{q}{p_1}-1}\omega(E(\lz))\,d\lz\\
&\quad\le\bz^{-\frac{q}{p_1}}\int_{0}^{\bz T}\lz^{\frac{q}{p_1}-1}\omega(E(\lz))\,d\lz\\
&\quad\le\bz^{-\frac{q}{p_1}}\int_{0}^{\bz\lz_0}\lz^{\frac{q}{p_1}-1}\omega(E(\lz))\,d\lz+
\dz^{\frac{s-1}s}\int_{\lz_0}^T\lz^{\frac{q}{p_1}-1}\omega(E(\lz))\,d\lz
+\mathrm{I}\\
&\quad\le\dz^{\frac{s-1}s}\int_{0}^T\lz^{\frac{q}{p_1}-1}\omega(E(\lz))\,d\lz
+C\omega(Q_0)\lz_0^{\frac{q}{p_1}}+\mathrm{I},
\end{align*}
where $C$ is a positive constant independent of $T$,
which further implies that, for any $T\in(\lz_0,\fz)$,
\begin{align}\label{3.11}
\bz^{-\frac{q}{p_1}}\lf[1-\dz^{\frac{s-1}s}\bz^{\frac{q}{p_1}}\r]
\int_{0}^{T}\lz^{\frac{q}{p_1}-1}\omega(E(\lz))\,d\lz\le
C\omega(Q_0)\lz_0^{\frac{q}{p_1}}+\mathrm{I}.
\end{align}
From $\bz=[2\dz^{(s-1)/s}]^{-p_1/q}$, Lemma \ref{l3.2}
and letting $T\to\fz$ in \eqref{3.11}, we deduce that
$$\int_{Q_0}|\wz{G}|^q\omega\,dx\le C\omega(Q_0)\lz_0^{\frac{q}{p_1}}
+C_{(\kappa,\,[\omega]_{A_{q(1-\gz)/p_2}(\rn)})}
\lf[\int_{2B_0}|\wz{f}|^{q_0}
\omega^{\frac{q_0}{q}}\,dx\r]^{\frac{q}{q_0}},
$$
which, together with $\lz_0=C_3\fint_{2B_0}|\wz{G}|^{p_1}\,dx$,
$|Q_0|\sim|B_0|$ and Lemma \ref{l3.1}(v), further implies that \eqref{3.6} holds true.

Now we prove \eqref{3.7}. By \eqref{3.4} and Lemma \ref{l3.2}, we find that
there exists a positive constant $C_4$, depending only on $n$, such that,
for any $\lz\in(0,\fz)$,
\begin{equation}\label{3.12}
|E(\lz)|\le\frac{C_4}{\lz}\int_{2B_0}|\wz{G}|^{p_1}\,dx.
\end{equation}
Let $C_3:=C_4|2B_0||Q_0|^{-1}\dz^{-1}$. Then, from \eqref{3.8} and \eqref{3.12},
it follows that, for any $\lz\in(\lz_0,\fz)$, $|E(\lz)|<\dz|Q_0|$.

From now on, we fix $\lz\in(\lz_0,\fz)$. By the fact that $E(\lz)$ is open in $Q_0$
and the Calder\'on--Zygmund decomposition (see, for instance,
\cite[p.\,17]{St93} and \cite[Lemma 4.2.2]{sh18}),
we conclude that there exist an index set $I$ and a sequence $\{Q_k\}_{k\in I}$
of disjoint and maximal dyadic subcubes of $Q_0$ such that
\begin{itemize}
\item[(a)] for any $k\in I$, $Q_k\subset E(\lz)$;
\vspace{-0.5em}
\item[(b)]for any $k\in I$, $\wz{Q}_k\subset Q_0$, but $\wz{Q}_k\nsubseteq E(\lz)$,
where $\wz{Q}_k$ denotes the dyadic parent cube of $Q_k$;
\vspace{-0.5em}
\item[(c)] $|E(\lz)\backslash(\bigcup_{k\in I} Q_k)|=0$.
\end{itemize}
For any $k\in I$, denote by $B_k$ the ball with the same center and the same
diameter as $Q_k$, namely, its radius $r_{B_k}:=\sqrt{n}\ell(Q_k)/2$, where $\ell(Q_k)$ denotes the side length of $Q_k$.
From the definition of $E(\lz)$ and the above property (b) of $\{Q_k\}_{k\in I}$,
it follows that there exists a positive constant $C_5$, depending only on $n$, such that,
for any ball $\wz{B}\subset2B_0$ satisfying $\wz{B}\cap B_k\neq\emptyset$ and
$r_{\wz{B}}\ge r_{B_k}$,
\begin{equation}\label{3.13}
\fint_{\wz{B}}|\wz{G}|^{p_1}\,dx\le C_5\lz,
\end{equation}
which, combined with the definition of $\cm_{2B_0}$, further implies that,
for any $x\in Q_k$ with $k\in I$,
\begin{equation}\label{3.14}
\cm_{2B_0}\lf(|\wz{G}|^{p_1}\r)(x)\le\max\lf\{\cm_{2B_k}\lf(|\wz{G}|^{p_1}\r)(x), C_5\lz\r\}.
\end{equation}
Let $B:=B(x_0,r)$ be a ball satisfying $B\subset2B_0$ and $|B|\le\bz_1|\boz|/5^n$,
where $x_0\in\rn$ and $r\in(0,\fz)$.
If $2B\subset\boz$, let $\wz{G}_B:=G_B$ and $\wz{R}_B:=R_B$; if $2B\cap\boz=\emptyset$,
let $\wz{G}_B:=0$ and $\wz{R}_B:=0$; if $2B\cap\boz\neq\emptyset$
and $2B\cap\boz^\complement\neq\emptyset$, let $\wz{G}_B:=G_{B_1}\mathbf{1}_{\boz}$
and $\wz{R}_B:=R_{B_1}\mathbf{1}_{\boz}$, where $B_1:=B(y_0,4r)$ and $y_0\in\partial\boz$
satisfy $2B\subset B(y_0,4r)$. Then it is easy to see that, for any ball $B$ satisfying
$2B\subset2B_0$ and $|B|\le\bz_1|\boz|/5^n$,
$|\wz{G}|\le|\wz{G}_B|+|\wz{R}_B|$ on $2B$, \eqref{3.1} and \eqref{3.2}
also hold true if $G$, $f$, $G_B$, $R_B$ and $\boz$ are replaced, respectively, by
$\wz{G}$, $\wz{f}$, $\wz{G}_B$, $\wz{R}_B$ and $2B_0$ (see, for instance,
\cite[pp.\,2435-2436]{g12} for more details).

Take $\dz\in(0,1/3)$ to be sufficiently small such that $\bz> C_5$.
Then, by \eqref{3.14}, the fact that
$$|\wz{G}|^{p_1}\le 2^{p_1-1}\lf(|\wz{G}_{B_k}|^{p_1}+|\wz{R}_{B_k}|^{p_1}\r)$$
on $2B_k$, \eqref{3.4} and Lemma \ref{l3.2}, we conclude that, for any $k\in I$,
\begin{align}\label{3.15}
|E(\bz\lz)\cap Q_k|&\le\lf|\lf\{x\in Q_k:\ \cm_{2B_k}\lf(|\wz{G}|^{p_1}\r)(x)>\bz\lz\r\}\r|\\
&\le\lf|\lf\{x\in Q_k:\ \cm_{2B_k}\lf(|\wz{G}_{B_k}|^{p_1}\r)(x)>\frac{\bz\lz}{2^{p_1}}\r\}\r|\nonumber\\
&\qquad+\lf|\lf\{x\in Q_k:\ \cm_{2B_k}\lf(|\wz{R}_{B_k}|^{p_1}\r)(x)>\frac{\bz\lz}{2^{p_1}}\r\}\r|\nonumber\\
&\le\frac{2^{p_1}C_4}{\bz\lz}\int_{2B_k}|\wz{G}_{B_k}|^{p_1}\,dx
+\frac{C_{(n,\,p_1,\,p_3)}}{(\bz\lz)^{p_3/p_1}}
\int_{2B_k}|\wz{R}_{B_k}|^{p_3}\,dx.\nonumber
\end{align}
Now we claim that, for any $k\in I$, if
\begin{equation}\label{3.16}
Q_k\cap\lf\{x\in Q_0:\ \cm_{\gz,2B_0}\lf(|\wz{f}|^{p_2}\r)(x)\le(\kappa\lz)^b\r\}\neq\emptyset,
\end{equation}
then
\begin{equation}\label{3.17}
|E(\bz\lz)\cap Q_k|\le\wz{C}^{-\frac s{s-1}}\dz|Q_k|,
\end{equation}
where $\wz{C}$ is as in Lemma \ref{l3.1}(vi).
Indeed, for such a $k$, from \eqref{3.2}, \eqref{3.13} and \eqref{3.16}, it follows that
\begin{equation}\label{3.18}
\fint_{2B_k}|\wz{G}_{B_k}|^{p_1}\,dx\le C_{(p_1)}\lf(C_2^{p_1}\kappa+\uc^{p_1} C_5\r)\lz.
\end{equation}
Moreover, by \eqref{3.1}, \eqref{3.13} and \eqref{3.16}, we find that
\begin{equation}\label{3.19}
\fint_{2B_k}|\wz{R}_{B_k}|^{p_3}\,dx\le
C_1^{p_3}\lf[(C_5\lz)^{\frac1{p_1}}+(\kappa\lz)^{\frac1{p_1}}\r]^{p_3}
=\lf[C_1\lf(C_5^{\frac1{p_1}}+\kappa^{\frac1{p_1}}\r)\r]^{p_3}\lz^{\frac{p_3}{p_1}}.
\end{equation}
From \eqref{3.15}, \eqref{3.18}, \eqref{3.19} and $\bz=[2\dz^{(s-1)/s}]^{-p_1/q}$,
it follows that, for any $k\in I$ satisfying \eqref{3.16},
\begin{align}\label{3.20}
&|E(\bz\lz)\cap Q_k|\\
&\quad\le\lf\{\frac{2^{p_1}C_4C_{(p_1)}(C_2^{p_1}\kappa+\uc^{p_1} C_5)}{\bz}+
\frac{C_{(n,p_1,p_3)}[C_1(C_5^{1/p_1}+\kappa^{1/p_1})]^{p_3}}
{\bz^{p_3/p_1}}\r\}\frac{|2B_k|}{|Q_k|}|Q_k|\nonumber\\
&\quad\le C_{(n,p_1,p_3)}\lf[\lf(C_2^{p_1}\kappa+\uc^{p_1} C_5\r)\dz^{\frac{(s-1)p_1}{sq}-1}+
\lf(C_5^{1/p_1}+\kappa^{1/p_1}\r)^{p_3}\dz^{\frac{(s-1)p_3}{sq}-1}\r]\dz|Q_k|.\nonumber
\end{align}
By the assumption $s>(\frac{p_3}{q})'$, we know that $\frac{(s-1)p_3}{sq}>1$.
Now, we choose $\dz\in(0,1/3)$ sufficiently small such that
\begin{equation}\label{3.21}
C_{(n,p_1,p_3)}\lf(C_5^{\frac1{p_1}}+\kappa^{\frac1{p_1}}\r)^{p_3}\dz^{\frac{(s-1)p_3}{sq}-1}\le\frac12
\wz{C}^{-\frac s{s-1}},
\end{equation}
where $\wz{C}$ is as in Lemma \ref{l3.1}(vi).
Then, for such a fixed $\dz$, we take $\kappa\in(0,1)$ and $\uc_0\in(0,1)$ small
enough such that, for any $\uc\in[0,\uc_0)$,
$$C_{(n,p_1,p_3)}\lf(C_2^{p_1}\kappa+\uc^{p_1} C_5\r)\dz^{\frac{(s-1)p_1}{sq}-1}\le
\frac12\wz{C}^{-\frac{s}{s-1}},$$
which, together with \eqref{3.20} and \eqref{3.21}, further implies that,
for any $k\in I$ such that \eqref{3.16} holds true,
$$|E(\bz\lz)\cap Q_k|\le\wz{C}^{-\frac s{s-1}}\dz|Q_k|.
$$
Thus, for any $k\in I$ satisfying \eqref{3.16}, \eqref{3.17} holds true.

Now we prove \eqref{3.7} via using \eqref{3.17}. From $\omega\in RH_s(\rn)$,
Lemma \ref{l3.1}(vi) and \eqref{3.17}, we deduce that, for any $k$ satisfying \eqref{3.16},
\begin{equation}\label{3.22}
\omega(E(\bz\lz)\cap Q_k)\le\dz^{\frac{s-1}s}\omega(Q_k).
\end{equation}
Let $I_1:=\{k\in I:\ \eqref{3.16}\ \text{holds true for}\ k\}$.
Then, by \eqref{3.22} and the term $(c)$, we conclude that, for any $\lz\in(\lz_0,\fz)$,
\begin{align*}
\omega(E(\bz\lz))&=\omega\lf(E(\bz\lz)\bigcap\lf[\bigcup_{k\in I} Q_k\r]\r)\\
&\le\sum_{k\in I_1}\omega\lf(E(\bz\lz)\cap Q_k\r)+\omega\lf(\lf\{x\in Q_0:\
\cm_{\gz,2B_0}\lf(|\wz{f}|^{p_2}\r)(x)>(\kappa\lz)^b\r\}\r)\\
&\le\dz^{\frac{s-1}s}\sum_{k\in I_1}\omega\lf(Q_k\r)+\omega\lf(\lf\{x\in Q_0:\
\cm_{\gz,2B_0}\lf(|\wz{f}|^{p_2}\r)(x)>(\kappa\lz)^b\r\}\r)\\
&\le\dz^{\frac{s-1}s}\omega(E(\lz))+\omega\lf(\lf\{x\in Q_0:\
\cm_{\gz,2B_0}\lf(|\wz{f}|^{p_2}\r)(x)>(\kappa\lz)^b\r\}\r).
\end{align*}
This finishes the proof of \eqref{3.7} and hence of Theorem \ref{t3.1}.
\end{proof}

\section{Proofs of Theorems \ref{t1.1} and \ref{t1.2}\label{s4}}

\hskip\parindent In this section, we prove Theorems \ref{t1.1} and \ref{t1.2} by using Theorem \ref{t3.1}.
We begin with the following auxiliary conclusion.

\begin{lemma}\label{l4.1}
Assume that $n\ge2$, $\boz\subset\rn$ is a bounded Lipschitz domain,
$\az\in L^\fz(\partial\boz)$ is as in \eqref{1.4},
$p\in(1,\fz)$, $1/p+1/p'=1$ and $F\in L^{(p')_\ast}(\boz)$, where $(p')_\ast$ is as in \eqref{1.5}
with $p$ replaced by $p'$.
Let $v\in W^{1,p'}(\boz)$ be the weak solution of the Robin problem
\begin{equation}\label{4.1}
\begin{cases}
-\mathrm{div}(A\nabla v)=F\ \ & \text{in}\ \ \boz,\\
\dfrac{\partial v}{\partial\boldsymbol{\nu}}+\az v=0\ \ & \text{on}\ \ \partial\boz.
\end{cases}
\end{equation}
Assume further that the weak solution $u\in W^{1,p}(\boz)$ of
the Robin problem
\begin{equation}\label{4.2}
\begin{cases}
-\mathrm{div}(A\nabla u)=\mathrm{div}(\mathbf{f})\ \ & \text{in}\ \ \boz,\\
\dfrac{\partial u}{\partial\boldsymbol{\nu}}+\az u
=\mathbf{f}\cdot\boldsymbol{\nu}\ \ & \text{on}\ \ \partial\boz
\end{cases}
\end{equation}
satisfies the estimate
\begin{equation}\label{4.3}
\|u\|_{W^{1,p}(\boz)}\le
C\|\mathbf{f}\|_{L^p(\boz;\rn)},
\end{equation}
where $C$ is a positive constant independent of $u$
and $\mathbf{f}$. Then there exists a positive constant $C$,
independent of $v$ and $F$, such that
$$\|v\|_{W^{1,p'}(\boz)}\le C\lf\|F\r\|_{L^{(p')_\ast}(\boz)}.$$
\end{lemma}

\begin{proof}
By the assumptions \eqref{4.1} and \eqref{4.2}, we conclude that
$$\int_{\boz}A\nabla u\cdot\nabla v\,dx+\int_{\partial\boz}\az uv\,d\sz(x)
=-\int_{\boz}\mathbf{f}\cdot\nabla v\,dx=
\int_{\boz} Fu\,dx,
$$
which, combined with \eqref{4.3}, the Sobolev inequality
(see, for instance, \cite[Thoerem 7.26]{gt01} and
\cite[Section 2.3.4, Theorem 3.4]{n12}) and the H\"older inequality,
further implies that
\begin{align}\label{4.4}
\|\nabla v\|_{L^{p'}(\boz;\rn)}&=\sup_{\|\mathbf{f}\|_{L^p(\boz;\rn)}\le1}
\lf|\int_\boz\nabla v\cdot\mathbf{f}\,dx\r|=\sup_{\|\mathbf{f}\|_{L^p(\boz;\rn)}\le1}
\lf|\int_\boz Fu\,dx\r|\\
&\le\sup_{\|\mathbf{f}\|_{L^p(\boz;\rn)}\le1}\|F\|_{L^{(p')_\ast}(\boz)}
\|u\|_{L^{q}(\boz)}
\ls\sup_{\|\mathbf{f}\|_{L^p(\boz;\rn)}\le1}\|F\|_{L^{(p')_\ast}(\boz)}
\|u\|_{W^{1,p}(\boz)}\nonumber\\
&\ls\sup_{\|\mathbf{f}\|_{L^p(\boz;\rn)}\le1}
\|\mathbf{f}\|_{L^p(\boz;\rn)}\|F\|_{L^{(p')_\ast}(\boz)}
\ls\|F\|_{L^{(p')_\ast}(\boz)},\nonumber
\end{align}
where $q\in(1,\fz)$ is given by $1/q+1/(p')_\ast=1$.
Moreover, denote by $R(\cdot,\cdot)$ the \emph{Robin-Green function} of the boundary
value problem \eqref{4.1}. Then, for any $x\in\boz$,
\begin{equation}\label{4.5}
v(x)=\int_{\boz}R(x,y)F(y)\,dy.
\end{equation}
Furthermore, it is well known that, for any $x,\,y\in\boz$,
$|R(x,y)|\ls|x-y|^{2-n}$ when $n\ge3$, and $|R(x,y)|\ls\ln(1+|x-y|^{-1})$
when $n=2$ (see, for instance, \cite[p.\,3207]{ck14}),
which, together with \eqref{4.5} and the fact that $\boz$ is bounded,
further implies that, for any $x\in\boz$,
$$|v(x)|\ls\int_{\boz}|x-y|^{1-n}|F(y)|\,dy.
$$
From this and the well-known results on the boundedness of the fractional integral
(see, for instance, \cite[Theorem 1.2.3]{g14a}),
we deduce that $\|v\|_{L^{p'}(\boz)}\ls\|F\|_{L^{(p')_\ast}(\boz)}$,
which, combined with \eqref{4.4}, then completes the proof Lemma \ref{l4.1}.
\end{proof}

\begin{lemma}\label{l4.2}
Let $n\ge2$, $\boz\subset\rn$ be a bounded Lipschitz domain, $\az\in L^\fz(\partial\boz)$ be as in \eqref{1.4},
$p\in(1,\fz)$, $1/p+1/p'=1$, $\mathbf{g}\in L^{p'}(\boz;\rn)$ and
$p_\ast\in(1,\fz)$ be as in \eqref{1.5}.
Assume that $u\in W^{1,p'}(\boz)$ is the weak solution of the Robin problem
\eqref{4.2} with $\mathbf{f}$ replaced by $\mathbf{g}$.
If the weak solution $v\in W^{1,p}(\boz)$ of the Robin problem $(R)_p$
with $\mathbf{f}\in L^p(\boz;\rn)$, $F\in L^{p_\ast}(\boz)$
and $g\equiv0$ satisfies the estimate \eqref{1.13}, then
$\|u\|_{W^{1,p'}(\boz)}\le C\|\mathbf{g}\|_{L^{p'}(\boz;\rn)},$
where $C$ is a positive constant independent of $v$ and $\mathbf{g}$.
\end{lemma}

\begin{proof}
Let $w\in W^{1,p}(\boz)$ be the weak solution of the Robin
problem \eqref{4.2} with $\mathbf{f}\in L^p(\boz;\rn)$.
Then
$$\int_\boz\mathbf{g}\cdot \nabla w\,dx=-\int_\boz A\nabla u\cdot\nabla w\,dx=
\int_\boz \mathbf{f}\cdot\nabla u\,dx,
$$
which, together with \eqref{4.3} and the H\"older inequality, further implies that
\begin{align}\label{4.6}
\|\nabla u\|_{L^{p'}(\boz;\rn)}&=\sup_{\|\mathbf{f}\|_{L^{p}(\boz;\rn)}\le1}
\lf|\int_\boz\mathbf{f}\cdot\nabla u\,dx\r|
=\sup_{\|\mathbf{f}\|_{L^{p}(\boz;\rn)}\le1}
\lf|\int_\boz\mathbf{g}\cdot \nabla w\,dx\r|\\
&\le\sup_{\|\mathbf{f}\|_{L^{p}(\boz;\rn)}\le1}\|\mathbf{g}\|_{L^{p'}(\boz;\rn)}
\|\nabla w\|_{L^{p}(\boz;\rn)}
\nonumber\\
&\ls\sup_{\|\mathbf{f}\|_{L^{p}(\boz;\rn)}\le1}\|\mathbf{g}\|_{L^{p'}(\boz;\rn)}
\|\mathbf{f}\|_{L^{p}(\boz;\rn)}\ls\|\mathbf{g}\|_{L^{p'}(\boz;\rn)}.\nonumber
\end{align}
Let $h$ be the weak solution of the Robin problem
\begin{equation*}
\begin{cases}
-\mathrm{div}(A\nabla h)=F\ \ & \text{in}\ \ \boz,\\
\dfrac{\partial h}{\partial\boldsymbol{\nu}}+\az h=0\ \ & \text{on}\ \ \partial\boz
\end{cases}
\end{equation*}
with $F\in L^{p_\ast}(\boz)$. Then, by the assumption condition in Lemma \ref{l4.2},
we conclude that
$$\|h\|_{W^{1,p}(\boz)}\ls\|F\|_{L^{p_\ast}(\boz)},
$$
which, combined with
$$\int_\boz uF\,dx=\int_\boz A\nabla u\cdot\nabla h\,dx+\int_{\partial\boz}
\az uh\,d\sz(x)=-\int_{\boz}\mathbf{g}\cdot\nabla h\,dx,
$$
further implies that
\begin{align*}
\|u\|_{L^{q}(\boz)}&=\sup_{\|F\|_{L^{p_\ast}(\boz)}\le1}
\lf|\int_\boz uF\,dx\r|
=\sup_{\|F\|_{L^{p_\ast}(\boz)}\le1}
\lf|\int_\boz\mathbf{g}\cdot \nabla h\,dx\r|\nonumber\\
&\le\sup_{\|F\|_{L^{p_\ast}(\boz)}\le1}\|\mathbf{g}\|_{L^{p'}(\boz;\rn)}
\|\nabla h\|_{L^{p}(\boz;\rn)}\nonumber\\
&\ls\sup_{\|F\|_{L^{p_\ast}(\boz)}\le1}\|\mathbf{g}\|_{L^{p'}(\boz;\rn)}
\|F\|_{L^{p_\ast}(\boz)}\ls\|\mathbf{g}\|_{L^{p'}(\boz;\rn)},
\end{align*}
where $q\in(1,\fz)$ is given by $1/q+1/p_\ast=1$.
From this, \eqref{4.6} and the H\"older inequality, it follows that
$$\|u\|_{W^{1,p'}(\boz)}\ls\|u\|_{L^{1}(\boz)}+\|\nabla u\|_{L^{p'}(\boz;\rn)}
\ls\|u\|_{L^{q}(\boz)}+\|\nabla u\|_{L^{p'}(\boz;\rn)}
\ls\|\mathbf{g}\|_{L^{p'}(\boz;\rn)}.
$$
This finishes the proof of Lemma \ref{l4.2}.
\end{proof}

By Lemmas \ref{l4.1} and \ref{l4.2}, we have the following corollary.

\begin{corollary}\label{c4.1}
Let $n\ge2$, $\boz\subset\rn$ be a bounded Lipschitz domain, $\az\in L^\fz(\partial\boz)$ be as in \eqref{1.4},
$p\in(1,\fz)$, and $p_\ast\in(1,\fz)$ be as in \eqref{1.5}.
Assume that the weak solution $u\in W^{1,p}(\boz)$ of the Robin problem $(R)_p$,
with $\mathbf{f}\in L^p(\boz;\rn)$, $F\in L^{p_\ast}(\boz)$ and $g\equiv0$,
satisfies the estimate \eqref{1.13}.
Then the weak solution $v$ of the Robin problem
\begin{equation*}
\begin{cases}
-\mathrm{div}(A\nabla v)=\mathrm{div}(\mathbf{g})+H\ \ & \text{in}\ \ \boz,\\
\dfrac{\partial v}{\partial\boldsymbol{\nu}}+\az v
=\mathbf{g}\cdot\boldsymbol{\nu}+h\ \ & \text{on}\ \ \partial\boz,
\end{cases}
\end{equation*}
with $\mathbf{g}\in L^{p}(\boz;\rn)$, $H\in L^{p_\ast}(\boz)$
and $h\in W^{-1/p,p}(\partial\boz)$, satisfies
$$\|v\|_{W^{1,p}(\boz)}\le C\lf[\|\mathbf{g}\|_{L^{p}(\boz;\rn)}
+\|H\|_{L^{p_\ast}(\boz)}+\|h\|_{W^{-1/p,p}(\partial\boz)}\r],$$
where $C$ is a positive constant independent of $v$, $\mathbf{g}$, $H$
and $h$.
\end{corollary}

\begin{proof}
Let $u_1$ be the weak solution of the Robin problem
\begin{equation*}
\begin{cases}
-\mathrm{div}(A\nabla u_1)=\mathrm{div}(\mathbf{f}_1)\ \ & \text{in}\ \ \boz,\\
\dfrac{\partial u_1}{\partial\boldsymbol{\nu}}+\az u_1
=\mathbf{f}_1\cdot\boldsymbol{\nu}\ \ & \text{on}\ \ \partial\boz
\end{cases}
\end{equation*}
with $\mathbf{f}_1\in L^{p'}(\boz;\rn)$.
Then, by the assumption conditions in Corollary \ref{c4.1} and Lemma \ref{l4.2},
we conclude that
\begin{equation}\label{4.7}
\|u_1\|_{W^{1,p'}(\boz)}\ls\|\mathbf{f}_1\|_{L^{p'}(\boz;\rn)}.
\end{equation}
Let $w_1$ be the weak solution of the Robin problem
\begin{equation*}
\begin{cases}
-\mathrm{div}(A\nabla w_1)=H\ \ & \text{in}\ \ \boz,\\
\dfrac{\partial w_1}{\partial\boldsymbol{\nu}}+\az w_1
=h\ \ & \text{on}\ \ \partial\boz.
\end{cases}
\end{equation*}
Furthermore, assume that $w_2$ is the weak solution of the Robin problem
\begin{equation*}
\begin{cases}
-\mathrm{div}(A\nabla w_2)=\mathrm{div}(\mathbf{g})\ \ & \text{in}\ \ \boz,\\
\dfrac{\partial w_2}{\partial\boldsymbol{\nu}}+\az w_2
=\mathbf{g}\cdot\boldsymbol{\nu}\ \ & \text{on}\ \ \partial\boz.
\end{cases}
\end{equation*}
Then $v=w_1+w_2$ and
\begin{equation}\label{4.8}
\|w_2\|_{W^{1,p}(\boz)}\ls\|\mathbf{g}\|_{L^p(\boz;\rn)}.
\end{equation}
Moreover,
$$-\int_\boz\mathbf{f}_1\cdot\nabla w_1\,dx=
\int_\boz A\nabla u_1\cdot\nabla w_1\,dx+\int_{\partial\boz}
\az u_1w_1\,d\sz(x)=\int_{\boz}Hu_1\,dx+\langle h,u_1\rangle_{\partial\boz},
$$
which, together with the H\"older inequality, the Sobolev inequality
and the trace inequality
$$\|u_1\|_{W^{1/p,p'}(\partial\boz)}\ls\|u_1\|_{W^{1,p'}(\boz)}$$
(see, for instance, \cite[Section 2.5.4, Theorem 5.5]{n12}), further implies that
\begin{align*}
\lf|\int_\boz\mathbf{f}_1\cdot\nabla w_1\,dx\r|&\le
\|H\|_{L^{p_\ast}(\boz)}\|u_1\|_{L^{(p_\ast)'}(\boz)}+\|h\|_{W^{-1/p,p}(\partial\boz)}
\|u_1\|_{W^{1/p,p'}(\partial\boz)}\\
&\ls\lf[\|H\|_{L^{p_\ast}(\boz)}+\|h\|_{W^{-1/p,p}(\partial\boz)}\r]
\|u_1\|_{W^{1,p'}(\boz)}.
\end{align*}
From this and \eqref{4.7}, it follows that
\begin{equation}\label{4.9}
\|\nabla w_1\|_{L^{p}(\boz;\rn)}\ls
\|H\|_{L^{p_\ast}(\boz)}+\|h\|_{W^{-1/p,p}(\partial\boz)}.
\end{equation}
Let $u_2$ be the weak solution of the Robin problem
\begin{equation*}
\begin{cases}
-\mathrm{div}(A\nabla u_2)=F\ \ & \text{in}\ \ \boz,\\
\dfrac{\partial u_2}{\partial\boldsymbol{\nu}}+\az u_2
=0\ \ & \text{on}\ \ \partial\boz,
\end{cases}
\end{equation*}
with $F\in L^{p_\ast}(\boz)$.
Then $\|u_2\|_{W^{1,p}(\boz)}\ls\|F\|_{L^{p_\ast}(\boz)}$.
Using this estimate, similarly to the estimation of \eqref{4.9},
we conclude that
$$\|w_1\|_{L^p(\boz)}\ls\|H\|_{L^{p_\ast}(\boz)}+\|h\|_{W^{-1/p,p}(\partial\boz)},
$$
which, combined with \eqref{4.9}, further implies that
\begin{equation}\label{4.10}
\|w_1\|_{W^{1,p}(\boz)}\ls
\|H\|_{L^{p_\ast}(\boz)}+\|h\|_{W^{-1/p,p}(\partial\boz)}.
\end{equation}
Thus, from \eqref{4.8}, \eqref{4.10} and $v=w_1+w_2$,
we deduce that
$$\|v\|_{W^{1,p}(\boz)}\ls\|\mathbf{g}\|_{L^p(\boz;\rn)}
+\|H\|_{L^{p_\ast}(\boz)}+\|h\|_{W^{-1/p,p}(\partial\boz)}.
$$
This finishes the proof of Corollary \ref{c4.1}.
\end{proof}

To show Theorem \ref{t1.2}, we also need the  following
Lemma \ref{l4.3}, whose proof is similar to that of \cite[Lemma 4.38]{bm16}
and we omit the details here.

\begin{lemma}\label{l4.3}
Let $n\ge2$, $\boz\subset\rn$ be a Lipschitz domain, $0<p_0<q\le\fz$ and
$r_0\in(0,\diam(\boz))$. Assume that $x\in\overline{\boz}$ and
the weak reverse H\"older inequality
\begin{equation*}
\lf[\fint_{B_\boz(x,r)}|g|^q\,dx\r]^{\frac1q}
\le C_6\lf[\fint_{B_\boz(x,2r)}
|g|^{p_0}\,dx\r]^{\frac1{p_0}}
\end{equation*}
holds true for a given measurable function $g$ on $\boz$ and any $r\in(0,r_0)$,
where $C_6$ is a positive constant, independent of $x$ and $r$,
which may depend on $g$.
Then, for any given $p\in(0,\fz]$, there exists a positive constant $C$,
depending only on $p$, $p_0$, $q$ and $C_6$, such that
\begin{equation*}
\lf[\fint_{B_\boz(x,r)}|g|^q\,dx\r]^{\frac1q}
\le C\lf[\fint_{B_\boz(x,2r)}
|g|^{p}\,dx\r]^{\frac1p}
\end{equation*}
holds true for any $r\in(0,r_0)$.
\end{lemma}

Now we prove Theorem \ref{t1.1} by using Theorem \ref{t3.1}
and Lemmas \ref{l4.1}, \ref{l4.2} and \ref{l4.3}.

\begin{proof}[Proof of Theorem \ref{t1.1}]
We first show that (i) implies (ii).
Assume that (i) holds true. Let $r_0\in(0,\diam(\boz))$
be a constant and $B:=B(x_0,r)$ as in (ii),
namely, $r\in(0,r_0/4)$ and either $x_0\in\partial\boz$ or $B(x_0,2r)\subset\boz$.
Let $\eta\in C^\fz_{\mathrm{c}}(\rn)$ satisfy that $0\le\eta\le1$, $\eta\equiv1$ on $B$,
$\supp(\eta)\subset2B$, and $|\nabla\eta|\ls r^{-1}$. Moreover,
assume that $-\mathrm{div}(A\nabla v)=0$
in $2B\cap\boz$ and $\frac{\partial v}{\partial\boldsymbol{\nu}}+\az v=0$
on $B(x_0,2r)\cap\partial\boz$ when $x_0\in\partial\boz$. Then
\begin{equation}\label{4.11}
\begin{cases}
-\mathrm{div}(A\nabla([v-\overline{v}]\eta))=
-\mathrm{div}([v-\overline{v}]A\nabla\eta)-A\nabla v\cdot\nabla\eta\ \ & \text{in}\ \ \boz,\\
\dfrac{\partial([v-\overline{v}]\eta)}{\partial\boldsymbol{\nu}}+\az(v-\overline{v})\eta
=(v-\overline{v})A\nabla\eta\cdot\boldsymbol{\nu}-\az\overline{v}\eta
\ & \text{on}\ \ \partial\boz,
\end{cases}
\end{equation}
where $\overline{v}:=\fint_{2B\cap\boz}v\,dy$.
Indeed, by the assumption that $-\mathrm{div}(A\nabla v)=0$
in $2B\cap\boz$ and $\frac{\partial v}{\partial\boldsymbol{\nu}}+\az v=0$
on $B(x_0,2r)\cap\partial\boz$, we conclude that, for any $\phi \in C^\fz(\rn)$,
\begin{align*}
&\int_\boz A\nabla((v-\overline{v})\eta)\cdot\nabla\phi\,dx
+\int_{\partial\boz}\az(v-\overline{v})\eta\phi\,d\sz(x)\nonumber\\
&\quad=\int_\boz
A(v-\overline{v})\nabla\eta\cdot\nabla\phi\,dx
+\int_{\boz}\eta A\nabla(v-\overline{v})\cdot\nabla\phi\,dx
+\int_{\partial\boz}\az(v-\overline{v})\eta\phi\,d\sz(x)\nonumber\\
&\quad=\int_{\boz}A(v-\overline{v})\nabla\eta\cdot\nabla\phi\,dx
+\int_\boz A\nabla v\cdot\nabla(\eta\phi)\,dx-
\int_\boz \phi A\nabla v\cdot\nabla\eta\,dx
+\int_{\partial\boz}\az(v-\overline{v})\eta\phi\,d\sz(x)\nonumber\\
&\quad=\int_\boz A(v-\overline{v})\nabla\eta\cdot\nabla\phi\,dx
-\int_{\boz}\phi A\nabla v\cdot\nabla\eta\,dx
-\int_{\partial\boz}\az\overline{v}\eta\phi\,d\sz(x),
\end{align*}
which implies that \eqref{4.11} holds true.
From (i), \eqref{4.11},
Corollary \ref{c4.1} and \eqref{1.3}, we deduce that
\begin{align}\label{4.12}
\lf\|(v-\overline{v})\eta\r\|_{W^{1,p}(\boz)}
&\ls\lf\|A(v-\overline{v})\nabla\eta\r\|_{L^p(\boz;\rn)}+
\lf\|A\nabla v\cdot\nabla\eta\r\|_{L^{p_\ast}(\boz)}
+\|\az\overline{v}\eta\|_{W^{-1/p,p}(\partial\boz)}.
\end{align}
By the facts that $\supp(\eta)\subset2B$ and $|\nabla\eta|\ls r^{-1}$,
and the Sobolev inequality, we conclude that
\begin{align}\label{4.13}
\lf\|A(v-\overline{v})\nabla\eta\r\|_{L^p(\boz;\rn)}
\ls\lf\|(v-\overline{v})\nabla\eta\r\|_{L^p(\boz;\rn)}
\ls r^{-1}\lf\|v-\overline{v}\r\|_{L^p(2B\cap\boz)}\ls
r^{-1}\lf\|\nabla v\r\|_{L^{p_\ast}(2B\cap\boz;\rn)}.
\end{align}
Furthermore, it is easy to see that
\begin{align}\label{4.14}
\lf\|A\nabla v\cdot\nabla\eta\r\|_{L^{p_\ast}(\boz)}\ls
r^{-1}\lf\|\nabla v\r\|_{L^{p_\ast}(2B\cap\boz;\rn)}.
\end{align}
Moreover, by the facts that
$$ W^{1/p,p'}(\partial\boz)\subset L^s(\partial\boz)
\quad\text{and}\quad \lf(W^{1/p,p'}(\partial\boz)\r)^\ast=W^{-1/p,p}(\partial\boz)
$$
(see, for instance, \cite[Section 2.5.8, Corollary 5.1]{n12}),
where $s:=\frac{(n-1)p'}{(n-1)-p'/p}$ and $(W^{1/p,p'}(\partial\boz))^\ast$
denotes the \emph{dual space} of $W^{1/p,p'}(\partial\boz)$, we find that
$$
L^{s'}(\partial\boz)\subset W^{-1/p,p}(\partial\boz),
$$
which, together with $s'=(n-1)p/n$ and $\supp(\eta)\subset2B$, implies that
\begin{align*}
\|\az\overline{v}\eta\|_{W^{-1/p,p}(\partial\boz)}\ls
|\overline{v}|\|\az\eta\|_{L^{\frac{(n-1)p}n}(\partial\boz)}\ls
|\overline{v}|\|\az\|_{L^{\fz}(\partial\boz)}r^{n/p}.
\end{align*}
From this, \eqref{4.12}, \eqref{4.13} and \eqref{4.14}, we deduce that
$$\lf\|(v-\overline{v})\eta\r\|_{W^{1,p}(\boz)}\ls
r^{-1}\lf\|\nabla v\r\|_{L^{p_\ast}(2B\cap\boz;\rn)}
+\|\az\|_{L^{\fz}(\partial\boz)}|\overline{v}|r^{n/p},
$$
which, combined with the facts that $\eta\equiv1$ on $B$ and
$\frac{1}{p_\ast}=\frac{1}{p}+\frac{1}{n}$, further implies that
\begin{align}\label{4.15}
\lf[\fint_{B\cap\boz}(|v|+|\nabla v|)^p\,dx\r]^{\frac1p}&\ls
\lf[1+\|\az\|_{L^\fz(\partial\boz)}\r]\lf[\lf(\fint_{2B\cap\boz}|\nabla v|^{p_\ast}\,dx\r)^{\frac1{p_\ast}}
+\fint_{2B\cap\boz}|v|\,dx\r]\\
&\ls\lf[1+\|\az\|_{L^\fz(\partial\boz)}\r]
\lf[\fint_{2B\cap\boz}(|v|+|\nabla v|)^{p_\ast}\,dx\r]^{\frac1{p_\ast}}.\nonumber
\end{align}
By \eqref{4.15} and Lemma \ref{l4.3}, we find that
\eqref{1.14} holds true. Thus, we show that (i) implies (ii).

Now we prove that (ii) implies (iii). Assume that (ii) holds true. Let $q\in[p_0,p]$,
$q_0\in[1,\frac{q}{p_0}]$, $r_0\in[(\frac{p}{q})',\fz]$
and $\omega\in A_{q_0}(\rn)\cap RH_{r_0}(\rn)$.
Let $\mathbf{f}\in L^q_\omega(\boz;\rn)$, $F\in L^{q_\ast}_{\omega^{q_\ast/q}}(\boz)$,
and $u_1,\,u_2\in W^{1,p_0}(\boz)$ be, respectively, the weak solutions of the Robin problems
\begin{equation}\label{4.16}
\begin{cases}
-\mathrm{div}(A\nabla u_1)=\mathrm{div}(\mathbf{f})\ \ &\text{in}\ \ \boz,\\
\dfrac{\partial u_1}{\partial\boldsymbol{\nu}}+\az u_1=\mathbf{f}\cdot
\boldsymbol{\nu} \ \ &\text{on}\ \ \partial\boz
\end{cases}
\end{equation}
and
\begin{equation}\label{4.17}
\begin{cases}
-\mathrm{div}(A\nabla u_2)=F\ \ &\text{in}\ \ \boz,\\
\dfrac{\partial u_2}{\partial\boldsymbol{\nu}}+\az u_2=0 \ \ &\text{on}\ \ \partial\boz.
\end{cases}
\end{equation}
Then $u=u_1+u_2$.

From Lemma \ref{l3.1}(vii), it follows that
$L^q_\omega(\boz)\subset L^{\frac q{q_0}}(\boz)\subset L^{p_0}(\boz)$,
which further implies that $\mathbf{f}\in L^{p_0}(\boz;\rn)$.
Let $B:=B(x_B,r_B)\subset\rn$ be a ball satisfying $r_B\in(0,r_0/4)$ and either $2B\subset\boz$ or $x_B\in\partial\boz$.
Take $\phi\in C^\fz_{\mathrm{c}}(\rn)$ such that $\phi\equiv1$ on $2B$, $0\le\phi\le1$
and $\supp(\phi)\subset 4B$. Let $v_1,\,w_1\in W^{1,p_0}(\boz)$ be, respectively, the
weak solutions of the Robin problems
\begin{equation}\label{4.18}
\begin{cases}
-\mathrm{div}(A\nabla v_1)=\mathrm{div}(\phi\mathbf{f})\ \ &\text{in}\ \ \boz,\\
\dfrac{\partial v_1}{\partial\boldsymbol{\nu}}+\az v_1=\phi\mathbf{f}\cdot
\boldsymbol{\nu} \ \ &\text{on}\ \ \partial\boz
\end{cases}
\end{equation}
and
\begin{equation}\label{4.19}
\begin{cases}
-\mathrm{div}(A\nabla w_1)=\mathrm{div}((1-\phi)\mathbf{f})\ \ &\text{in}\ \ \boz,\\
\dfrac{\partial w_1}{\partial\boldsymbol{\nu}}+\az w_1=(1-\phi)\mathbf{f}\cdot
\boldsymbol{\nu} \ \ &\text{on}\ \ \partial\boz.
\end{cases}
\end{equation}
Then $u_1=v_1+w_1$.

Let $G:=|u_1|+|\nabla u_1|$, $f:=|\mathbf{f}|$, $G_B:=|v_1|+|\nabla v_1|$
and $R_B:=|w_1|+|\nabla w_1|$. It is easy to see that $0\le G\le G_B+R_B$.
By the assumption \eqref{1.12}, we conclude that
$$\|v_1\|_{W^{1,p_0}(\boz)}\ls\|\phi\mathbf{f}\|_{L^{p_0}(\boz;\rn)},
$$
which further implies that
\begin{align}\label{4.20}
\lf(\fint_{2B_\boz}G_B^{p_0}\,dx\r)^{\frac1{p_0}}
&\ls\lf[\frac{1}{|2B_\boz|}\int_{\boz}(|v_1|+|\nabla v_1|)^{p_0}\,dx\r]^{\frac1{p_0}}\\
&\ls\lf(\frac{1}{|2B_\boz|}\int_{\boz}|\mathbf{f}\phi|^{p_0}\,dx\r)^{\frac1{p_0}}
\ls\lf(\fint_{4B_\boz}f^{p_0}\,dx\r)^{\frac1{p_0}}.\nonumber
\end{align}
Moreover, from \eqref{4.19} and (ii),
it follows that \eqref{1.14} holds true for $w_1$,
which, together with the self-improvement property of the weak reverse H\"older inequality
(see, for instance, \cite[pp.\,122-123]{g83}), further implies that
there exists an $\uc_0\in(0,\fz)$ such that the inequality \eqref{1.14}
holds true with $p$ replaced by $p+\uc_0$.
By this and Lemma \ref{l4.3}, we find that,
for any $t\in(0,p_0]$, the weak reverse H\"older inequality
\begin{equation}\label{4.21}
\lf[\fint_{B_\boz}(|w_1|+|\nabla w_1|)^{p+\uc_0}\,dx\r]^{\frac1{p+\uc_0}}
\ls\lf[\fint_{2B_\boz}
(|w_1|+|\nabla w_1|)^{t}\,dx\r]^{\frac1t}
\end{equation}
holds true. This, combined with \eqref{4.20}, further implies that
\begin{align}\label{4.22}
\lf[\fint_{2B_\boz}R_B^{p+\uc_0}\,dx
\r]^{\frac1{p+\uc_0}}
&\ls\lf[\fint_{4B_\boz}(|w_1|+|\nabla w_1|)^{p_0}\,dx\r]^{\frac1{p_0}}\\
&\ls\lf(\fint_{4B_\boz}G^{p_0}\,dx\r)^{\frac1{p_0}}
+\lf(\fint_{4B_\boz}f^{p_0}\,dx\r)^{\frac1{p_0}}.\nonumber
\end{align}
Moreover, from the assumption $\omega\in A_{q_0}(\rn)\cap RH_{r_0}(\rn)$
with $q_0\in[1,\frac{q}{p_0}]$ and $r_0\in[(\frac{p}{q})',\fz]$, we deduce that $\omega\in A_{q/p_0}(\rn)\cap RH_s(\rn)$
with $s\in((\frac{p+\uc_0}{q})',\fz]$.
By \eqref{4.20} and \eqref{4.22}, we find that \eqref{3.1} and \eqref{3.2}
hold true with $p_1=p_2:=p_0$, $p_3:=p+\uc_0$ and $\gz:=0$. Then, applying Theorem \ref{t3.1}
and the H\"older inequality, we conclude that
\begin{align}\label{4.23}
&\lf[\frac{1}{\omega(\boz)}\int_{\boz}(|u_1|+|\nabla u_1|)^q\omega\,dx\r]^{\frac1q}\\
&\quad\ls\lf[\fint_{\boz}(|u_1|+|\nabla u_1|)^{p_0}\,dx\r]^{\frac1{p_0}}+
\lf[\frac{1}{\omega(\boz)}\int_{\boz}|\mathbf{f}|^q\omega\,dx\r]^{\frac1q}\nonumber\\
&\quad\ls\lf(\fint_{\boz}|\mathbf{f}|^{p_0}\,dx\r)^{\frac1{p_0}}+
\lf[\frac{1}{\omega(\boz)}\int_{\boz}|\mathbf{f}|^q\omega\,dx\r]^{\frac1q}
\ls\lf[\frac{1}{\omega(\boz)}\int_{\boz}|\mathbf{f}|^q\omega\,dx\r]^{\frac1q}.\nonumber
\end{align}

Let $v_2,\,w_2\in W^{1,p_0}(\boz)$ be, respectively, the
weak solutions of the Robin problems
\begin{equation}\label{4.24}
\begin{cases}
-\mathrm{div}(A\nabla v_2)=F\phi\ \ &\text{in}\ \ \boz,\\
\dfrac{\partial v_2}{\partial\boldsymbol{\nu}}+\az v_2=0 \ \ &\text{on}\ \ \partial\boz
\end{cases}
\end{equation}
and
\begin{equation}\label{4.25}
\begin{cases}
-\mathrm{div}(A\nabla w_2)=F(1-\phi)\ \ &\text{in}\ \ \boz,\\
\dfrac{\partial w_2}{\partial\boldsymbol{\nu}}+\az w_2=0 \ \ &\text{on}\ \ \partial\boz.
\end{cases}
\end{equation}
Then $u_2=v_2+w_2$ and, by the assumption \eqref{1.12}, we find that
\begin{equation}\label{4.26}
\|v_2\|_{W^{1,p_0}(\boz)}\ls\|F\phi\|_{L^{(p_0)_\ast}(\boz)}.
\end{equation}
Let $G:=|u_2|+|\nabla u_2|$, $f:=|F|$, $G_B:=|v_2|+|\nabla v_2|$
and $R_B:=|w_2|+|\nabla w_2|$. From \eqref{4.26} and the facts
that $\phi\equiv1$ on $2B$, $0\le\phi\le1$
and $\supp(\phi)\subset 4B$, it follows that
\begin{align}\label{4.27}
\lf(\fint_{2B_\boz}G_B^{p_0}\,dx\r)^{\frac1{p_0}}&
\ls\lf[\frac{1}{|2B_\boz|}\int_{\boz}(|v_2|+|\nabla v_2|)^{p_0}\,dx\r]^{\frac1{p_0}}\\
&\ls\lf[|2B_\boz|^{1-\frac{(p_0)_\ast}{p_0}}\frac{1}{|2B_\boz|}\int_{\boz}
|F\phi|^{(p_0)_\ast}\,dx\r]^{\frac{1}{(p_0)_\ast}}\nonumber\\
&\ls\lf[|4B_\boz|^{1-\frac{(p_0)_\ast}{p_0}}\fint_{4B_\boz}
f^{(p_0)_\ast}\,dx\r]^{\frac{1}{(p_0)_\ast}}.\nonumber
\end{align}
Furthermore, by the fact that $\phi\equiv1$ on $2B$, \eqref{4.25}
and the assumption that (ii) holds true, we further conclude that \eqref{4.21}
also holds true with $w_1$ replaced by $w_2$, which, together with
\eqref{4.27}, implies that
\begin{align}\label{4.28}
\lf(\fint_{2B_\boz}R_B^{p+\uc_0}\,dx
\r)^{\frac1{p+\uc_0}}&\ls\lf[\fint_{4B_\boz}(|w_2|+|\nabla w_2|)^{p_0}\,dx\r]^{\frac1{p_0}}\\
&\ls\lf(\fint_{4B_\boz}G^{p_0}\,dx\r)^{\frac1{p_0}}
+\lf[|4B_\boz|^{1-\frac{(p_0)_\ast}{p_0}}\fint_{4B_\boz}
f^{(p_0)_\ast}\,dx\r]^{\frac{1}{(p_0)_\ast}}.\nonumber
\end{align}
From \eqref{4.27} and \eqref{4.28}, we deduce that \eqref{3.1} and \eqref{3.2}
hold true with $p_1:=p_0$, $p_2:=(p_0)_\ast$, $p_3:=p+\uc_0$ and $\gz:=1-\frac{(p_0)_\ast}{p_0}$,
which, combined with Theorem \ref{t3.1}, $1/q_\ast=1/q+1/n$ and the H\"older inequality, implies that
\begin{align}\label{4.29}
&\lf[\frac{1}{\omega(\boz)}\int_{\boz}(|u_2|+|\nabla u_2|)^q\omega\,dx\r]^{\frac1q}\\
&\quad\ls\lf[\fint_{\boz}(|u_2|+|\nabla u_2|)^{p_0}\,dx\r]^{\frac1{p_0}}+
\lf[\frac{1}{\omega(\boz)}\int_{\boz}|F|^{q_\ast}\omega^{\frac{q_\ast}{q}}
\,dx\r]^{\frac{1}{q_\ast}}[\omega(\boz)]^{\frac{1}{n}}\nonumber\\
&\quad\ls\lf[\frac{1}{\omega(\boz)}\int_{\boz}|F|^{q_\ast}\omega^{\frac{q_\ast}{q}}
\,dx\r]^{\frac{1}{q_\ast}}[\omega(\boz)]^{\frac{1}{n}}.\nonumber
\end{align}
Thus, by \eqref{4.23}, \eqref{4.29} and $u=u_1+u_2$, we conclude that
\eqref{1.15} holds true.
This show that (ii) implies (iii).

Finally, we prove that (iii) implies (i). By taking
$q:=p$ and $\omega\equiv1$ in (iii),
we know that (i) holds true. This finishes the proof of Theorem \ref{t1.1}.
\end{proof}

To show Theorem \ref{t1.2}, we need the following Lemma \ref{l4.4},
which was essentially obtained in \cite{gz20}. Indeed, if $A\in\mathrm{VMO}(\rn)$,
Lemma \ref{l4.4} was established in \cite{gz20}. Moreover, the proof for
Lemma \ref{l4.4} presented in \cite{gz20} is also valid when
$A$ satisfies the $(\dz,R)$-BMO
condition for sufficiently small $\dz\in(0,\fz)$ and some $R\in(0,\fz)$.

\begin{lemma}\label{l4.4}
Let $n\ge3$, $\boz\subset\rn$ be a bounded Lipschitz domain and
$\az\in L^\fz(\partial\boz)$ as in \eqref{1.4}.
Assume that the matrix $A$ is real-valued,
symmetric, bounded and measurable, and satisfies \eqref{1.3}.
Then there exist positive constants $\uc_0,\,\dz_0\in(0,\fz)$,
depending only on $n$ and the Lipschitz constant of $\boz$, such that, for any given $p\in[2, 3+\uc_0)$,
if $A$ satisfies the $(\dz,R)$-$\mathrm{BMO}$
condition for some $\dz\in(0,\dz_0)$ and $R\in(0,\fz)$ or $A\in\mathrm{VMO}(\rn)$,
then the weak reverse H\"older inequality \eqref{1.14},
with $p_0$ replaced by $2$, holds true for any function $v\in W^{1,2}(B(x_0,2r)\cap\boz)$ satisfying
$\mathrm{div}(A\nabla v)=0$ in $B(x_0,2r)\cap\boz$ and
$\frac{\partial v}{\partial\boldsymbol{\nu}}+\az v=0$
on $B(x_0,2r)\cap\partial\boz$ when $x_0\in\partial\boz$, where $r\in(0,r_0)$
with $r_0\in(0,\diam(\boz))$ being a constant.
\end{lemma}

Now we prove Theorem \ref{t1.2} via using Lemma \ref{l4.4} and Theorem \ref{t1.1}.

\begin{proof}[Proof of Theorem \ref{t1.2}]
We first prove (i). By Remark \ref{r1.1x}, we know that
the Robin problem $(R)_2$ is uniquely solvable and \eqref{1.12} holds true
with $p_0:=2$. From this, Theorem \ref{t1.1} and Lemma \ref{l4.4},
we deduce that there exists a positive constant $\uc_0$, depending only on $n$
and the Lipschitz constant of $\boz$, such that, for any given $p\in[2,3+\uc_0)$,
the Robin problem $(R)_p$ with $g\equiv0$ is uniquely solvable and
\eqref{1.13} holds true. By this and Lemmas \ref{l4.1} and \ref{l4.2},
we conclude that, for any given $p\in((3+\uc_0)',3+\uc_0)$,
the Robin problem $(R)_p$ with $g\equiv0$ is uniquely solvable and
\eqref{1.13} holds true, which, together with Corollary \ref{c4.1},
further implies that, for any given $p\in((3+\uc_0)',3+\uc_0)$,
\eqref{1.17} holds true. This finishes the proof of (i).

Now we show (ii). Let $p\in(\wz{p}_0,3+\uc_0)$
and $\omega\in A_{\frac{p}{\wz{p}_0}}(\rn)\cap RH_{(\frac{3+\uc_0}{p})'}(\rn)$.
Then, from Lemma \ref{l3.1}(i) and the self-improvement property of
the reverse H\"older inequality, it follows that there exists
an $\epsilon\in(0,\min\{p-\wz{p}_0,3+\uc_0-p\})$
such that $\omega\in A_{\frac{p}{\wz{p}_0+\epsilon}}(\rn)\cap
RH_{(\frac{3+\uc_0-\epsilon}{p})'}(\rn)$. Let $s_1:=\wz{p}_0+\epsilon$
and $s_2:=3+\uc_0-\epsilon$. Then $\max\{n/(n-1),(3+\uc_0)'\}<s_1<p<s_2<3+\uc_0$.
Moreover, by (i), we find that
the Robin problems $(R)_{s_1}$ and $(R)_{s_2}$ with $g\equiv0$ are uniquely solvable
and the weak solutions $u$ satisfy \eqref{1.13},
which, combined with Theorem \ref{t1.1}, $\omega\in A_{\frac{p}{s_1}}(\rn)\cap
RH_{(\frac{s_2}{p})'}(\rn)$ and $s_1>n/(n-1)$, implies that
the weighted Robin problem $(R)_{p,\,\omega}$ with
$\mathbf{f}\in L^p_\omega(\boz;\rn)$ and $F\in L^{p_\ast}_{\omega^{p_\ast/p}}(\boz)$
is uniquely solvable and the weak solution $u$ satisfies the estimate \eqref{1.18}.
This finishes the proof of (ii) and hence of Theorem \ref{t1.2}.
\end{proof}

\section{Proof of Theorem \ref{t1.3}\label{s5}}

\hskip\parindent In this section, we prove Theorem \ref{t1.3}
by using Theorem \ref{t1.1}. We begin with the following Meyer type estimates in the
case of the Robin boundary case.

\begin{lemma}\label{l5.1}
Let $n\ge3$, $\boz\subset\rn$ be a bounded Lipschitz domain
and $r_0\in(0,\diam(\boz))$ a constant. Assume that the matrix $A$ is real-valued,
symmetric, bounded and measurable, and satisfies \eqref{1.3}.
Let $u\in W^{1,2}(B(x_0,4r)\cap\boz)$ be a solution of the equation
$\mathrm{div}(A\nabla u)=0$ in $B(x_0,4r)\cap\boz$ with $\frac{\partial u}{\partial\boldsymbol{\nu}}+\az u=0$
on $B(x_0,4r)\cap\partial\boz$, where $x_0\in\overline{\boz}$, $r\in(0,r_0/4)$
and $\az\in L^\fz(\partial\boz)$ is as in \eqref{1.4}.
Then there exist positive constants $C$, $\uc_0\in(0,\fz)$, independent of $u$, $x_0$ and $r$, such that
\begin{align}\label{5.1}
\lf[\fint_{B_\boz(x_0,r)}(|u|+|\nabla u|)^{2+\uc_0}\,dx\r]^{\frac1{2+\uc_0}}
\le C\lf[1+\|\az\|_{L^\fz(\partial\boz)}\r]\lf[\fint_{B_\boz(x_0,2r)}(|u|+|\nabla u|)^2\,dx\r]^{\frac12}.
\end{align}
\end{lemma}

To show Theorem \ref{l5.1}, we need the following Lemmas \ref{l5.2} and \ref{l5.3}.

\begin{lemma}\label{l5.2}
Let $n\ge3$, $\boz\subset\rn$ be a bounded Lipschitz domain
and $r_0\in(0,\diam(\boz))$ a constant.  Assume that $x_0\in\partial\boz$
and $r\in(0,r_0)$. Let $u$ be as in Lemma \ref{l5.1}. Then
there exists a positive constant $C$, independent of $u$, $x_0$ and $r$, such that
\begin{align}\label{5.2}
\lf[\fint_{B_\boz(x_0,r)}|\nabla u|^{2}\,dx\r]^{\frac12}
\le C\lf\{\fint_{B_\boz(x_0,2r)}|\nabla u|\,dx+\lf[\fint_{B(x_0,2r)\cap\partial\boz}
|\az u|^p\,d\sz(x)\r]^{\frac1p}\r\},
\end{align}
where $p:=\frac{2(n-1)}{n-2}$.
\end{lemma}

Lemma \ref{l5.2} is just \cite[Lemma 3.13]{ob13} (see also \cite[Lemma 3.7]{tob13}).

\begin{lemma}\label{l5.3}
Let $n\ge3$, $\boz\subset\rn$ be a bounded Lipschitz domain
and $r_0\in(0,\diam(\boz))$ a constant.  Assume that $x_0\in\partial\boz$
and $r\in(0,r_0)$. Let $u$ be as in Lemma \ref{l5.1}. Then
there exists a positive constant $C$, independent of $u$, $x_0$ and $r$, such that
\begin{align}\label{5.3}
\sup_{x\in B_\boz(x_0,r)}|u(x)|\le
C\lf[\fint_{B_\boz(x_0,2r)}|u|^{2}\,dx\r]^{\frac12}.
\end{align}
\end{lemma}

Lemma \ref{l5.3} can be proved via using the Moser iteration method,
which is similar to that of \cite[Lemma 3.1]{ls04}; we omit the details here.

Now we show Lemma \ref{l5.1} by using Lemmas \ref{l5.2} and \ref{l5.3}.

\begin{proof}[Proof of Lemma \ref{l5.1}]
We prove this lemma via considering the following two cases for $x_0$ and $r$.

\emph{Case 1)} $B(x_0,4r)\subset\boz$. In this case, $\mathrm{div}(A\nabla u)=0$ in $B(x_0,4r)$.
Let $\eta\in C^\fz_{\mathrm{c}}(2B(x_0,r))$ be such that $0\le\eta\le1$, $\eta\equiv1$ on $B(x_0,r)$
and $|\nabla\eta|\ls r^{-1}$.
Take $v:=(u-\overline{u})\eta^2$ as a test function,
where $\overline{u}:=\fint_{2B(x_0,r)}u\,dx.$
Then
$$\int_{B(x_0,2r)}A\nabla u\cdot\nabla v\,dx=0,
$$
which, together with $|\nabla\eta|\ls r^{-1}$, further implies that
$$\int_{B(x_0,r)}|\nabla u|^2\,dx\ls\frac{1}{r^2}\int_{B(x_0,2r)}
|u-\overline{u}|^2\,dx.
$$
By this and the Sobolev inequality, we conclude that
\begin{align}\label{5.4}
\lf[\fint_{B(x_0,r)}|\nabla u|^{2}\,dx\r]^{\frac12}\ls
\lf[\fint_{B(x_0,2r)}|\nabla u|^{\frac{2n}{n+2}}\,dx\r]^{\frac{n+2}{2n}}.
\end{align}
Moreover, from the Sobolev inequality and the H\"older inequality, it follows that
\begin{align}\label{5.5}
\lf[\fint_{B(x_0,2r)}|u|^{2}\,dx\r]^{\frac12}&\le
\lf[\fint_{B(x_0,2r)}|u-\overline{u}|^2\,dx\r]^{\frac12}+|\overline{u}|\\
&\ls\lf[\fint_{B(x_0,2r)}(|u|+|\nabla u|
)^{\frac{2n}{n+2}}\,dx\r]^{\frac{n+2}{2n}},\nonumber
\end{align}
which, combined with \eqref{5.4}, implies that
\begin{align}\label{5.6}
\lf[\fint_{B(x_0,r)}(|u|+|\nabla u|)^{2}\,dx\r]^{\frac12}\ls
\lf[\fint_{B(x_0,2r)}(|u|+|\nabla u|)^{\frac{2n}{n+2}}
\,dx\r]^{\frac{n+2}{2n}}.
\end{align}

\emph{Case 2)} $x_0\in\partial\boz$.
In this case, let $p:=\frac{2(n-1)}{n-2}$.
Then, by Lemmas \ref{l5.2} and \ref{l5.3}, we find that
\begin{align*}
\lf[\fint_{B_\boz(x_0,r/2)}|\nabla u|^{2}\,dx\r]^{\frac12}&\ls\fint_{B_\boz(x_0,r)}|\nabla u|\,dx
+\lf[\fint_{B(x_0,r)\cap\partial\boz}
|\az u|^p\,d\sz(x)\r]^{\frac1p}\nonumber\\
&\ls\fint_{B_\boz(x_0,r)}|\nabla u|\,dx
+\lf[\sup_{x\in B_\boz(x_0,r)}|u(x)|\r]\|\az\|_{L^\fz(\partial\boz)}\nonumber\\
&\ls\fint_{B_\boz(x_0,2r)}|\nabla u|\,dx+
\lf[\fint_{B_\boz(x_0,2r)}|u|\,dx\r]
\|\az\|_{L^\fz(\partial\boz)},
\end{align*}
which implies that
\begin{align}\label{5.7}
\lf[\fint_{B_\boz(x_0,r/2)}|\nabla u|^{2}\,dx\r]^{\frac12}
\ls\lf[1+\|\az\|_{L^\fz(\partial\boz)}\r]\fint_{B_\boz(x_0,2r)}(|u|+|\nabla u|)\,dx.
\end{align}
Furthermore, similarly to \eqref{5.5}, we know that
\begin{align*}
\lf[\fint_{B_\boz(x_0,r/2)}|u|^{2}\,dx\r]^{\frac12}
\ls\lf[1+\|\az\|_{L^\fz(\partial\boz)}\r]\fint_{B_\boz(x_0,2r)}(|u|+|\nabla u|)\,dx,
\end{align*}
which, together with \eqref{5.7}, further implies that
\begin{align*}
\lf[\fint_{B_\boz(x_0,r/2)}(|u|+|\nabla u|)^{2}\,dx\r]^{\frac12}
\ls\lf[1+\|\az\|_{L^\fz(\partial\boz)}\r]\fint_{B_\boz(x_0,2r)}(|u|+|\nabla u|)\,dx.
\end{align*}
By this and a simple covering argument, we conclude that
\begin{align}\label{5.8}
\lf[\fint_{B_\boz(x_0,r)}(|u|+|\nabla u|)^{2}\,dx\r]^{\frac12}
\ls\lf[1+\|\az\|_{L^\fz(\partial\boz)}\r]\fint_{B_\boz(x_0,2r)}(|u|+|\nabla u|)\,dx.
\end{align}

Finally, from \eqref{5.6}, \eqref{5.8} and the self-improvement property of the weak reverse H\"older inequality
(see, for instance, \cite[pp.\,122-123]{g83}), it follows that
there exists an $\uc_0\in(0,\fz)$ such that, for any
$x_0\in\overline{\boz}$ and $r\in(0,r_0/4)$,
\begin{align*}
\lf[\fint_{B_\boz(x_0,r)}(|u|+|\nabla u|)^{2+\uc_0}\,dx\r]^{\frac1{2+\uc_0}}\ls\lf[1+\|\az\|_{L^\fz(\partial\boz)}\r]
\lf[\fint_{B_\boz(x_0,2r)}(|u|+|\nabla u|)^2\,dx\r]^{\frac12},
\end{align*}
which is the desired estimate.
This finishes the proof of Lemma \ref{l5.1}.
\end{proof}

Moreover, to show Theorem \ref{t1.3}, we need the following Lemma \ref{l5.4}.

\begin{lemma}\label{l5.4}
Let $n\ge3$, $\boz\subset\rn$ be a bounded Lipschitz domain
and $r_0\in(0,\diam(\boz))$ a constant. Assume that the matrix $A$ is real-valued,
symmetric, bounded and measurable, and satisfies \eqref{1.3}.
Let $u\in W^{1,2}(B(x_0,4r)\cap\boz)$ be a solution of the equation
$\mathrm{div}(A\nabla u)=0$ in $B(x_0,4r)\cap\boz$ with
$\frac{\partial u}{\partial\boldsymbol{\nu}}+\az u=0$
on $B(x_0,4r)\cap\partial\boz$, where $x_0\in\overline{\boz}$, $r\in(0,r_0/4)$,
and $\az\in L^\fz(\partial\boz)$
satisfies $\az\ge \az_0$ with $\az_0\in(0,\fz)$ being a constant.
Then there exist a function $\tz:=\tz(r)$, $p\in(2,\fz)$
and a function $v\in W^{1,p}(B(x_0,r)\cap\boz)$ such that
\begin{equation}\label{5.9}
\lf[\fint_{B_\boz(x_0,r)}(|v|+|\nabla v|)^p\,dx\r]^{\frac1p}
\le C\lf[\fint_{B_\boz(x_0,4r)}(|u|+|\nabla u|)^2\,dx\r]^{\frac12}
\end{equation}
and
\begin{align}\label{5.10}
\lf\{\fint_{B_\boz(x_0,r)}[|u-v|+|\nabla(u-v)|]^2\,dx\r\}^{\frac12}
\le\tz(r)\lf[\fint_{B_\boz(x_0,4r)}(|u|+|\nabla u|)^2\,dx\r]^{\frac12},
\end{align}
where $C$ is a positive constant independent of $u,\,v$, $x_0$ and $r$.
\end{lemma}

\begin{proof}
Let $A_0:=\{c_{ij}\}_{i,j=1}^n$, where, for any $i,\,j\in\{1,\,\ldots,\,n\}$,
$c_{ij}:=\fint_{B(x_0,2r)}a_{ij}\,dx$.
Assume that $v\in W^{1,2}(B_\boz(x_0,2r))$
is the solution of the following boundary value problem
\begin{equation}\label{5.11}
\begin{cases}
-\mathrm{div}(A_0\nabla v)=0\ \ &\text{in}\ \ B_\boz(x_0,2r),\\
A_0\nabla v\cdot\boldsymbol{\nu}+\az v=A\nabla u\cdot\boldsymbol{\nu}+\az u\ \ &\text{on} \ \
\partial B_\boz(x_0,2r).
\end{cases}
\end{equation}
Using $u-v$ as a test function in \eqref{5.11}, we find that
\begin{align*}
&\int_{B_\boz(x_0,2r)}A_0\nabla(u-v)\cdot\nabla(u-v)\,dx
+\int_{\partial B_\boz(x_0,2r)}\az(u-v)^2\,d\sz(x)\\
&\quad\quad\quad=\int_{B_\boz(x_0,2r)}(A_0-A)\nabla u\cdot\nabla(u-v)\,dx,
\end{align*}
which, combined with \eqref{1.3} and the H\"older inequality, implies that,
for any $\uc\in(0,\fz)$, there exists a positive constant $C_{(\uc)}$, depending only on $\uc$,
such that
\begin{align}\label{5.12}
&\mu_0\int_{B_\boz(x_0,2r)}|\nabla(u-v)|^2\,dx
+\int_{\partial B_\boz(x_0,2r)}\az(u-v)^2\,d\sz(x)\\
&\quad\le
\int_{B_\boz(x_0,2r)}|A_0-A||\nabla u||\nabla(u-v)|\,dx\nonumber\\
&\quad\le\uc\int_{B_\boz(x_0,2r)}|\nabla(u-v)|^2\,dx+ C_{(\uc)}
\int_{B_\boz(x_0,2r)}|A_0-A|^2|\nabla u|^2\,dx,\nonumber
\end{align}
where $\mu_0\in(0,1)$ is as in \eqref{1.3}.
Take $\uc:=\mu_0/2$ in \eqref{5.12}. Then, by \eqref{5.12},
the assumption $\az\ge\az_0$ and the Friedrichs inequality
(see, for instance, \cite[pp.\,11-12, Theorem 1.9]{n12}
and \cite[Theorem 6.1]{ls04}), we conclude that
\begin{align}\label{5.13}
&\int_{B_\boz(x_0,2r)}[|u-v|+|\nabla(u-v)|]^2\,dx\\
&\quad\le C\lf[\int_{B_\boz(x_0,2r)}|\nabla(u-v)|^2\,dx
+\int_{\partial B_\boz(x_0,2r)}(u-v)^2\,d\sz(x)\r]\nonumber\\
&\quad\le C_7\int_{B_\boz(x_0,2r)}|A_0-A|^2|\nabla u|^2\,dx.\nonumber
\end{align}
Moreover, from Lemma \ref{l5.1}, we deduce that there exist positive
constants $p_1\in(1,\fz)$ and $C_8\in(0,\fz)$, independent of $u$, $x_0$, and $r$, such that
$$\lf[\fint_{B_\boz(x_0,2r)}
(|u|+|\nabla u|)^{2p_1}\,dx\r]^{\frac1{2p_1}}\le C_8
\lf[\fint_{B_\boz(x_0,4r)}
(|u|+|\nabla u|)^{2}\,dx\r]^{\frac12},
$$
which, together with \eqref{5.13} and the H\"older inequality, further implies that
\begin{align}\label{5.14}
&\lf[\fint_{B_\boz(x_0,2r)}
(|u-v|+|\nabla (u-v)|)^2\,dx\r]^{\frac12}\\
&\quad\le C_7\lf[\fint_{B(x_0,2r)}
|A_0-A|^{2p_1'}\,dx\r]^{\frac1{2p_1'}}\lf[\fint_{B_\boz(x_0,2r)}
|\nabla u|^{2p_1}\,dx\r]^{\frac1{2p_1}}\nonumber\\
&\quad\le\tz(r)\lf[\fint_{B_\boz(x_0,4r)}
(|u|+|\nabla u|)^{2}\,dx\r]^{\frac12},\nonumber
\end{align}
where $1/p_1+1/p_1'=1$ and
\begin{equation}\label{5.15}
\tz(r):=C_7C_8\sup_{x\in\overline{\boz},\,t\in(0,r]}
\lf[\frac{1}{|B(x,2t)|}\int_{B(x,2t)}
\lf|A(y)-\frac{1}{|B(x,2t)|}\int_{B(x,2t)}A(z)\,dz\r|^{2p_1'}\,dy\r]^{\frac1{2p_1'}}.
\end{equation}
Thus, \eqref{5.10} holds true.
Furthermore, from the known regularity theory of second order elliptic equations
(see, for instance, \cite[Chapter 1]{k94} and \cite[Chapter V]{g83}), it follows that there exists a
$p\in(2,\fz)$ such that
$$\lf[\fint_{B_\boz(x_0,2r)}
|\nabla v|^p\,dx\r]^{\frac1p}\ls\lf[\fint_{B_\boz(x_0,4r)}
(|v|+|\nabla v|)^2\,dx\r]^{\frac12},
$$
which, combined with \eqref{5.14} and the fact that $A=\{a_{ij}\}_{i,j=1}^n\in L^\fz(\rn;\rr^{n^2})$,
further implies that
\begin{equation}\label{5.16}
\lf[\fint_{B_\boz(x_0,2r)}
|\nabla v|^p\,dx\r]^{\frac1p}\ls\lf[\fint_{B_\boz(x_0,4r)}
(|u|+|\nabla u|)^2\,dx\r]^{\frac12}.
\end{equation}
Furthermore, by the Sobolev inequality and \eqref{5.16}, we conclude that
\begin{align*}
\lf[\fint_{B_\boz(x_0,2r)}
|v|^p\,dx\r]^{\frac1p}&\ls\lf[\fint_{B_\boz(x_0,2r)}
|v-\overline{v}|^p\,dx\r]^{\frac1p}+|\overline{v}|\ls\lf[\fint_{B_\boz(x_0,2r)}
|\nabla v|^p\,dx\r]^{\frac1p}+|\overline{v}|\nonumber\\
&\ls\lf[\fint_{B_\boz(x_0,2r)}(|u|+|\nabla u|)^2\,dx\r]^{\frac12},
\end{align*}
where $\overline{v}:=\fint_{B_\boz(x_0,2r)}v\,dx$,
which, together with \eqref{5.16} again, further implies that
\begin{equation*}
\lf[\fint_{B_\boz(x_0,2r)}
(|v|+|\nabla v|)^p\,dx\r]^{\frac1p}\ls\lf[\fint_{B_\boz(x_0,4r)}
(|u|+|\nabla u|)^2\,dx\r]^{\frac12}.
\end{equation*}
Thus, \eqref{5.9} holds true. This finishes the proof of Lemma \ref{l5.4}.
\end{proof}

\begin{lemma}\label{l5.5}
Let $n\ge3$, $\boz\subset\rn$ be a bounded $C^1$ or (semi-)convex domain,
and $\az\in L^\fz(\partial\boz)$ satisfy $\az\ge \az_0$ with $\az_0\in(0,\fz)$ being a constant.
Assume that $p\in(2,\fz)$ and $A_0:=\{a_{ij}\}_{i,j=1}^n$ is a symmetric matrix
with constant coefficients that satisfy \eqref{1.3}. Let $v\in W^{1,2}(B_\boz(x_0,2r))$ be a weak solution
of the equation $\mathrm{div}(A_0\nabla v)=0$ in $B_\boz(x_0,2r)$ with
$\frac{\partial v}{\partial\boldsymbol{\nu}}+\az u=0$ on $B(x_0,2r)\cap\partial\boz$,
where $B(x_0,r)$ is a ball such that $r\in(0,r_0/4)$
and either $x_0\in\partial\boz$ or $B(x_0,2r)\subset\boz$, and $r_0\in(0,\diam(\boz))$
is a constant. Then there exists a positive constant $C$, depending only
on $n$, $p$ and $\boz$, such that
\begin{equation}\label{5.17}
\lf[\fint_{B_\boz(x_0,r)}(|v|+|\nabla v|)^p\,dx\r]^{\frac1p}
\le C\lf[\fint_{B_\boz(x_0,2r)}(|v|+|\nabla v|)^2\,dx\r]^{\frac12}.
\end{equation}
\end{lemma}

To prove Lemma \ref{l5.5}, we need the following three lemmas.

\begin{lemma}\label{l5.6}
Let $n\ge3$ and $\boz\subset\rn$ be a bounded Lipschitz domain.
Assume that $x\in\partial\boz$, $r\in(0,r_0)$ and $\Delta u=0$
in $4B(x,r)\cap\boz$, where $r_0\in(0,\diam(\boz))$ is a constant.
If $|\nabla u|\in L^2(4B(x,r)\cap\boz)$ and
$\frac{\partial u}{\partial\boldsymbol{\nu}}\in L^2(4B(x,r)\cap\partial\boz)$,
then $|\nabla u|\in L^2(4B(x,r)\cap\partial\boz)$ and there exists a positive constant $C$,
independent of $u$, $x$ and $r$, such that
$$\int_{B(x,r)\cap\partial\boz}[(\nabla u)_r^{\ast}]^2\,d\sz(x)
\le C\lf[\int_{4B(x,r)\cap\partial\boz}\lf|\frac{\partial u}{\partial\boldsymbol{\nu}}\r|^2
\,d\sz(x)+\frac{1}{r}\int_{4B(x,r)\cap\boz}|\nabla u|^2\,dy\r],
$$
where, for any $x\in\partial\boz$,
$$(\nabla u)_r^{\ast}(x):=\sup\lf\{|\nabla v(y)|:\ y\in\boz\cap B(x,r),\ |x-y|<2\dz(y)\r\}
$$
with $\dz(y):=\dist(y,\partial\boz)$.
\end{lemma}

Lemma \ref{l5.7} is just \cite[Lemma 4.4]{ob13} (see also \cite[Lemma 4.2]{tob13}).

\begin{lemma}\label{l5.7}
Let $n\ge3$ and $\boz\subset\rn$ be a bounded Lipschitz domain.
Assume that $q\in(1,n)$, $p:=\frac{q(n-1)}{n-q}$ and $v\in W^{1,2}(B\cap\boz)$
satisfies $\int_{B\cap\boz}v\,dx=0$. Then there exists a positive constant $C$,
depending only on $n$, $q$ and the Lipschitz constant of $\boz$, such that
$$\lf[\int_{B\cap\partial\boz}|v|^p\,d\sz(x)\r]^{\frac1p}\le C
\lf(\int_{B\cap\boz}|\nabla v|^q\,dx\r)^{\frac1q}.
$$
\end{lemma}

Lemma \ref{l5.7} is just \cite[Lemma 3.1]{ks08}.

\begin{lemma}\label{l5.8}
Let $n\ge3$, $\boz\subset\rn$ be a bounded $C^1$
or (semi-)convex domain, $\az\in L^\fz(\partial\boz)$ satisfy
$\az\ge \az_0$ with $\az_0\in(0,\fz)$ being a constant, and $r_0\in(0,\diam(\boz))$ be a constant.
Assume that $\Delta v=0$ in $\boz$,
$(\nabla v)^\ast\in L^{2}(\partial\boz)$ and $\frac{\partial v}{\partial\boldsymbol{\nu}}+\az v=0$
on $B(x,3r)\cap\partial\boz$, where $x\in\partial\boz$ and $r\in(0,r_0)$.
Then, for any $p\in(2,\fz)$, there exists a positive constant $C\in(0,\fz)$
such that the weak reverse H\"older inequality
\begin{align*}
\lf\{\fint_{B(x,r)\cap\partial\boz}\lf[(\nabla v)^\ast\r]^p\,d\sz(x)\r\}^{\frac{1}{p}}
\le C\lf\{\fint_{B(x,2r)\cap\partial\boz}
\lf[(\nabla v)^\ast\r]^{2}\,d\sz(x)\r\}^{\frac{1}{2}}
\end{align*}
holds true. Here $(\nabla v)^\ast$ denotes the
non-tangential maximal function of $\nabla v$,
namely, for any $x\in\partial\boz$,
\begin{equation*}
(\nabla v)^\ast(x):=\sup\lf\{|\nabla v(y)|:\ y\in\boz,\ |x-y|<2\dz(y)\r\}
\end{equation*}
with $\dz(y):=\dist(y,\partial\boz):=\inf\{|y-z|:\ z\in\partial\boz\}$.
\end{lemma}

When $\boz\subset\rn$ is a bounded (semi-)convex domain,
the conclusion of Lemma \ref{l5.8} is a simple
corollary of \cite[Theorems 1.2 and 3.2]{yyy18}.
Moreover, when $\boz\subset\rn$ is a bounded $C^1$ domain, the conclusion of Lemma \ref{l5.8}
is a simple corollary of \cite[Theorem 1.2]{yyy18} and \cite[Theorem 5.1]{ls04}.
We omit the details here.

We now show Lemma \ref{l5.5} via using Lemmas \ref{l5.6}, \ref{l5.7} and \ref{l5.8}.

\begin{proof}[Proof of Lemma \ref{l5.5}]
Observing that the matrix $A_0$ is symmetric and elliptic, and
has constant coefficients, by a change of the coordinate system,
we may assume that $A_0=I$, namely, $\Delta v=0$ in $B(x_0,2r)\cap\boz$
and $\frac{\partial v}{\partial\boldsymbol{\nu}}+\az v=0$ on $B(x_0,2r)\cap\partial\boz$.
We now consider the following two cases.

\emph{Case 1)} $B(x_0,2r)\subset\boz$. In this case, by the interior estimate of harmonic
functions (see, for instance, \cite[Section 2.7]{gt01}), we know that
$$\sup_{y\in B(x_0,r)}|v(y)|\ls\lf[
\fint_{B(x_0,2r)}|v|^2\,dx\r]^{\frac12}
\quad\text{and}\quad
\sup_{y\in B(x_0,r)}|\nabla v(y)|\ls\lf[
\fint_{B(x_0,2r)}|\nabla v|^2\,dx\r]^{\frac12},
$$
which implies that \eqref{5.17} holds true in this case.

\emph{Case 2)} $x_0\in\partial\boz$.
In this case, it is easy to see that
$$
\lf[\fint_{B_\boz(x_0,r)}|\nabla v|^p\,dx\r]^{\frac1p}
\ls\lf\{\fint_{B(x_0,r)\cap\partial\boz}
\lf[(\nabla v)^\ast\r]^p\,d\sz(x)\r\}^{\frac1p},
$$
which, combined with Lemma \ref{l5.8}, further implies that
\begin{equation}\label{5.18}
\lf[\fint_{B_\boz(x_0,r)}|\nabla v|^p\,dx\r]^{\frac1p}
\ls\lf\{\fint_{B(x_0,2r)\cap\partial\boz}
\lf[(\nabla v)^\ast\r]^2\,d\sz(x)\r\}^{\frac12}.
\end{equation}
For any $x\in\partial\boz$, let
$$(\nabla v)_{r,\,e}^{\ast}(x):=\sup\lf\{|\nabla v(y)|:\ y\in\boz\cap B(x,r)^\complement,\ |x-y|<2\dz(y)\r\}.
$$
Then
\begin{align}\label{5.19}
\lf\{\fint_{B(x_0,2r)\cap\partial\boz}
\lf[(\nabla v)_{r,\,e}^\ast\r]^2\,d\sz(x)\r\}^{\frac12}
\ls\sup_{y\in B_\boz(x,4r)\backslash B(x,r)}|\nabla v(y)|
\ls\lf[\fint_{B_\boz(x_0,4r)}|\nabla v|^2\,dx\r]^{\frac12}.
\end{align}
Furthermore, by Lemma \ref{l5.6} and the fact that $\frac{\partial v}{\partial\boldsymbol{\nu}}+\az v=0$
on $B(x_0,2r)\cap\partial\boz$, we conclude that
\begin{align}\label{5.20}
&\lf\{\fint_{B(x_0,r)\cap\partial\boz}
\lf[(\nabla v)_{r}^\ast\r]^2\,d\sz(x)\r\}^{\frac12}\\
&\quad\ls\lf[\fint_{B(x_0,4r)\cap\partial\boz}
\lf|\az v\r|^2\,d\sz(x)\r]^{\frac12}+\lf[
\fint_{B_\boz(x_0,4r)}|\nabla v|^2\,dy\r]^{\frac12}.\nonumber
\end{align}
From Lemma \ref{l5.7} and the H\"older inequality,
we deduce that
\begin{align*}
&\lf[\int_{B(x_0,4r)\cap\partial\boz}
\lf|\az v\r|^2\,d\sz(x)\r]^{\frac12}\nonumber\\
&\quad\le\lf[\int_{B(x_0,4r)\cap\partial\boz}
\lf|\az \overline{v}\r|^2\,d\sz(x)\r]^{\frac12}+\lf[\int_{B(x_0,4r)\cap\partial\boz}
\lf|\az(v-\overline{v})\r|^2\,d\sz(x)\r]^{\frac12}\nonumber\\
&\quad\le|\overline{v}|\lf[\int_{B(x_0,4r)\cap\partial\boz}
\lf|\az\r|^2\,d\sz(x)\r]^{\frac12}\nonumber\\
&\quad\quad+\lf[\int_{B(x_0,4r)\cap\partial\boz}
\lf|v-\overline{v}\r|^{\frac{2(n-1)}{n-2}}\,d\sz(x)\r]^{\frac{n-2}{2(n-1)}}
\lf[\int_{B(x_0,4r)\cap\partial\boz}
\lf|\az\r|^{2(n-1)}\,d\sz(x)\r]^{\frac{1}{2(n-1)}}\nonumber\\
&\quad\ls \lf\{r^{\frac{n-1}{2}}\lf[\fint_{B_\boz(x_0,4r)}|v|\,dx\r]
+r^{\frac{n+1}{2}}
\lf[\fint_{B_\boz(x_0,4r)}|\nabla v|^2\,dy\r]^{\frac12}\r\}\|\az\|_{L^\fz(\partial\boz)},
\end{align*}
where $\overline{v}:=\fint_{B(x_0,4r)\cap\boz}v\,dy$,
which further implies that
\begin{align}\label{5.21}
&\lf[\fint_{B(x_0,4r)\cap\partial\boz}
\lf|\az v\r|^2\,d\sz(x)\r]^{\frac12}\ls
\lf[\fint_{B_\boz(x_0,4r)}(|v|+|\nabla v|)^2\,dx\r]^{\frac12}
\|\az\|_{L^\fz(\partial\boz)}.
\end{align}
Thus, by \eqref{5.20} and \eqref{5.21}, we conclude that
\begin{align*}
\lf\{\fint_{B(x_0,r)\cap\partial\boz}
\lf[(\nabla v)_{r}^\ast\r]^2\,d\sz(x)\r\}^{\frac12}
\ls\lf[1+\|\az\|_{L^\fz(\partial\boz)}\r]
\lf[\fint_{B_\boz(x_0,4r)}(|v|+|\nabla v|)^2\,dx\r]^{\frac12},
\end{align*}
which, together with \eqref{5.19} and the fact that $(\nabla v)^\ast\le
(\nabla v)_{r}^\ast+(\nabla v)_{r,\,e}^\ast$, further implies that
\begin{align*}
\lf\{\fint_{B(x_0,r)\cap\partial\boz}
\lf[(\nabla v)^\ast\r]^2\,d\sz(x)\r\}^{\frac12}
\ls\lf[1+\|\az\|_{L^\fz(\partial\boz)}\r]
\lf[\fint_{B_\boz(x_0,4r)}(|v|+|\nabla v|)^2\,dx\r]^{\frac12}.
\end{align*}
From this estimate and Lemma \ref{l5.8}, it follows that
$$
\lf\{\fint_{B(x_0,r)\cap\partial\boz}
\lf[(\nabla v)^\ast\r]^p\,d\sz(x)\r\}^{\frac1p}
\ls\lf[1+\|\az\|_{L^\fz(\partial\boz)}\r]
\lf[\fint_{B_\boz(x_0,4r)}(|v|+|\nabla v|)^2\,dx\r]^{\frac12},
$$
which further implies that
\begin{align}\label{5.22}
\lf[\fint_{B_\boz(x_0,r)}|\nabla v|^p\,dx\r]^{\frac1p}
\ls\lf[1+\|\az\|_{L^\fz(\partial\boz)}\r]
\lf[\fint_{B_\boz(x_0,4r)}(|v|+|\nabla v|)^2\,dx\r]^{\frac12}.
\end{align}
Moreover, by the H\"older inequality,
the Sobolev inequality and \eqref{5.22}, we find that
\begin{align*}
\lf[\fint_{B_\boz(x_0,r)}|v|^p\,dx\r]^{\frac1p}&\le
\lf[\fint_{B_\boz(x_0,r)}|v-\overline{v}|^p\,dx\r]^{\frac1p}+|\overline{v}|
\ls\lf[\fint_{B_\boz(x_0,r)}|\nabla v|^p\,dx\r]^{\frac1p}+|\overline{v}|\\
&\ls\lf[1+\|\az\|_{L^\fz(\partial\boz)}\r]
\lf[\fint_{B_\boz(x_0,4r)}(|v|+|\nabla v|)^2\,dx\r]^{\frac12},
\end{align*}
where $\overline{v}:=\fint_{B_\boz(x_0,r)}v\,dy$,
which, combined with \eqref{5.22}, implies that \eqref{5.17} holds true.
This finishes the proof of Lemma \ref{l5.5}.
\end{proof}

To prove Theorem \ref{t1.3}(ii), we need the following weighted Sobolev inequality,
which was established in \cite[Theorem 1.5]{fks82}.

\begin{lemma}\label{l5.9}
Let $n\ge2$ and $\boz\subset\rn$ be a bounded Lipschitz domain.
Assume that $p\in[1,\fz)$, $\omega\in A_p(\rn)$ and
$u\in W^{1,p}_\omega(\boz)$ satisfies $\int_\boz u\,dx=0$.
Then there exists a positive constant $C$, depending only on $n$, $p$,
the Lipschitz constant of $\boz$, and $[\omega]_{A_p(\rn)}$, such that
$$\lf[\frac{1}{\omega(\boz)}\int_{\boz}|u|^{\frac{np}{n-1}}\omega\,dx\r]^{\frac{n-1}{np}}
\le C|\boz|^{\frac1n}
\lf[\frac{1}{\omega(\boz)}\int_{\boz}|\nabla u|^p\omega\,dx\r]^{\frac{1}{p}}.
$$
\end{lemma}

Now we show Theorem \ref{t1.3} by using Theorem \ref{t1.1}
and Lemmas \ref{l5.5} and \ref{l5.9}.

\begin{proof}[Proof of Theorem \ref{t1.3}]
We first prove (i) via splitting the proof into the following two steps.

\textbf{Step 1.} Assume that $u_1$ is the weak
solution of the following Robin problem
\begin{equation}\label{5.23}
\begin{cases}
-\mathrm{div}(A\nabla u_1)=\mathrm{div}(\mathbf{g}_1)\ \ &\text{in}\ \ \boz,\\
\dfrac{\partial u_1}{\partial\boldsymbol{\nu}}+\az u_1=\mathbf{g}_1\cdot\boldsymbol{\nu} \ \ &\text{on}\ \ \partial\boz.
\end{cases}
\end{equation}
In this step, we show that, for any given $p\in(1,\fz)$,
there exists a $\dz_0\in(0,\fz)$, depending only on $n$, $p$ and $\boz$,
such that, if $A$ satisfies the $(\dz,R)$-$\mathrm{BMO}$ condition for some $\dz\in(0,\dz_0)$
and $R\in(0,\fz)$ or $A\in\mathrm{VMO}(\rn)$, then
\begin{equation}\label{5.24}
\|u_1\|_{W^{1,p}(\boz)}\ls\|\mathbf{g}_1\|_{L^p(\boz;\rn)}.
\end{equation}
We first assume that $p\in(2,\fz)$. By the definition of $\mathrm{VMO}(\rn)$,
we find that, if $A\in\mathrm{VMO}(\rn)$,
then there exists an $r_1\in(0,\fz)$ such that, for any $r\in(0,r_1)$,
$\tz(r)<\uc_0/2$, where $\tz(r)$ and $\uc_0$ are, respectively, as in \eqref{5.15}
and Theorem \ref{t3.1}.
Let $B(x_0,r)\subset\rn$ be such
that $r\in(0,\min\{r_0,r_1\}/4)$ and either $x_0\in\partial\boz$ or $B(x_0,2r)\subset\boz$,
where $r_0\in(0,\diam(\boz))$ is as in Lemma \ref{l5.5}.
Assume that $v_1\in W^{1,2}(B_\boz(x_0,2r))$ is a weak solution of the equation
$\mathrm{div}(A\nabla v_1)=0$ in $B_\boz(x_0,2r)$ with
$\frac{\partial v_1}{\partial\boldsymbol{\nu}}+\az v_1=0$
on $B(x_0,2r)\cap\partial\boz$. Let $w_1\in W^{1,2}(B_\boz(x_0,2r))$ be a weak solution
of the equation $\mathrm{div}(A_0\nabla w_1)=0$ in $B_\boz(x_0,2r)$ with
$A_0\nabla w_1\cdot\boldsymbol{\nu}+\az w_1=A\nabla v_1\cdot\boldsymbol{\nu}+\az v_1$ on
$\partial B_\boz(x_0,2r)$, where $A_0:=\{c_{ij}\}_{i,j=1}^n$ with
$c_{ij}:=\fint_{B(x_0,2r)}a_{ij}\,dx$ for any $i,\,j\in\{1,\,\ldots,\,n\}$.
Then, from Lemmas \ref{l5.4} and \ref{l5.5}, we deduce that
\begin{equation}\label{5.25}
\lf[\fint_{B_\boz(x_0,r)}(|w_1|+|\nabla w_1|)^{p+1}\,dx\r]^{\frac1{p+1}}
\ls\lf[\fint_{B_\boz(x_0,2r)}(|v_1|+|\nabla v_1|)^2\,dx\r]^{\frac12}
\end{equation}
and
\begin{align}\label{5.26}
\lf[\fint_{B_\boz(x_0,r)}(|v_1-w_1|+|\nabla(v_1-w_1)|)^2\,dx\r]^{\frac12}
\le\tz(r)\lf[\fint_{B_\boz(x_0,2r)}(|v_1|+|\nabla v_1|)^2\,dx\r]^{\frac12},
\end{align}
where $\tz(r)$ is as in \eqref{5.15}.

By the well-known John-Nirenberg inequality on $\mathrm{BMO}(\rn)$
(see, for instance, \cite{g12,St93}), we know that
there exists a $\dz_0\in(0,\fz)$ sufficiently small such that, if
$A$ satisfies the $(\dz,R)$-$\mathrm{BMO}$ condition with some $\dz\in(0,\dz_0)$
and $R\in(r_0,\fz)$, then, for any $r\in(0,r_0)$,
$\tz(r)<\uc_0/2$, where $\uc_0$ is as in Theorem \ref{t3.1}.
By this, \eqref{5.25}, \eqref{5.26}, $r\in(0,\min\{r_0,r_1\}/4)$
and Theorem \ref{t3.1} with $\omega\equiv1$, we conclude that,
if $A$ satisfies the $(\dz,R)$-$\mathrm{BMO}$ condition with some $\dz\in(0,\dz_0)$
and $R\in(r_0,\fz)$ or $A\in\mathrm{VMO}(\rn)$, then
$|\nabla v_1|\in L^p(B_\boz(x_0,2r))$ and
\begin{equation}\label{5.27}
\lf[\fint_{B_\boz(x_0,r)}(|v_1|+|\nabla v_1|)^{p}\,dx\r]^{\frac1p}
\ls\lf[\fint_{B_\boz(x_0,2r)}(|v_1|+|\nabla v_1|)^2\,dx\r]^{\frac12},
\end{equation}
which, together with Theorem \ref{t1.1}
and the fact that the Robin problem $(R)_2$ is uniquely solvable, further implies that \eqref{5.24} holds true
in the case $p\in(2,\fz)$.

Assume now that $p\in(1,2)$. In this case, $p'\in(2,\fz)$,
where $1/p+1/p'=1$. From this, Theorem \ref{t1.1}, Lemma \ref{l4.2},
and the fact that \eqref{5.24} holds true for any given $p\in(2,\fz)$,
we deduce that \eqref{5.24} also holds true for any given $p\in(1,2)$.
Furthermore, by Theorem \ref{t1.2}(i), we know that \eqref{5.24} holds true when $p=2$.
Thus, \eqref{5.24} holds true for any given $p\in(1,\fz)$.

\textbf{Step 2.} Let $u_2$ be the weak solution of the following Robin problem
\begin{equation}\label{5.28}
\begin{cases}
-\mathrm{div}(A\nabla u_2)=F_2\ \ & \text{in}\ \ \boz,\\
\dfrac{\partial u_2}{\partial\boldsymbol{\nu}}+\az u_2=0\ \ & \text{on}\ \ \partial\boz.
\end{cases}
\end{equation}
Then, from \eqref{5.24} and Lemma \ref{l4.1}, it follows that,
for any given $p\in(1,\fz)$, there exists a $\dz_0\in(0,\fz)$, depending only on $n$, $p$
and $\boz$,  such that, if $A$ satisfies the $(\dz,R)$-$\mathrm{BMO}$ condition for
some $\dz\in(0,\dz_0)$ and $R\in(0,\fz)$ or $A\in\mathrm{VMO}(\rn)$, then
\begin{equation}\label{5.29}
\|u_2\|_{W^{1,p}(\boz)}\ls\|F_2\|_{L^{p_\ast}(\boz)}.
\end{equation}
Assume that $u$ is the weak solution of the Robin problem
\begin{equation*}
\begin{cases}
-\mathrm{div}(A\nabla u)=\mathrm{div}(\mathbf{f})+F\ \ & \text{in}\ \ \boz,\\
\dfrac{\partial u}{\partial\boldsymbol{\nu}}+\az u
=\mathbf{f}\cdot\boldsymbol{\nu}+g\ \ & \text{on}\ \ \partial\boz.
\end{cases}
\end{equation*}
Then, via using \eqref{5.24}, \eqref{5.29} and Corollary \ref{c4.1}, we conclude that,
for any given $p\in(1,\fz)$, there exists a $\dz_0\in(0,\fz)$, depending only on $n$, $p$
and $\boz$,  such that, if $A$ satisfies the $(\dz,R)$-$\mathrm{BMO}$ condition for
some $\dz\in(0,\dz_0)$ and $R\in(0,\fz)$ or $A\in\mathrm{VMO}(\rn)$, then
$$\|u\|_{W^{1,p}(\boz)}\ls\|\mathbf{f}\|_{L^{p}(\boz;\rn)}
+\|F\|_{L^{p_\ast}(\boz)}+\|g\|_{W^{-1/p,p}(\partial\boz)}.$$
This finishes the proof of (i).

Now we show (ii) via also splitting the proof into the following two steps.

\emph{Step 1.} Let $u_1$ be the weak solution of the Robin problem \eqref{5.23}.
In this step we prove that, for any given $p\in(1,\fz)$ and $\omega\in A_p(\rn)$,
there exists a $\dz_0\in(0,\fz)$, depending only on $n$, $p$, $[\omega]_{A_p(\rn)}$
and $\boz$,
such that, if $A$ satisfies the $(\dz,R)$-$\mathrm{BMO}$ condition for
some $\dz\in(0,\dz_0)$ and $R\in(0,\fz)$ or $A\in\mathrm{VMO}(\rn)$, then
\begin{equation}\label{5.30}
\|u_1\|_{W^{1,p}_\omega(\boz)}\ls\|\mathbf{g}_1\|_{L^p_\omega(\boz;\rn)}.
\end{equation}

By $\omega\in A_p(\rn)$ and Lemma \ref{l3.1}(i),
we find that there exists a $p_0\in(1,p)$ such that $\omega\in A_{\frac p{p_0}}(\rn)$.
Furthermore, from Lemma \ref{l3.1}(ii), it follows that there exists a sufficiently
large $p_1\in(p,\fz)$ such that $\omega\in RH_{(\frac{p_1}{p})'}(\rn)$.
Then
\begin{equation}\label{5.31}
\omega\in A_{\frac p{p_0}}(\rn)\cap RH_{(\frac{p_1}{p})'}(\rn).
\end{equation}
Moreover, by \eqref{5.27} and Lemma \ref{l4.3}, we find that the inequality \eqref{5.27} holds true with
$p$ and $2$ replaced, respectively, by $p_1$ and $p_0$, namely,
\begin{equation}\label{5.32}
\lf[\fint_{B_\boz(x_0,r)}(|v_1|+|\nabla v_1|)^{p_2}\,dx\r]^{\frac1{p_2}}
\ls\lf[\fint_{B_\boz(x_0,2r)}(|v_1|+|\nabla v_1|)^{p_0}\,dx\r]^{\frac1{p_0}},
\end{equation}
where $v_1$ is as in \eqref{5.27}.
Furthermore, from Theorem \ref{t1.3}(i), it follows that, for any given $q\in(1,\fz)$,
the Robin problem $(R)_q$ is uniquely solvable and \eqref{1.23} holds true,
which, combined with \eqref{5.32}, \eqref{5.31}, Theorem \ref{t1.1} and Remark \ref{r1.1}(a),
implies that there exists a $\dz_0\in(0,\fz)$, depending only on $n$, $p$, $\boz$ and $[\omega]_{A_p(\rn)}$,
such that, if $A$ satisfies the $(\dz,R)$-$\mathrm{BMO}$ condition with some $\dz\in(0,\dz_0)$
and $R\in(0,\fz)$ or $A\in\mathrm{VMO}(\rn)$, then \eqref{5.30} holds true in this case.

\emph{Step 2.} Let $u_2$ be the weak solution of the Robin problem \eqref{5.28}.
In this step, we show that, for any given $p\in(1,\fz)$ and $\omega\in A_p(\rn)$,
there exists a $\dz_0\in(0,\fz)$, depending only on $n$, $p$, $[\omega]_{A_p(\rn)}$ and $\boz$,
such that, if $A$ satisfies the $(\dz,R)$-$\mathrm{BMO}$ condition for some $\dz\in(0,\dz_0)$
and $R\in(0,\fz)$ or $A\in\mathrm{VMO}(\rn)$, then
\begin{equation}\label{5.33}
\|u_2\|_{W^{1,p}_\omega(\boz)}\ls\|F_2\|_{L^q_{\omega^a}(\boz)},
\end{equation}
where $q:=np/(n+p-1)$ and $a:=(n-1)/(n+p-1)$.

Let $\omega_1:=\omega^{-p'/p}$. Then, by Lemma \ref{l3.1}(iii),
we know that $\omega_1\in A_{p'}(\rn)$. Assume that
$\mathbf{h}_1\in L^{p'}_{\omega_1}(\boz;\rn)$ and
$v_1$ is the weak solution of the Robin problem \eqref{5.23} with
$\mathbf{g}_1$ replaced by $\mathbf{h}_1$. Then
\begin{equation*}
\int_\boz\mathbf{h}_1\cdot\nabla u_2\,dx=-\int_{\boz}A\nabla u_2\cdot\nabla v_1\,dx
-\int_{\partial\boz}\az u_2v_1\,d\sz(x)=\int_\boz F_2v_1\,dx,
\end{equation*}
which, together with the H\"older inequality, Lemma \ref{l5.9} and the fact that
\eqref{5.30} holds true for any given $p\in(1,\fz)$ and $\omega\in A_p(\rn)$,
further implies that
\begin{align}\label{5.34}
\|\nabla u_2\|_{L^p_\omega(\boz;\rn)}&=\sup_{\|\mathbf{h_1}\|_{L^{p'}_{\omega_1}(\boz;\rn)}\le1}
\lf|\int_\boz\mathbf{h}_1\cdot\nabla u_2\,dx\r|
=\sup_{\|\mathbf{h}_1\|_{L^{p'}_{\omega_1}(\boz;\rn)}\le1}\lf|\int_\boz F_2v_1\,dx\r|\\
&\ls\sup_{\|\mathbf{h}_1\|_{L^{p'}_{\omega_1}(\boz;\rn)}\le1}\|F_2\|_{L^q_{\omega^a}(\boz)}
\|v_1\|_{L^{q'}_{\omega_1}(\boz)}
\ls\sup_{\|\mathbf{h}_1\|_{L^{p'}_{\omega_1}(\boz;\rn)}\le1}\|F_2\|_{L^q_{\omega^a}(\boz)}
\|v_1\|_{W^{1,p'}_{\omega_1}(\boz)}\nonumber\\
&\ls\sup_{\|\mathbf{h}_1\|_{L^{p'}_{\omega_1}(\boz;\rn)}\le1}\|F_2\|_{L^q_{\omega^a}(\boz)}
\|\mathbf{h}_1\|_{L^{p'}_{\omega_1}(\boz;\rn)}\ls
\|F_2\|_{L^q_{\omega^a}(\boz)}.\nonumber
\end{align}
Furthermore, from $\omega\in A_p(\rn)$, $q=np/(n+p-1)$,
$a=(n-1)/(n+p-1)$ and Lemma \ref{l3.1}(iv), it follows that
$\omega^a\in A_q(\rn)$. Then, by (i) and (vii) of Lemma \ref{l3.1}, we find that
there exists a $q_0\in(1,q)$ such that
\begin{align}\label{5.35}
\|F_2\|_{L^{\frac{q}{q_0}}(\boz)}\ls\|F_2\|_{L^q_{\omega^a}(\boz)}.
\end{align}
Furthermore, from Theorem \ref{t1.3}(i), we deduce that
there exists an $s_0\in(1,\fz)$ such that
$$\|u_2\|_{L^{s_0}(\boz)}\ls\|F_2\|_{L^{\frac{q}{q_0}}(\boz)},
$$
which, combined with Lemma \ref{l5.9}, \eqref{5.34}, \eqref{5.35}
and the H\"older inequality, further implies that
\begin{align*}
\|u_2\|_{L^p_\omega(\boz)}&\le\|u_2-\overline{u}_2\|_{L^p_\omega(\boz)}+
|\overline{u}_2|[\omega(\boz)]^{\frac1p}\nonumber\\
&\le\|u_2-\overline{u}_2\|_{L^{\frac{np}{n-1}}_\omega(\boz)}[\omega(\boz)]^{\frac1{np}}
+\|u_2\|_{L^{s_0}(\boz)}[\omega(\boz)]^{\frac1{p}}|\boz|^{-\frac{1}{s_0}}\nonumber\\
&\ls\|\nabla u_2\|_{{L^{p}_\omega(\boz;\rn)}}|\boz|^{\frac{1}{n}}
+\|F_2\|_{L^q_{\omega^a}(\boz)}\ls\|F_2\|_{L^q_{\omega^a}(\boz)},
\end{align*}
where $\overline{u}_2:=\fint_\boz u_2\,dx$.
By this and \eqref{5.34}, we conclude that \eqref{5.33} holds true.

Let $u$ be the weak solution of the Robin problem
\begin{equation*}
\begin{cases}
-\mathrm{div}(A\nabla u)=\mathrm{div}(\mathbf{f})+F\ \ & \text{in}\ \ \boz,\\
\dfrac{\partial u}{\partial\boldsymbol{\nu}}+\az u
=\mathbf{f}\cdot\boldsymbol{\nu}\ \ & \text{on}\ \ \partial\boz.
\end{cases}
\end{equation*}
Then $u=u_{\mathbf{f}}+u_{F}$, where $u_{\mathbf{f}}$ and $u_{F}$ are,
respectively, the weak solutions of the Robin problems \eqref{5.23} and \eqref{5.28}
with $\mathbf{g}_1$ and $F_2$ replaced, respectively, by
$\mathbf{f}$ and $F$. From this, \eqref{5.30} and \eqref{5.33}, we deduce that
\eqref{1.24} holds true for any given $p\in(1,\fz)$ and $\omega\in A_p(\rn)$.
This finishes the proof of (ii).

Finally, we show (iii).
Let $p\in(n/(n-1),\fz)$ and $\omega\in A_1(\rn)$.
Then, by Lemma \ref{l3.1}(ii), we find that there exists a $p_2\in(p,\fz)$
such that $\omega\in RH_{(\frac{p_2}{p})'}(\rn)$.
Furthermore, take $p_0\in(n/(n-1),p)$.
Then $\omega\in A_1(\rn)\subset A_{p/p_0}(\rn)$. Thus,
$\omega\in A_{p/p_0}(\rn)\cap RH_{(\frac{p_2}{p})'}(\rn)$.
From Theorem \ref{t1.3}(i), it follows that the Robin problems $(R)_{p_0}$ and $(R)_{p_2}$
are uniquely solvable, and \eqref{1.12} and \eqref{1.13} hold true
with $p$ replaced by $p_2$, which, together with Theorem \ref{t1.1},
further implies that there exists a positive constant $\dz_0\in(0,\fz)$,
depending only on $n$, $p$, $\boz$ and $[\omega]_{A_1(\rn)}$,  such that, if
$A$ satisfies the $(\dz,R)$-$\mathrm{BMO}$ condition for some $\dz\in(0,\dz_0)$
and $R\in(0,\fz)$ or $A\in\mathrm{VMO}(\rn)$, then the weighted Robin problem $(R)_{p,\,\omega}$ with
$\mathbf{f}\in L^p_\omega(\boz;\rn)$ and $F\in L^{p_\ast}_{\omega^{p_\ast/p}}(\boz)$ is uniquely solvable
and, for any weak solution $u$,
\begin{equation*}
\|u\|_{W^{1,p}_\omega(\boz)}\ls\|\mathbf{f}\|_{L^p_\omega(\boz;\rn)}
+\|F\|_{L^{p_\ast}_{\omega^{p_\ast/p}}(\boz)}.
\end{equation*}
This finishes the proof of (iii) and hence of Theorem \ref{t1.3}.
\end{proof}

\section{Proofs of Theorems \ref{t2.1}, \ref{t2.2}, \ref{t2.3}
and \ref{t2.4}}\label{s6}

\hskip\parindent In this section, we show Theorems \ref{t2.1}, \ref{t2.2}, \ref{t2.3} and \ref{t2.4}
by using Theorems \ref{t1.2} and \ref{t1.3}, properties of Muckenhoupt
weights and the extrapolation theorem recently established in \cite{ch18}.

To prove Theorems \ref{t2.1} and \ref{t2.2},
we need the following lemma, which is well known (see, for instance,
\cite[Section 7.1.2]{g14} and \cite[Lemma 3.4]{mp12}).

\begin{lemma}\label{l6.1}
\begin{itemize}
\item[{\rm(i)}] Let $s\in[1,\fz)$, $\omega\in A_s(\rn)$, $z\in\rn$ and $k\in(0,\fz)$
be a constant. Assume that $\tau^z(\omega)(\cdot):=\omega(\cdot-z)$ and
$\omega_k:=\min\{\omega,\,k\}$. Then $\tau^z(\omega)\in A_s(\rn)$ with
$[\tau^z(\omega)]_{A_s(\rn)}=[\omega]_{A_s(\rn)}$ and
$\omega_k\in A_s(\rn)$ with $[\omega_k]_{A_s(\rn)}\le c_{(s)}[\omega]_{A_s(\rn)}$,
where $c_{(s)}:=1$ when $s\in[1,2]$, and $c_{(s)}:=2^{s-1}$ when $s\in(2,\fz)$.
\item[{\rm(ii)}] For any $x\in\rn$, let $\omega_b(x):=|x|^b$,
where $b\in\rr$ is a constant. Then, for any given $s\in(1,\fz)$, $\omega_b\in A_s(\rn)$
if and only if $b\in(-n,n[s-1])$. Moreover, $[\omega_b]_{A_s(\rn)}\le C_{(n,\,s,\,b)}$,
where $C_{(n,\,s,\,b)}$ is a positive constant depending only on $n$, $s$ and $b$.
\item[{\rm(iii)}] Let $q\in(1,\fz]$, $\omega\in RH_q(\rn)$, $z\in\rn$ and $k\in(0,\fz)$
be a constant. Assume that $\tau^z(\omega)$ and $\omega_k$ are as in Lemma \ref{l6.1}(i).
Then $\tau^z(\omega)\in RH_q(\rn)$ with $[\tau^z(\omega)]_{RH_q(\rn)}=[\omega]_{RH_q(\rn)}$, and
$\omega_k\in RH_q(\rn)$ with $[\omega_k]_{RH_q(\rn)}\le c_{(q)}[\omega]_{RH_q(\rn)}$,
where $c_{(q)}$ is a positive constant depending only on $q$.
\item[{\rm(iv)}] For any $x\in\rn$, let $\omega_b(x):=|x|^b$,
where $b\in\rr$ is a constant. Assume that $q\in(1,\fz]$.
If $b\in(-n/q,\fz)$, then $\omega_b\in RH_q(\rn)$; furthermore,
$[\omega_b]_{RH_q(\rn)}\le C_{(n,\,q,\,b)}$,
where $C_{(n,\,q,\,b)}$ is a positive constant depending only on $n$, $q$ and $b$.
\end{itemize}
\end{lemma}

Now we show Theorem \ref{t2.1} by using Theorem \ref{t1.2}(ii) and Lemma \ref{l6.1}.

\begin{proof}[Proof of Theorem \ref{t2.1}]
Since the proofs of (i) and (ii) are similar,
we only show (ii) here. We prove (ii) via
borrowing some ideas from \cite{mp12,mp11}.

Let $u$ be the weak solution of the Robin problem $(R)_2$
with $\mathbf{f}\in \cm^{\tz}_p(\boz;\rn)$, $F\in \cm^{\wz{\tz}}_{p_\ast}(\boz)$
and $g\equiv0$, where $p\in(\max\{n/(n-1),(3+\uc_0)'\},3+\uc_0)$,
$\tz\in(pn/(3+\uc_0),n]$, $p_\ast$ is as in \eqref{1.5} and
$\wz{\tz}$ is given by $\wz{\tz}:=p_\ast(1+\frac{\tz}{p})$.
For any $\rho\in(0,\diam(\boz)]$, $\epsilon\in(0,\tz-n+n/(\frac{3+\uc_0}{p})')$
and $z,\,x\in\boz$, let
$$\omega_z(x):=\min\lf\{|x-z|^{-n+\tz-\epsilon},\,\rho^{-n+\tz-\epsilon}\r\}.
$$
Then, by (i) and (ii) of Lemma \ref{l6.1}, we conclude that, for any given $z\in\boz$,
$\omega_z\in A_s(\rn)$ with any given $s\in(1,\fz)$,
and there exists a positive constant $C_{(n,\,s,\,\tz)}$, depending only on
$n$, $s$ and $\tz$, such that $[\omega_z]_{A_s(\rn)}\le C_{(n,\,s,\,\tz)}$.
Furthermore, from the assumption $\tz>np/(3+\uc_0)$, it follows that
$\tz>n-n/(\frac{3+\uc_0}{p})'$, which, combined with $\epsilon\in(0,\tz-n+n/(\frac{3+\uc_0}{p})')$,
further implies that $\tz-n-\epsilon>-n/(\frac{3+\uc_0}{p})'$. By this, and (iii)
and (iv) of Lemma \ref{l6.1}, we find that, for any given $z\in\boz$,
$\omega_z\in RH_{(\frac{3+\uc_0}{p})'}(\rn)$ and
$[\omega_z]_{RH_{(\frac{3+\uc_0}{p})'}(\rn)}\ls1$. Thus, for any given $z\in\boz$,
$\omega_z\in A_{\frac{p}{\wz{p}_0}}(\rn)\cap RH_{(\frac{3+\uc_0}{p})'}(\rn)$,
$[\omega_z]_{A_{\frac{p}{\wz{p}_0}}(\rn)}\ls1$ and
$[\omega_z]_{RH_{(\frac{3+\uc_0}{p})'}(\rn)}\ls1$,
where $\wz{p}_0:=\max\{\frac n{n-1},(3+\uc_0)'\}$.
From this, \eqref{1.18} and the fact that, for any $x\in B(z,\rho)$,
$\omega_z(x)=\rho^{-n+\tz-\epsilon}$, we deduce that, for any $z\in\boz$
and $\rho\in(0,\diam(\boz)]$,
\begin{align}\label{6.1}
\||u|+|\nabla u|\|_{L^{p}(B(z,\rho)\cap\boz)}&=\rho^{\frac{n-\tz+\epsilon}{p}}
\||u|+|\nabla u|\|_{L^{p}_{\omega_z}(B(z,\rho)\cap\boz)}\\
&\ls\rho^{\frac{n-\tz+\epsilon}{p}}
\lf[\|\mathbf{f}\|_{L^{p}_{\omega_z}(\boz;\rn)}
+\|F\|_{L^{p_\ast}_{\omega_z^{p_\ast/p}}(\boz)}\r].\nonumber
\end{align}
Moreover, by the Fubini theorem and the fact that, for any given $z\in\boz$,
$\omega_z\le \rho^{-n+\tz-\epsilon}$, we know that
\begin{align}\label{6.2}
\int_\boz |\mathbf{f}|^p\omega_z\,dx&=\int_0^\fz\lf[\int_{\{x\in\boz:\ \omega_z(x)>t\}}
|\mathbf{f}|^p\,dx\r]\,dt\le\int_0^{\rho^{-n+\tz-\epsilon}}\int_{B(z,t^{-\frac{1}{n-\tz+\epsilon}})\cap\boz}
|\mathbf{f}|^p\,dx\,dt\\
&\le\|\mathbf{f}\|_{\cm^\tz_p(\boz)}^p
\int_0^{\rho^{-n+\tz-\epsilon}}t^{-\frac{n-\tz}{n-\tz+\epsilon}}\,dt
\ls\|\mathbf{f}\|_{\cm^\tz_p(\boz;\rn)}^p \rho^{-\epsilon}.\nonumber
\end{align}
Moreover, similarly to \eqref{6.2}, we have
$$
\int_\boz |F|^{p_\ast}\omega_z^{\frac{p_\ast}{p}}\,dx\ls
\|F\|_{\cm^{\wz{\tz}}_{p_\ast}(\boz)}^{p_\ast} \rho^{n-\wz{\tz}}
\rho^{(-n+\tz-\epsilon)\frac{p_\ast}{p}},
$$
which, together with \eqref{6.1}, \eqref{6.2} and $\wz{\tz}=p_\ast(1+\frac{\tz}{p})$,
further implies that, for any $z\in\boz$ and $\rho\in(0,\diam(\boz)]$,
\begin{align*}
\||u|+|\nabla u|\|_{L^{p}(B(z,\rho)\cap\boz)}&\ls\rho^{\frac{n-\tz+\epsilon}{p}}
\lf[\|\mathbf{f}\|_{\cm^\tz_p(\boz;\rn)} \rho^{-\frac{\epsilon}{p}}
+\|F\|_{\cm^{\wz{\tz}}_{p_\ast}(\boz)}\rho^{\frac{n-\wz{\tz}}{p_\ast}}
\rho^{(-n+\tz-\epsilon)\frac{1}{p}}\r]\nonumber\\
&\ls\rho^{\frac{n-\tz}{p}}
\lf[\|\mathbf{f}\|_{\cm^\tz_p(\boz;\rn)}
+\|F\|_{\cm^{\wz{\tz}}_{p_\ast}(\boz)}\r].
\end{align*}
From this and the definition of $\cm^\tz_p(\boz)$, we deduce that
$$\||u|+|\nabla u|\|_{\cm^{\tz}_p(\boz)}\ls
\|\mathbf{f}\|_{\cm^{\tz}_p(\boz;\rn)}+\|F\|_{\cm^{\wz{\tz}}_{p_\ast}(\boz)}.
$$
This finishes the proof of Theorem \ref{t2.1}.
\end{proof}

\begin{proof}[Proof of Theorem \ref{t2.2}]
The proof of Theorem \ref{t2.2} is similar to that of Theorem \ref{t2.1}
and we omit the details.
\end{proof}

To show Theorem \ref{t2.3} via using the Rubio de Francia extrapolation theorem
in the scale of Musielak--Orlicz spaces (or generalized Orlicz spaces),
we need the following Lemma \ref{l6.2}, which is just \cite[Corollary 4.21]{ch18}.

\begin{lemma}\label{l6.2}
Let $n\ge2$, $\boz\subset\rn$ be a bounded Lipschitz domain, $f,\,h$ be given non-negative
measurable functions on $\boz$  and $1<p_1<p<p_2<\fz$.
Assume that, for any $\omega\in A_{p/p_1}(\rn)\cap RH_{(p_2/p)'}(\rn)$,
\begin{equation*}
\|f\|_{L^p_\omega(\boz)}\le C\|g\|_{L^p_\omega(\boz)},
\end{equation*}
where $C$ is a positive constant depending only
on $n$, $p$, $\boz$, $[\omega]_{A_{p/p_1}(\rn)}$
and $[\omega]_{RH_{(p_2/p)'}(\rn)}$.
If $\fai\in\Phi_w(\boz)$ satisfies Assumptions $(A0)$--$(A2)$,
$\mathrm{(aInc)}_{q_1}$ and $\mathrm{(aDec)}_{q_2}$ for some $p_1<q_1\le q_2<p_2$,
then there exists a positive constant $C$,
depending only on $n$, $\boz$ and $\fai$, such that
$\|f\|_{L^{\fai}(\boz)}\le C\|h\|_{L^{\fai}(\boz)}$.
\end{lemma}

Now we prove Theorem \ref{t2.3} via using Theorem \ref{t1.2},
Remark \ref{r1.3} and Lemma \ref{l6.2}.

\begin{proof}[Proof of Theorem \ref{t2.3}]
Let $u$ be the weak solution of the Robin problem
\begin{equation*}
\begin{cases}
-\mathrm{div}(A\nabla u)=\mathrm{div}(\mathbf{f})\ \ &\text{in}\ \ \boz,\\
\dfrac{\partial u}{\partial\boldsymbol{\nu}}+\az u=\mathbf{f}\cdot\boldsymbol{\nu} \ \ &\text{on}\ \ \partial\boz.
\end{cases}
\end{equation*}
Assume that $p\in((3+\uc_0)',3+\uc_0)$,
where $\uc_0\in(0,\fz)$ is as in Theorem \ref{t1.2}.
Then, by Theorem \ref{t1.2} and Remark \ref{r1.3}, we conclude that,
for any given $\omega\in A_{\frac{p}{(3+\uc_0)'}}(\rn)\cap
RH_{(\frac{3+\uc_0}{p})'}(\rn)$, there exists a positive constant $\dz_0\in(0,\fz)$,
depending only on $n$, $p$, the Lipschitz constant of $\boz$, $[\omega]_{A_{\frac{p}{(3+\uc_0)'}}(\rn)}$
and $[\omega]_{RH_{(\frac{3+\uc_0}{p})'}(\rn)}$,  such that, if
$A$ satisfies the $(\dz,R)$-$\mathrm{BMO}$ condition for some $\dz\in(0,\dz_0)$
and $R\in(0,\fz)$ or $A\in\mathrm{VMO}(\rn)$, then
\begin{equation*}
\||u|+|\nabla u|\|_{L^p_\omega(\boz)}\ls\|\mathbf{f}\|_{L^p_\omega(\boz;\rn)}.
\end{equation*}
From this and Lemma \ref{l6.2} with $f:=|u|+|\nabla u|$, $h:=|\mathbf{f}|$,
$p_1:=(3+\uc_0)'$ and $p_2:=3+\uc_0$,
we deduce that
$$\|u\|_{L^{\fai}(\boz)}+\lf\|\nabla u\r\|_{L^{\fai}(\boz;\rn)}
\sim\|f\|_{L^{\fai}(\boz)}
\ls\|h\|_{L^{\fai}(\boz)}\sim\lf\|\mathbf{f}\r\|_{L^{\fai}(\boz;\rn)},
$$
which completes the proof of Theorem \ref{t2.3}.
\end{proof}

To prove Theorem \ref{t2.4} via using the extrapolation technique,
we need the following Lemma \ref{l6.3}, which is just \cite[Corollary 4.8]{ch18}.

\begin{lemma}\label{l6.3}
Let $n\ge2$, $\boz\subset\rn$ be a bounded Lipschitz domain, $f,\,h$ be given non-negative
measurable functions on $\boz$  and $1\le p_0\le q_0<\fz$.
Assume that, for any $\omega\in A_1(\rn)$,
\begin{equation*}
\|f\|_{L^{q_0}_\omega(\boz)}\le C\|g\|_{L^{p_0}_{\omega^{p_0/q_0}}(\boz)},
\end{equation*}
where $C$ is a positive constant depending only on $n$, $p_0$, $q_0$,
$\boz$ and $[\omega]_{A_1(\rn)}$.
Let $\fai\in\Phi_w(\boz)$ satisfies Assumptions $(A0)$--$(A2)$,
$\mathrm{(aInc)}_{q_0}$ and $\mathrm{(aDec)}_{q_+}$ for some $q_+\in(q_0,\fz)$,
and the function $\psi$ be given by, for any $(x,t)\in\boz\times[0,\fz)$,
$\psi^{-1}(x,t)=t^{\frac{1}{p_0}-\frac{1}{q_0}}\fai^{-1}(x,t)$.
Then there exists a positive constant $C$,
depending only on $n$, $\boz$ and $\fai$, such that
$\|f\|_{L^{\fai}(\boz)}\le C\|h\|_{L^{\psi}(\boz)}$.
\end{lemma}

We show Theorem \ref{t2.4} by using Theorem \ref{t1.3}(iii)
and Lemma \ref{l6.3}.

\begin{proof}[Proof of Theorem \ref{t2.4}]
Let $u_\mathbf{f}$ and $u_F$ be, respectively, the weak solutions of the Robin problems
\begin{equation*}
\begin{cases}
-\mathrm{div}(A\nabla u_\mathbf{f})=\mathrm{div}(\mathbf{f})\ \ &\text{in}\ \ \boz,\\
\dfrac{\partial u_\mathbf{f}}{\partial\boldsymbol{\nu}}+\az u_\mathbf{f}=\mathbf{f}\cdot\boldsymbol{\nu}
\ \ &\text{on}\ \ \partial\boz
\end{cases}
\end{equation*}
and
\begin{equation*}
\begin{cases}
-\mathrm{div}(A\nabla u_F)=F\ \ & \text{in}\ \ \boz,\\
\dfrac{\partial u_F}{\partial\boldsymbol{\nu}}+\az u_F=0\ \ & \text{on}\ \ \partial\boz.
\end{cases}
\end{equation*}
Then $u=u_\mathbf{f}+u_F$. Assume that $p\in(n/(n-1),\fz)$.
By Theorem \ref{t1.3}(iii), we find that,
for any given $\omega\in A_1(\rn)$, there exists a positive constant $\dz_0\in(0,\fz)$,
depending only on $n$, $p$, $\boz$ and $[\omega]_{A_1(\rn)}$, such that, if
$A$ satisfies the $(\dz,R)$-$\mathrm{BMO}$ condition for some $\dz\in(0,\dz_0)$
and $R\in(0,\fz)$ or $A\in\mathrm{VMO}(\rn)$, then
\begin{equation*}
\||u_\mathbf{f}|+|\nabla u_\mathbf{f}|\|_{L^p_\omega(\boz)}\ls\|\mathbf{f}\|_{L^p_\omega(\boz;\rn)}.
\end{equation*}
From this and Lemma \ref{l6.3} with $f:=|u_\mathbf{f}|+|\nabla u_\mathbf{f}|$,
$h:=|\mathbf{f}|$ and $p_0=q_0:=p$, it follows that
\begin{equation}\label{6.3}
\|u_\mathbf{f}\|_{L^{\fai}(\boz)}+\lf\|\nabla u_\mathbf{f}\r\|_{L^{\fai}(\boz;\rn)}
\sim\|f\|_{L^{\fai}(\boz)}
\ls\|h\|_{L^{\fai}(\boz)}\sim\lf\|\mathbf{f}\r\|_{L^{\fai}(\boz;\rn)}.
\end{equation}
Moreover, by Theorem \ref{t1.3}(iii), we conclude that,
for any given $\omega\in A_1(\rn)$,
\begin{equation*}
\||u_F|+|\nabla u_F|\|_{L^p_\omega(\boz)}\ls\|F\|_{L^{p_\ast}_{\omega^{p_\ast/p}}(\boz)},
\end{equation*}
which, combined with Lemma \ref{l6.3} in the case that $f:=|u_F|+|\nabla u_F|$,
$h:=|F|$, $p_0:=p_\ast$ and $q_0:=p$ and the fact that $\frac1{p_\ast}-\frac1p=\frac1n$,
further implies that
\begin{equation*}
\|u_F\|_{L^{\fai}(\boz)}+\lf\|\nabla u_F\r\|_{L^{\fai}(\boz;\rn)}
\sim\|f\|_{L^{\fai}(\boz)}
\ls\|h\|_{L^{\psi}(\boz)}\sim\lf\|F\r\|_{L^{\psi}(\boz)}.
\end{equation*}
From this estimate, \eqref{6.3} and $u=u_\mathbf{f}+u_F$, we deduce that
\begin{equation*}
\|u\|_{L^{\fai}(\boz)}+\lf\|\nabla u\r\|_{L^{\fai}(\boz;\rn)}
\ls\lf\|\mathbf{f}\r\|_{L^{\fai}(\boz;\rn)}+\lf\|F\r\|_{L^{\psi}(\boz)},
\end{equation*}
which completes the proof of Theorem \ref{t2.4}.
\end{proof}

\noindent{\textbf{Acknowledgements}}\quad
Sibei Yang would also like to thank Professor Jun Geng
for some helpful discussions on the topic of this article.

\bigskip

\noindent Sibei Yang

\medskip

\noindent School of Mathematics and Statistics, Gansu Key Laboratory of Applied Mathematics
and Complex Systems, Lanzhou University, Lanzhou 730000, People's Republic of China

\smallskip

\noindent{\it E-mail:} \texttt{yangsb@lzu.edu.cn}

\bigskip

\noindent Dachun Yang (Corresponding author) and Wen Yuan

\medskip

\noindent Laboratory of Mathematics and Complex Systems (Ministry of Education of China),
School of Mathematical Sciences, Beijing Normal University, Beijing 100875,
People's Republic of China

\smallskip

\smallskip

\noindent {\it E-mail}: \texttt{dcyang@bnu.edu.cn}

\end{document}